\documentclass{amsart}  
\usepackage{amssymb,amstext,amsmath,amscd,amsthm,amsfonts,enumerate,graphicx,latexsym,stmaryrd,multicol}   
\usepackage{scalerel}
\newcommand\reallywidehat[1]{\arraycolsep=0pt\relax%
\begin{array}{c}
\stretchto{
  \scaleto{
    \scalerel*[\widthof{\ensuremath{#1}}]{\kern-.5pt\bigwedge\kern-.5pt}
    {\rule[-\textheight/2]{1ex}{\textheight}} %WIDTH-LIMITED BIG WEDGE
  }{\textheight} % 
}{0.5ex}\\           % THIS SQUEEZES THE WEDGE TO 0.5ex HEIGHT
#1\\                 % THIS STACKS THE WEDGE ATOP THE ARGUMENT
\rule{-1ex}{0ex}
\end{array}
}
\usepackage{tikz-cd}
\usepackage{color}
\usepackage{hyperref} 
\usepackage{comment} 

\usepackage{amsmath, mathtools}
\hypersetup{colorlinks=true} 

\usepackage[all]{xy}
\theoremstyle{plain}

\newtheorem{theorem}{Theorem}[section]
\newtheorem{prop}[theorem]{Proposition}
\newtheorem{lem}[theorem]{Lemma}
\newtheorem{dfn}[theorem]{Definition} 
\newtheorem{cor}[theorem]{Corollary}

\newtheorem{subtheorem}{Theorem}[subsection]
\newtheorem{subproposition}[subtheorem]{Proposition}
\newtheorem{subthm}[subtheorem]{Theorem}  
\newtheorem{sublemma}[subtheorem]{Lemma}
\newtheorem{subcor}[subtheorem]{Corollary}  

{\Alph{theorem}}

\theoremstyle{definition}

\newtheorem{subdfn}[subtheorem]{Definition}

\newtheorem{rem}[theorem]{Remark}

\theoremstyle{remark}

\newtheorem{chunk}[theorem]{}

\newtheorem{subchunk}[subtheorem]{}

\numberwithin{equation}{subtheorem}

\theoremstyle{remark}

\def\depth{\operatorname{depth}}
\def \S {\mathcal S}  
 
\def\Tr{\operatorname{Tr}}
\def\tr{\operatorname{tr}}
\def\Ext{\operatorname{Ext}}
\def\Tor{\operatorname{Tor}}
\def\injdim{\operatorname{inj dim}}

\def\tor{\operatorname{Tor}} 
 
\def\add{\operatorname{add}} 
\def \H {\operatorname{H}}
\def\m{\mathfrak{m}}
\def\r{\operatorname{r}}
\def\ann{\operatorname{ann}} 

\def\mod{\operatorname{mod}}
\def\Soc{\operatorname{Soc}} 

\def\X{\mathcal{X}}
 
\def\A{\mathcal{A}}
\def\B{\mathcal{B}}
\def\E{\mathcal{E}}
 
\def\supp{\operatorname{Supp}}
\def \pd {\operatorname{pd}}  
\def\p{\mathfrak{p}} 
\def\Hom{\operatorname{Hom}} 
\def\Mor{\operatorname{Mor}} 
\def \Ul {\operatorname{Ul}}

\def \cm {\operatorname{CM}} 
\def\syz{\Omega} 
\def\spec{\operatorname{Spec}}

\def \Ass {\operatorname{Ass}} 
\def \Min {\operatorname{Min}}
\def \Mod {\operatorname{Mod}} 
 
\def \rnk {\operatorname{rank}} 
\def \C {\mathcal{C}}

\def \MD {\operatorname{MD}}

\begin{document}

%\allowdisplaybreaks
\title{Exact Subcategories, subfunctors of $\Ext$, and some applications}    
\author{Hailong Dao}   
\address{Department of Mathematics, University of Kansas, Lawrence, KS 66045-7523, USA} 
\email{hdao@ku.edu}

\author{Souvik Dey}
\address{Department of Algebra, Charles University, Faculty of Mathematics and Physics, Sokolovska´ 83, 186 75, Praha, Czech Republic}
\email{dey0976@gmail.com, souvik@ku.edu}

\author{Monalisa Dutta}   
\address{Department of Mathematics, University of Kansas, Lawrence, KS 66045-7523, USA} 
\email{m819d903@ku.edu}
 
 
\begin{abstract} 
Let $(\A,\E)$ be an exact category. We establish basic results that allow one to identify sub(bi)functors of $\Ext_{\E}(-,-)$ using additivity of numerical functions and restriction to subcategories. We also study a small number of these new functors over commutative local rings in detail and find a range of applications from detecting regularity to understanding Ulrich modules.    
\end{abstract}   
\maketitle  

\section[Introduction]{Introduction}

The Yoneda characterization of $\Ext$ is familiar to most students of homological algebra. Let $A,B$ be two objects in an abelian category $\A$. Then $\Ext_{\A}(A,B)$ is the set of all equivalence classes of sequences of the form $0\to B\to C\to A\to 0$, where two sequences $\alpha,\beta$ are equivalent if we have the following commutative diagram:
$$\begin{tikzcd}
\alpha: & 0\arrow[r] & B \arrow[r] \arrow[d, equal] & C\arrow[d, "f"]\arrow[r] & A \arrow[d,equal]\arrow[r] & 0\\
\beta: & 0\arrow[r] & B\arrow[r] & C'  \arrow[r]  & A  \arrow[r] & 0         
\end{tikzcd}$$

Now $\Ext_{\A}(A,B)$ can be given an abelian group structure by the well-known Baer sum as described in, for instance, \cite[Tag 010I]{st}. This consideration can be carried out more generally in any exact category (see \cite[Subsection 1.2]{vec}).  

The purpose of this note is to study the following rather natural questions: what if we place additional constraints on the short exact sequences? When do we get a subfunctor of $\Ext^1$? Can one apply such functors to study ring and module theory, similar to the ways homological algebra has been very successfully applied in the last decades?  

Let us elucidate our goals with a concrete example. Let $(R,\m,k)$ be a Noetherian local ring. Let $A,B$ be finitely generated $R$-modules. We consider exact sequences $0\to B\to C\to A\to 0$ of $R$-modules, with the added condition that $\mu(C)=\mu(A)+\mu(B)$, where $\mu(-)$ denotes the minimal number of generators. As we shall see later, the equivalence classes of such sequences do form a subfunctor of $\Ext^1_R$, which we denote by $\Ext^1_R(-,-)^{\mu}$. More surprisingly, the vanishing of a single module  $\Ext^1_R(-,-)^{\mu}$  can be used to characterize the regularity of $R$, a feature that is lacking with the classical $\Ext^1$.  

Analogous versions of classical homological functors have been studied by various authors, notably starting with Hochschild (\cite{Hochschild}) who studied relative $\Ext$ and $\Tor$ groups of modules over a ring with respect to subrings of the original ring. This work is further developed by Butler-Horrocks as well as Auslander-Solberg, where the point of view is switched to allowable exact sequences that give rise to sub(bi)functors of $\Ext^1$.  The whole circle of ideas is now thriving on its own under the name ``relative homological algebra", with exact structures playing a fundamental role, see \cite{reduc,Buan,Crivei12,EnochsJenda_v2, vec, Enomoto1, Enomoto2, matsui, Rump11,ZZ} for an incomplete list of literature and \cite{Solberg, Theo} for some excellent surveys. In commutative algebra, as far as we know, this line of inquiry has not been exploited thoroughly, however traces of it can be found in \cite{tony, thesis} and \cite[Section 1,2]{ES}. 

Although the existing literature provides excellent starting points and inspiring ideas for this present work, it is not always easy to extract the precise results needed for our intended applications. For instance, while the connections between subfunctors of $\Ext^1$ and certain sub-exact structures on a fixed category are well known (\cite[Section 1.2]{vec}), checking the conditions of substructures in each case can be time-consuming.  

We are able to find criteria that can be applied in broad settings to identify exact subcategories and subfunctors. Here is a sample result applicable to our motivating example above, which follows from Theorem \ref{subadd} and Proposition \ref{exsub}. 

\begin{theorem}
Let $({\A},{\E})$ be an exact category. Let $\phi:{\A} \to \mathbb Z$ be a function such that $\phi$ is constant on isomorphism classes of objects in $\A$, $\phi$ is additive on finite biproducts, and $\phi$ is sub-additive on kernel-cokernel pairs in $\E$ (i.e., if $\begin{tikzcd}
M \arrow[r, tail] & N \arrow[r, two heads] & L
\end{tikzcd}$ is in $\E$, then $\phi(N)\le \phi(M)+\phi(L)$). Set ${\E}^{\phi}:=\{\text{kernel-cokernel pairs in } {\E} \text{ on which } \phi \text { is additive} \}$. Then ${\E}^{\phi}$ gives rise, via the Yoneda construction, to a subfunctor of $\Ext_{\E}(-,-)$. 
\end{theorem}

Our above Theorem is partly motivated by, and can be used to recover and extend recent interesting work of Puthenpurakal in \cite{tony}, see Theorem \ref{thmTony}.  

Another situation we would like to have convenient criteria for subfunctors is when one ``restricts" to a certain subcategory. A concrete example we have in mind concerns Ulrich modules and their recent generalizations. These modules form a subcategory of  Cohen--Macaulay modules over a commutative ring and have been receiving increasing attention over the years due to very interesting and useful algebraic and geometric properties that their existence or abundance imply. We are able to show that even the generalized notion of ``$I$-Ulrich modules", recently introduced in \cite{dms}, induces subfunctors of $\Ext^1_R(-,-)$. See Proposition \ref{funcspec} and Corollary \ref{iulrich} in this regard.  

While this work is mainly concerned with foundational results, we also study the properties of some chosen new subfunctors, just to see if they are worth our efforts to show their existence! The early returns seem promising: these functors can be used to detect a wide range of ring and module-theoretic properties.  Below we shall describe the organization of the paper and describe the most interesting findings in more detail. 

Section \ref{sec2} is devoted to preliminary results on subfunctors of additive functors. While these results are perhaps not new, we were unable to locate convenient references, hence their inclusion. 

Section \ref{sec3} establishes various foundational results on exact subcategories, which form the cornerstone of the theory. As mentioned above, our applications require some extra care in preparation, and we try to give complete proofs whenever possible.  

Section \ref{sec4} concerns our first main application. We study sub-additive numerical functions $\phi$ on an exact category and show that under mild conditions they induce exact subcategories and hence subfunctors of $\Ext^1(-,-)$, which we denote by $\Ext^1(-,-)^{\phi}$. See Theorem \ref{subadd}.  A key consequence, Theorem \ref{thmTony} is inspired by, as well as extends, \cite[Theorem 3.11, Theorem 3.13]{tony}. We also give similar results about subfunctors of $\Ext^1$ induced by half-exact functors, in the spirit of \cite{aus}, see Corollaries \ref{half} and \ref{half2}.

In Section \ref{sec5} we focus on two special types of subfunctors, which arise from simple applications of previous sections. Already these cases appear to be interesting and useful. The first one is $\Ext^1_R(-,-)^{\mu}$, where $\mu$ is the minimal number of generators function mentioned above. We compute the values of this subfunctor on all (pair of) finitely generated modules over a DVR (Corollary \ref{dvr}), as well as for certain pairs of modules over a Cohen--Macaulay ring of minimal multiplicity (Proposition \ref{mr}).

Using this subfunctor, we prove the following characterization of the regularity of local rings, which is the combination of Theorem \ref{regdchar}, Theorem \ref{reg}, Proposition \ref{dvr} and Corollary \ref{ulfaith}. Note that the regularity can be detected by the vanishing of a single $\Ext^1_R(-,-)^{\mu}$-module. In the following statement, we mention that for a finitely generated module $M$ over the local ring $R$, $\syz^i_R M$ denotes the $i$-th syzygy in a minimal free resolution of $M$.

 \begin{theorem}\label{mainintr} Let $(R,\m,k)$ be a local ring of depth $t>0$. Then, the following are equivalent:  

\begin{enumerate}[\rm(1)]
    \item $R$ is regular. 
    \item $\Ext^1_R(k,R/(x_1,...,x_{t-1})R)^{\mu}=0$ for some $R$-regular sequence $x_1,...,x_{t-1}$.
    \item $\Ext^1_R(k,M)^{\mu}=0$ for some finitely generated $R$-module $M$ of projective dimension $t-1$. 
    \item $R$ is Cohen--Macaulay, and $\Ext^1_R(\syz^{t-1}_R k,N)^{\mu}=0$ for some finitely generated non-zero $R$-module $N$ of finite injective dimension. 
    
\vspace{2mm}

    Moreover, if $t=1$, then the above are also equivalent to each of the following.

\vspace{2mm}

    \item $\Ext^1_R(M,N)^{\mu}=\m\Ext^1_R(M,N)$ for all finitely generated $R$-modules $M$ and $N$. 
   
    \item $R$ is Cohen--Macaulay and there exist Ulrich modules $M,N$ such that $N$ is faithful, $M\ne 0$ and for every $R$-module $X$ that fits into a short exact sequence $0\to N\to X\to M \to 0 $ one has $X$ is also an Ulrich module.  
\end{enumerate}  

\end{theorem}  

We note here that the usual $\Ext$-modules (without the $\mu$) in (2), (3) and (4) of the above theorem are always non-zero. Moreover, the statement of part (6) of Theorem \ref{mainintr} apparently has nothing to do with subfunctor of $\Ext^1$, but we do not know a proof of $(1)\iff(6)$ (which is contained in Corollary \ref{ulfaith}) without resorting to $\Ext^1_{\Ul(R)}(-,-): \Ul(R)^{op}\times \Ul(R)\to \mod R$.  

It is worth mentioning that one can also use $\Ext^1_R(-,-)^{\mu}$ to detect the property of $R$ being a hypersurface of minimal multiplicity (Corollary \ref{hyper}) or the weak $\m$-fullness of a submodule (Proposition \ref{3.9}). 

The second type arises from $I$-Ulrich modules, where $I$ is any $\m$-primary ideal in a Noetherian local ring  $(R,\m)$. See \ref{uls} and \ref{iulrich} for the definition of $I$-Ulrich modules and the fact that they form an exact category. In Theorem \ref{projgor}, using a subfunctor of $\Ext^1$ corresponding to Ulrich modules over $1$-dimensional local Cohen--Macaulay rings, we give some characterizations of modules belonging to $\add_R(B(\m))$ and also characterize when $B(\m)$ is a Gorenstein ring in terms of annihilator of $\Ext^1_R$ of Ulrich modules. Here, $B(-)$ denotes the blow-up. We give some applications of Theorem \ref{projgor}, one of which relates annihilation of $\Ext^1_R(\Ul(R),\m)$ with that of $\Ext^1_R\left(\Ul_{\omega}(R),B(\omega)\right)$ (see Corollary \ref{5.2.16}). %We also present a characterization of almost Gorenstein rings of minimal multiplicity using the annihilator of $\Ext^1$ against Ulrich modules (see Corollary \ref{algor}).        

Finally, we should mention that one of the main applications of our results has appeared in a separate work, where we study the splitting of short exact sequences of Ulrich modules and connections to other properties of singularities (\cite{ulsplit}).

\section[Preliminaries on subfunctors of additive functors]{Preliminaries on subfunctors of additive functors}\label{sec2}

Unless otherwise stated, all rings in this paper are assumed to be commutative, Noetherian and with unity. For a ring $R$, $Q(R)$ will denote its total ring of fractions. For an $R$-module $M$, $\lambda_R(M)$ will denote its length. For a finitely generated $R$-module $M$, and $i\geq 1$, by $\syz^i_R M$ we mean $\text{Im} f_{i}$, where we have an exact sequence $F_{i}\xrightarrow{f_{i}} F_{i-1}\xrightarrow{f_{i-1}}\cdots \to F_0\to M \to 0$, with each $F_j$ being a finitely generated projective $R$-module. When $R$ is moreover local, we choose this so that $\text{Im}(f_j)\subseteq \m F_{j-1}$ for each $j$ (\cite[Proposition 1.3.1]{bh}). 

For definitions, and basic properties of additive categories, additive functors, $R$-linear categories, and $R$-linear functors,  we refer the reader to \cite[Tag 09SE, Tag 010M, Tag 09MI]{st}. 
We now recall the definition of subfunctors as in \cite{nlab}. 

\begin{dfn}\label{subfucntor} Let ${\A},{\B}$ be two categories. Let $F: {\A} \to {\B}$ be a covariant (resp. contravariant) functor. A covariant (resp. contravariant) functor $G: {\A} \to {\B}$ is called a subfunctor of $F$ if for every $M\in {\A}$, there exists a monomorphism $j_M:G(M)\to F(M)$, and moreover, for every $M,N\in {\A}$, and $f\in \Mor_{\A}(M,N)$, the following diagrams are commutative, where the left one stands for the covariant case and the right one for the contravariant case:    
\\
$$\begin{tikzcd}
F(M) \arrow[r, "F(f)"]                  & F(N)                  & F(N) \arrow[r, "F(f)"]                  & F(M)                  \\
G(M) \arrow[r, "G(f)"] \arrow[u, "j_M"] & G(N) \arrow[u, "j_N"] & G(N) \arrow[r, "G(f)"] \arrow[u, "j_N"] & G(M) \arrow[u, "j_M"]
\end{tikzcd}$$
\end{dfn}

\begin{chunk} For our purposes, we will always take $\B$ to be either the category of abelian groups $\mathbf{Ab}$, $\Mod R$ or $\mod R$ (hence monomorphisms are just injective morphisms) for some commutative ring $R$, and $j_M$ will usually be just the inclusion map. 
\end{chunk}  

The following Lemma is probably well-known, but we could not find an appropriate reference, hence we include a proof. This will be used throughout the remainder of the article, possibly without further reference.      

\begin{lem}\label{subad} \begin{enumerate}[\rm(1)]
    \item Subfunctor of an additive functor, between two additive categories, is additive. 
    
    \item  Subfunctor of an $R$-linear functor, between two $R$-linear categories, is $R$-linear.
\end{enumerate}   
\end{lem}  

\begin{proof} We will only prove the covariant case of both, since the contravariant case is similar.  

(1)  Let ${\A},{\B}$ be two additive categories and let $F:{\A}\to{\B}$ be an additive functor. Also let $G: {\A} \to {\B}$ be a subfunctor of $F$. Then we have to show that the map $G:\operatorname{Mor}_{\A}(X,Y)\to \operatorname{Mor}_{\B}(G(X),G(Y))$ is a homomorphism of abelian groups for all $X,Y\in {\A}$. Fix two objects $X,Y\in{\A}$. Then we need to prove that $G(f+g)=G(f)+G(g)$ for all $f,g\in\operatorname{Mor}_{\A}(X,Y)$. Since $G$ is a subfunctor of $F$, we have the following commutative diagrams:
$$\begin{tikzcd}
F(X) \arrow[r, "F(f+g)"]                  & F(Y)\\
G(X) \arrow[r, "G(f+g)"] \arrow[u, "j_X"] & G(Y) \arrow[u, "j_Y"]
\end{tikzcd}
\begin{tikzcd}
F(X) \arrow[r, "F(f)"]                  & F(Y)\\
G(X) \arrow[r, "G(f)"] \arrow[u, "j_X"] & G(Y) \arrow[u, "j_Y"]
\end{tikzcd}
\begin{tikzcd}
F(X) \arrow[r, "F(g)"]                  & F(Y)\\
G(X) \arrow[r, "G(g)"] \arrow[u, "j_X"] & G(Y) \arrow[u, "j_Y"]
\end{tikzcd}$$
Then we get 
\begin{align*}
j_Y\circ G(f+g) &=F(f+g)\circ j_X\\
&= (F(f)+F(g))\circ j_X\text{ [Since }F\text{ is additive]}\\
&=F(f)\circ j_X+F(g)\circ j_X\\
&=j_Y\circ G(f)+j_Y\circ G(g)\\
&=j_Y\circ (G(f)+G(g))
\end{align*}
where the third and the fifth equalities follow from the fact that $\B$ is an additive category.
Now since $j_Y$ is a monomorphism, so $j_Y\circ G(f+g)=j_Y\circ (G(f)+G(g))$ implies that $G(f+g)=G(f)+G(g)$. Hence $G$ is additive.  

(2)  Let ${\A},{\B}$ be two $R$-linear categories and let $F:{\A}\to{\B}$ be an $R$-linear functor. Also let $G: {\A} \to {\B}$ be a subfunctor of $F$. Then we have to show that the map $G:\operatorname{Mor}_{\A}(X,Y)\to \operatorname{Mor}_{\B}(G(X),G(Y))$ is an $R$-linear map for all $X,Y\in {\A}$. Fix two objects $X, Y\in{\A}$. Then by part (1) we already have the additivity of $G$, so we only need to prove that $G(rf)=rG(f)$ for all $f\in\operatorname{Mor}_{\A}(X,Y)$ and for all $r\in R$. Since $G$ is a subfunctor of $F$, we have the following commutative diagrams:
$$\begin{tikzcd}
F(X) \arrow[r, "F(rf)"]                  & F(Y)\\
G(X) \arrow[r, "G(rf)"] \arrow[u, "j_X"] & G(Y) \arrow[u, "j_Y"]
\end{tikzcd}
\begin{tikzcd}
F(X) \arrow[r, "F(f)"]                  & F(Y)\\
G(X) \arrow[r, "G(f)"] \arrow[u, "j_X"] & G(Y) \arrow[u, "j_Y"]
\end{tikzcd}$$
Then we get 
\begin{align*}
j_Y\circ G(rf) &=F(rf)\circ j_X\\
&= (rF(f))\circ j_X\text{ [Since }F\text{ is }R\text{-linear]}\\
&=r(F(f)\circ j_X)\\
&=r(j_Y\circ G(f))\\
&=j_Y\circ (rG(f))
\end{align*}
where the third and the fifth equalities follow from the fact that $\B$ is an $R$-linear category.
Now since $j_Y$ is a monomorphism, so $j_Y\circ G(rf)=j_Y\circ (rG(f))$ implies that $G(rf)=rG(f)$. Hence $G$ is $R$-linear.
\end{proof}  

We finish this section with a submodule inclusion result relating subfunctors of $R$-linear functors which will be applied in Section \ref{sec5}. 

\begin{lem}\label{extended} Let $\A$ be an additive $R$-linear category, and $G:{\A}\to \Mod R$ be a subfunctor of an $R$-linear functor $F:{\A} \to \Mod R$, and for every object $A\in {\A}$, let $j_A:G(A)\to F(A)$ be the monomorphism as in the definition of a subfunctor. Let $I$ be an ideal of $R$. Let $\{A_i\}_{i=1}^n$ be objects in $\A$ such that $j_{A_i}\left(G(A_i)\right)\subseteq IF(A_i)$ (resp. $j_{A_i}\left(G(A_i)\right)\supseteq IF(A_i)$) for all $i=1,...,n$. Let $X:=\oplus_{i=1}^nA_i$. Then, it holds that $j_{X} \left(G(X)\right)\subseteq IF(X)$ (resp. $j_{X} \left(G(X)\right)\supseteq IF(X)$). 
\end{lem}    

\begin{proof} We will only do the covariant case, the contravariant case being similar. For every $i$, we have $\pi_i:X \to A_i$, which gives rise to the following commutative diagram  

$$\begin{tikzcd}
G(X) \arrow[r, "G(\pi_i)"] \arrow[d, "j_X"] & G(A_i) \arrow[d, "j_{A_i}"] \\
F(X) \arrow[r, "F(\pi_i)"']                 & F(A_i)                     
\end{tikzcd}$$

Consequently, we get the following commutative diagram 
$$\begin{tikzcd}
G(X) \arrow[r, "(G(\pi_i))_{i=1}^n"] \arrow[d, "j_X"] & \oplus_{i=1}^nG(A_i) \arrow[d, "\oplus_{i=1}^nj_{A_i}"] \\
F(X) \arrow[r, "(F(\pi_i))_{i=1}^n"']                 & \oplus_{i=1}^nF(A_i)                           
\end{tikzcd}$$ 

where the horizontal arrows are isomorphisms, since $F$ and consequently $G$ are additive functors (Lemma \ref{subad}). Call the top horizontal map $\theta$, and the bottom one $\alpha$, so that $\alpha\circ j_X=(\oplus_{i=1}^nj_{A_i})\circ \theta$, hence $j_X\circ \theta^{-1}=\alpha^{-1}\circ (\oplus_{i=1}^nj_{A_i})$. So, now we get,
\begin{align*}
j_X(G(X))=(j_X \circ \theta^{-1}) (\theta(G(X)))=(j_X\circ \theta^{-1})(\oplus_{i=1}^nG(A_i)) &=(\alpha^{-1}\circ (\oplus_{i=1}^nj_{A_i}))(\oplus_{i=1}^nG(A_i))\\
&=\alpha^{-1}(\oplus_{i=1}^n j_{A_i}(G(A_i)))
\end{align*}
So, if $j_{A_i}\left(G(A_i)\right)\subseteq IF(A_i)$ (resp. $j_{A_i}\left(G(A_i)\right)\supseteq IF(A_i)$) for all $i=1,...,n$, then  
$$j_X(G(X))\subseteq \alpha^{-1}(\oplus_{i=1}^n IF(A_i))=\alpha^{-1}(I\left(\oplus_{i=1}^n F(A_i)\right))=I\alpha^{-1}(\oplus_{i=1}^n F(A_i))=IF(X)$$

(resp. $j_X(G(X))\supseteq \alpha^{-1}(\oplus_{i=1}^n IF(A_i))=\alpha^{-1}(I\left(\oplus_{i=1}^n F(A_i)\right))=I\alpha^{-1}(\oplus_{i=1}^n F(A_i))=IF(X)$), 

where we have used $\alpha^{-1}(IM)=I\alpha^{-1}(M)$, since $\alpha$ is  an $R$-linear map. 
\end{proof} 

\section[Some generalities about exact subcategories]{Some generalities about exact subcategories}\label{sec3}   

In this section, we record some generalities about exact subcategories of an exact category that we will later use for subcategories of $\mod R$ when $R$ is a commutative Noetherian ring. All our sub-categories are strict (closed under isomorphism classes) and full, and we often abbreviate this as strictly-full. We will follow the definition of an exact category described in \cite[Definition 2.1]{Theo}. We try to provide complete proofs whenever possible.  

Given an exact category $({\A},{\E})$, we call a monomorphism $X\xrightarrow{i} Y$ to be an $\E$-inflation (also, called an admissible monic) if it is the part of a kernel-cokernel pair $X\xrightarrow{i} Y\xrightarrow{} Z$, which lies in $\E$. Dually, we call an epimorphism $Y\xrightarrow{p} Z$ to be an $\E$-deflation (also, called an admissible epic) if it is the part of a kernel-cokernel pair $X\xrightarrow{} Y\xrightarrow{p} Z$, which lies in $\E$. We will often denote an admissible monic (resp. an admissible epic) by $\begin{tikzcd}
{} \arrow[r, tail] & {}
\end{tikzcd}$ (resp. $\begin{tikzcd}
{} \arrow[r, two heads] & {}
\end{tikzcd}$).

We begin by stating a lemma on morphisms and kernel-cokernel pairs, which we will use frequently while proving that certain structures are closed under isomorphism classes of kernel-cokernel pairs. This should be standard and well-known, but we could not find an appropriate reference, hence we include a proof. 

\begin{lem}\label{kerco} Let $\A$ be  an additive category. Let $M,N,L\in{\A}$ be such that $M\xrightarrow{i} N\xrightarrow{d} L$ is a kernel-cokernel pair in $\A$. Also, let $M',N',L'\in{\A}$. Consider morphisms $M'\xrightarrow{i'} N'$, $N'\xrightarrow{d'} L'$. If we have the following diagram with commutative squares: 

\begin{tikzcd}
M \arrow[r, "i"] \arrow[d, "\phi_1"] & N \arrow[r, "d"] \arrow[d, "\phi_2"] & L \arrow[d, "\phi_3"] \\
M' \arrow[r, "i'"]         & N' \arrow[r, "d'"]         & L'         
\end{tikzcd}

where the vertical arrows $\phi_1,\phi_2,\phi_3$ are isomorphisms, then $M'\xrightarrow{i'} N'\xrightarrow{d'} L'$ is a kernel-cokernel pair in $\A$.  
\end{lem} 

\begin{proof} We will only prove that, $\ker(d')=i'$, since the proof of $\operatorname{Coker}(i')=d'$ can be given by a dual argument. 
\\
From the above commutative diagram we have, 
\begin{align}\label{comm}
    \phi_2\circ i=i'\circ\phi_1,\quad\phi_3\circ d=d'\circ\phi_2
\end{align}
Since $M\xrightarrow{i} N\xrightarrow{d} L$ is a kernel-cokernel pair in $\A$, we get $d\circ i=0$. Then by Equation \ref{comm} we have, $d'\circ i'=(\phi_3\circ d\circ\phi_2^{-1})\circ(\phi_2\circ i\circ\phi_1^{-1})=\phi_3\circ d\circ i\circ\phi_1^{-1}=\phi_3\circ 0\circ\phi_1^{-1}=0$. Now let, $K\in{\A}$ and consider a morphism $f':K\to N'$ such that $d'\circ f'=0$. Define $f:=\phi_2^{-1}\circ f':K\to N$. Then by Equation \ref{comm} we get that, $d\circ f=d\circ (\phi_2^{-1}\circ f')=(d\circ \phi_2^{-1})\circ f'=(\phi_3^{-1}\circ d')\circ f'=\phi_3^{-1}\circ (d'\circ f')=\phi_3^{-1}\circ 0=0$. Now, since $\ker(d)=i$, by the universal property of a kernel of a map, there exists a morphism $u:K\to M$ such that $i\circ u=f$. Now define $u':=\phi_1\circ u:K\to M'$. Then by Equation \ref{comm} we get that, $i'\circ u'=i'\circ (\phi_1\circ u)=(i'\circ \phi_1)\circ u=(\phi_2\circ i)\circ u=\phi_2\circ f=\phi_2\circ (\phi_2^{-1}\circ f')=f'$. This implies that, $\ker(d')=i'$. 
\end{proof}

We now record a Lemma on the intersection of exact subcategories. Note that, this is slightly different (and in view of \cite[Corollary 2]{Rump11}, more general) than \cite[Lemma 5.2]{reduc}.     

\begin{lem}\label{inter}
Let $({\A},{\E})$ be an exact category. Also let $({\A}_{\lambda},{\E}_{\lambda})$ be an arbitrary family of exact subcategories of $({\A},{\E})$. Then $(\cap_{\lambda} {\A}_{\lambda},\cap_{\lambda}{\E}_{\lambda})$ is exact subcategory of $({\A},{\E})$. 
\end{lem} 

\begin{proof}
Clearly, $(\cap_{\lambda} {\A}_{\lambda},\cap_{\lambda}{\E}_{\lambda})$ is a strictly full additive subcategory of $({\A},{\E})$. By \cite[Definition 2.1]{Theo} and \cite[Remark 2.4]{Theo}, it is enough to show that $\cap_{\lambda}{\E}_{\lambda}$ is closed under isomorphisms of kernel-cokernel pairs and $(\cap_{\lambda} {\A}_{\lambda},\cap_{\lambda}{\E}_{\lambda})$ satisfies the axioms [E0], [E0$^{\text{op}}$], [E1$^{\text{op}}$], [E2] and [E2$^{\text{op}}$] of \cite[Definition 2.1]{Theo}. First, we will show that $\cap_{\lambda}{\E}_{\lambda}$ is closed under isomorphisms of kernel-cokernel pairs. Let $\begin{tikzcd}
M \arrow[r, tail] & N \arrow[r, two heads] & L
\end{tikzcd}$ be a kernel-cokernel pair in $\cap_{\lambda}{\E}_{\lambda}$, so $M,N,L\in{\A}_{\lambda}$ and $\begin{tikzcd}
M \arrow[r, tail] & N \arrow[r, two heads] & L
\end{tikzcd}$ is a kernel-cokernel pair in ${\A}_{\lambda}$ for all $\lambda$. Also let, $M'\to N'\to L'$ be a kernel-cokernel pair in $\cap_{\lambda}{\A}_{\lambda}$ such that it is isomorphic to $\begin{tikzcd}
M \arrow[r, tail] & N \arrow[r, two heads] & L
\end{tikzcd}$. So, $M',N',L'\in{\A}_{\lambda}$ and $M'\to N'$, $N'\to L'$ are two morphisms in ${\A}_{\lambda}$ for all $\lambda$.
%Since ${\A}_{\lambda}$ is a strictly full additive subcategory of ${\A}$ for all $\lambda$, so $\cap_{\lambda}{\A}_{\lambda}$ is a strictly full additive subcategory of ${\A}_{\lambda}$ for all $\lambda$. 
Hence by Lemma \ref{kerco} we get that, $M'\to N'\to L'$ is a kernel-cokernel pair in ${\A}_{\lambda}$ for all $\lambda$. Since $({\A}_{\lambda},{\E}_{\lambda})$ is an exact category and $\begin{tikzcd}
M \arrow[r, tail] & N \arrow[r, two heads] & L
\end{tikzcd}$ is in ${\E}_{\lambda}$ for all $\lambda$, we have $\begin{tikzcd}
M' \arrow[r, tail] & N' \arrow[r, two heads] & L'
\end{tikzcd}$ is in ${\E}_{\lambda}$ for all $\lambda$. So, $\begin{tikzcd}
M' \arrow[r, tail] & N' \arrow[r, two heads] & L'
\end{tikzcd}$ is in $\cap_{\lambda}{\E}_{\lambda}$. Thus $\cap_{\lambda}{\E}_{\lambda}$ is closed under isomorphisms. Next, we will show that $(\cap_{\lambda} {\A}_{\lambda},\cap_{\lambda}{\E}_{\lambda})$ satisfies the axiom [E0]. Let $A\in\cap_{\lambda} {\A}_{\lambda}$. Since $({\A}_{\lambda}, {\E}_{\lambda})$ is an exact category for all $\lambda$, by \cite[Lemma 2.7]{Theo} we get that $\begin{tikzcd}
A \arrow[r, "1_A", tail] & A\oplus 0\cong A \arrow[r, "0", two heads] & 0
\end{tikzcd}$ is in ${\E}_{\lambda}$ for all $\lambda$. As ${\A}_{\lambda}$ is an additive subcategory of $\A$ for all $\lambda$, so $0_{\A}=0_{{\A}_{\lambda}}$ for all $\lambda$. Hence by definition $\begin{tikzcd}
A \arrow[r, "1_A", tail] & A\oplus 0\cong A \arrow[r, "0", two heads] & 0
\end{tikzcd}$ is in $\cap_{\lambda}{\E}_{\lambda}$, which implies $1_A$ is an admissible monic in $\cap_{\lambda}{\E}_{\lambda}$. So, $(\cap_{\lambda} {\A}_{\lambda},\cap_{\lambda}{\E}_{\lambda})$ satisfies the axiom [E0]. Next, we will show that, $(\cap_{\lambda} {\A}_{\lambda},\cap_{\lambda}{\E}_{\lambda})$ satisfies the axiom [E0$^{\text{op}}$]. Let $A\in\cap_{\lambda} {\A}_{\lambda}$. Since $({\A}_{\lambda}, {\E}_{\lambda})$ is an exact category for all $\lambda$, by \cite[Lemma 2.7]{Theo} we get that $\begin{tikzcd}
0 \arrow[r, "0", tail] & 0\oplus A\cong A \arrow[r, "1_A", two heads] & A
\end{tikzcd}$ is in ${\E}_{\lambda}$ for all $\lambda$. As ${\A}_{\lambda}$ is an additive subcategory of $\A$ for all $\lambda$, so $0_{\A}=0_{{\A}_{\lambda}}$ for all $\lambda$. Hence by definition $\begin{tikzcd}
0 \arrow[r, "0", tail] & A\oplus 0\cong A \arrow[r, "1_A", two heads] & A
\end{tikzcd}$ is in $\cap_{\lambda}{\E}_{\lambda}$, which implies $1_A$ is an admissible epic in $\cap_{\lambda}{\E}_{\lambda}$. So, $(\cap_{\lambda} {\A}_{\lambda},\cap_{\lambda}{\E}_{\lambda})$ satisfies the axiom [E0$^{\text{op}}$]. Next, we will show that $(\cap_{\lambda} {\A}_{\lambda},\cap_{\lambda}{\E}_{\lambda})$ satisfies the axiom [E1$^{\text{op}}$]. Let $\begin{tikzcd}
B' \arrow[r, "e", two heads] & B
\end{tikzcd}$ and $\begin{tikzcd}
B \arrow[r, "p", two heads] & C
\end{tikzcd}$ be two admissible epics in $\cap_{\lambda}{\E}_{\lambda}$. Then we will show that
$\begin{tikzcd}
B' \arrow[r, "p\circ e", two heads] & C
\end{tikzcd}$ is an admissible epic in $\cap_{\lambda}{\E}_{\lambda}$. Now $\begin{tikzcd}
B' \arrow[r, "e", two heads] & B
\end{tikzcd}$ and $\begin{tikzcd}
B \arrow[r, "p", two heads] & C
\end{tikzcd}$ are admissible epics in ${\E}_{\lambda}$ for all $\lambda$. Since $({\A}_{\lambda},{\E}_{\lambda})$ is an exact category for all $\lambda$, we get $\begin{tikzcd}
B' \arrow[r, "p\circ e", two heads] & C
\end{tikzcd}$ is an admissible epic in ${\E}_{\lambda}$ for all $\lambda$. Hence there exist objects $D_{\lambda}\in{\A}_{\lambda}$ and kernel-cokernel pairs $\sigma_{\lambda}:\begin{tikzcd}
D_{\lambda} \arrow[r, "i_{\lambda}", tail] & B' \arrow[r, "p\circ e", two heads] & C
\end{tikzcd}$ in ${\E}_{\lambda}$ for all $\lambda$. So, $\begin{tikzcd}
D_{\lambda} \arrow[r, "i_{\lambda}", tail] & B' \arrow[r, "p\circ e", two heads] & C
\end{tikzcd}$ is in $\E$ for all $\lambda$. Hence $\sigma_{\lambda}$ is a kernel-cokernel pair in $\A$ for all $\lambda$, so by the universal property of kernels we get that the kernel-cokernel pairs $\sigma_{\lambda}$'s are all isomorphic to each other. Hence all the $D_{\lambda}$'s are isomorphic to each other. Now fix a $\lambda$, say $\lambda_0$. Then $D_{\lambda_0}\cong D_{\lambda}$ for all $\lambda$. Since ${\A}_{\lambda}$ is a strict subcategory of $\A$ for all $\lambda$, we have $D_{\lambda_0}\in\cap_{\lambda}{\A}_{\lambda}$. Now ${\A}_{\lambda}$ is a strictly full subcategory of $\A$ for all $\lambda$, so $D_{\lambda_0}\in{\A}_{\lambda}$ implies that $\sigma_{\lambda_0}$ is a kernel-cokernel pair in ${\A}_{\lambda}$ for all $\lambda$. Since ${\E}_{\lambda}$ is closed under
isomorphisms of kernel-cokernel pairs in ${\A}_{\lambda}$ and $\sigma_{\lambda}$'s are all isomorphic to each other, we get $\sigma_{\lambda_0}\in{\E}_{\lambda}$ for all $\lambda$. So, $\sigma_{\lambda_0}\in\cap_{\lambda}{\E}_{\lambda}$. Hence $\begin{tikzcd}
D_{\lambda_0} \arrow[r, "i_{\lambda_0}", tail] & B' \arrow[r, "p\circ e", two heads] & C
\end{tikzcd}$ is in $\cap_{\lambda}{\E}_{\lambda}$, so $\begin{tikzcd}
B' \arrow[r, "p\circ e", two heads] & C
\end{tikzcd}$ is an admissible epic in $\cap_{\lambda}{\E}_{\lambda}$. Thus $(\cap_{\lambda}{\A}_{\lambda},\cap_{\lambda}{\E}_{\lambda})$ satisfies the axiom [E1$^{\text{op}}$]. Now we will show that $(\cap_{\lambda}{\A}_{\lambda},\cap_{\lambda}{\E}_{\lambda})$ satisfies the axiom [E2]. Let  $\begin{tikzcd}
A \arrow[r, "i", tail] & B
\end{tikzcd}$ be an admissible monic in $\cap_{\lambda}{\E}_{\lambda}$ and $A\xrightarrow{f}A'$ be an arbitrary morphism in $\cap_{\lambda}{\A}_{\lambda}$. Now $\begin{tikzcd}
A \arrow[r, "i", tail] & B
\end{tikzcd}$ is an admissible monic in ${\E}_{\lambda}$ for all $\lambda$. Since $({\A}_{\lambda},{\E}_{\lambda})$ is an exact category for all $\lambda$, by \cite[Proposition 2.12(iv)]{Theo} we have the following pushout commutative diagram with rows being kernel-cokernel pairs in ${\E}_{\lambda}$:
$$\begin{tikzcd}
& A \arrow[r, "i", tail] \arrow[d, "f"] & B \arrow[r, "p", two heads] \arrow[d, "f'_{\lambda}"] & C \arrow[d, equal] \\
\beta_{\lambda}:& A'\arrow[ur, phantom, "\scalebox{1.0}{$\text{PO}$}" description] \arrow[r, "i'_{\lambda}", tail]              & B'_{\lambda} \arrow[r, "p'_{\lambda}", two heads]                          & C               
\end{tikzcd}(\text{ Note that, }C\in\cap_{\lambda}{\A}_{\lambda})$$
Since $({\A}_{\lambda},{\E}_{\lambda})$ is an exact subcategory of $({\A},{\E})$, by \cite[Proposition 5.2]{Theo} we get that the square $\begin{tikzcd}
A \arrow[r, "i", tail] \arrow[d, "f"] & B\arrow[d, "f'_{\lambda}"]\\
A'\arrow[r, "i'_{\lambda}", tail]              & B'_{\lambda}
\end{tikzcd}$ is a pushout square in $\E$ for all $\lambda$. Then by the universal property of pushout we get that the kernel-cokernel pairs $\beta_{\lambda}$'s are all isomorphic to each other, so all the $B'_{\lambda}$'s are isomorphic to each other. Now fix a $\lambda$, say $\lambda_0$. Then $B'_{\lambda_0}\cong B'_{\lambda}$ for all $\lambda$. Since ${\A}_{\lambda}$ is a strict subcategory of $\A$ for all $\lambda$, we have $B'_{\lambda_0}\in\cap_{\lambda}{\A}_{\lambda}$. Now ${\A}_{\lambda}$ is a strictly full subcategory of $\A$ for all $\lambda$, so $B'_{\lambda_0}\in{\A}_{\lambda}$ implies that $\beta_{\lambda_0}$ is a kernel-cokernel pair in ${\A}_{\lambda}$ for all $\lambda$. Since ${\E}_{\lambda}$ is closed under
isomorphisms of kernel-cokernel pairs in ${\A}_{\lambda}$ and $\beta_{\lambda}$'s are all isomorphic to each other, we have $\beta_{\lambda_0}\in{\E}_{\lambda}$ for all $\lambda$. So, $\beta_{\lambda_0}\in\cap_{\lambda}{\E}_{\lambda}$. Hence $\begin{tikzcd}
A'\arrow[r, "i'_{\lambda_0}", tail]              & B'_{\lambda_0} \arrow[r, "p'_{\lambda_0}", two heads]                          & C               
\end{tikzcd}$ is in $\cap_{\lambda}{\E}_{\lambda}$, so $\begin{tikzcd}
A'\arrow[r, "i'_{\lambda_0}", tail]              & B'_{\lambda_0}
\end{tikzcd}$ is an admissible monic in $\cap_{\lambda}{\E}_{\lambda}$. Thus $(\cap_{\lambda}{\A}_{\lambda},\cap_{\lambda}{\E}_{\lambda})$ satisfies the axiom [E2]. A dual argument will show that $(\cap_{\lambda}{\A}_{\lambda},\cap_{\lambda}{\E}_{\lambda})$ satisfies the axiom [E2$^{\text{op}}$]. Hence $(\cap_{\lambda}{\A}_{\lambda},\cap_{\lambda}{\E}_{\lambda})$ is an exact subcategory of $({\A},{\E})$.
\end{proof}

Next, we record a useful lemma for proving the exactness of certain structures on additive subcategories of a given exact category.  

\begin{lem}\label{rest}
Let $({\A},{\E})$ be an exact category. Let ${\A}'$ be a strictly full additive subcategory of $\A$. Define ${\E}|_{\A'}:=\{\begin{tikzcd}
M\arrow[r, tail]              & N \arrow[r, two heads]                          & L               
\end{tikzcd}\text{ is in }{\E}:M,N,L\in{\A}'\}$. If the pullback (resp. pushout) in $\A$ of every $\mathcal E$-deflation (resp. inflation) of sequences in $\mathcal E|_{\A'}$ along every morphism in $\mathcal A'$ is again in $\mathcal A'$, then $({\A}',{\E}|_{\A'})$ is an exact subcategory of $({\A},{\E})$. 
\end{lem}

The proof of Lemma \ref{rest} depends on Lemma \ref{kerco} and the following proposition:

\begin{prop}\label{compepic}
Let $({\A},{\E})$ be an exact category. Suppose we have the following diagram:
$$\begin{tikzcd}
A' \arrow[r, "i'"] \arrow[d, "e'", two heads] & B' \arrow[d, "e", two heads] &   \\
A \arrow[ur, phantom, "\scalebox{1.0}{$\text{PB}$}" description]\arrow[r, "i", tail]                        & B \arrow[r, "p", two heads]  & C
\end{tikzcd}
$$
where the square commutative diagram is a pullback diagram and $e,p$ are admissible epics and $i$ is a kernel of $p$. Then $i'$ is an admissible monic with a cokernel given by $\begin{tikzcd}
B' \arrow[r, "p\circ e", two heads] & C
\end{tikzcd}$.
\end{prop}

\begin{proof} 
The proof follows from the construction in the proof of \cite[Proposition 2.15]{Theo}. The only missing point in the proof of \cite[Proposition 2.15]{Theo}, towards showing $i'$ is a kernel of $p\circ e$, is the following: it was not shown that $(p\circ e)\circ i'=0$. This can be easily checked as follows: $(p\circ e)\circ i'=p\circ (e\circ i')=p\circ (i\circ e')=(p\circ i)\circ e'=0\circ e'=0$, since $i$ is a kernel of $p$. 
\end{proof}  

%Now we  give a proof of Lemma \ref{rest}: 

\begin{proof}[Proof of Lemma \ref{rest}:] Clearly, each kernel-cokernel pair in ${\E}|_{\A'}$ is a kernel-cokernel pair in $\A$, hence also a kernel-cokernel pair in $\A'$. Thus ${\E}|_{\A'}$ consists of kernel-cokernel pairs in $\A'$. Now we will show that, $({\A}',{\E}|_{\A'})$ satisfies the axiom [E2]. Let $A,B,A'\in{\A}'$ and let $\begin{tikzcd}
A \arrow[r, "i", tail] & B
\end{tikzcd}$ be an admissible monic in ${\E}|_{\A'}$ and $A\xrightarrow{f}A'$ be an arbitrary morphism in ${\A}'$. Now $\begin{tikzcd}
A \arrow[r, "i", tail] & B
\end{tikzcd}$ is an admissible monic in $\E$. Since $({\A},{\E})$ is an exact category, by \cite[Proposition 2.12(iv)]{Theo} we have the following pushout commutative diagram with rows being kernel-cokernel pairs in $\E$:
$$\begin{tikzcd}
& A \arrow[r, "i", tail] \arrow[d, "f"] & B \arrow[r, "p", two heads] \arrow[d, "f'"] & C \arrow[d, equal] \\
\beta:& A'\arrow[ur, phantom, "\scalebox{1.0}{$\text{PO}$}" description] \arrow[r, "i'", tail]              & B' \arrow[r, "p'", two heads]                          & C
\end{tikzcd}$$
Now by the assumption we have, $B'\in{\A}'$. Since $\begin{tikzcd}
A \arrow[r, "i", tail] & B
\end{tikzcd}$ is an admissible monic in ${\E}|_{\A'}$, we have $C\in{\A}'$. Thus $\beta\in{\E}|_{\A'}$. Since $\A'$ is a strictly full subcategory of $\A$, the above diagram is a pushout diagram in $\A'$ as well. Thus $({\A}',{\E}|_{\A'})$ satisfies the axiom [E2]. Similarly, $({\A}',{\E}|_{\A'})$ satisfies the axiom [E2$^{\text{op}}$]. Now by \cite[Definition 2.1]{Theo} and \cite[Remark 2.4]{Theo}, it is enough to show that ${\E}|_{\A'}$ is closed under isomorphisms and $({\A}',{\E}|_{\A'})$ satisfies the axioms [E0], [E0$^{\text{op}}$] and [E1$^{\text{op}}$]. First,  we will show that ${\E}|_{\A'}$ is closed under isomorphisms. Let $\begin{tikzcd}
M \arrow[r, tail] & N \arrow[r, two heads] & L
\end{tikzcd}$ be a kernel-cokernel pair in ${\E}|_{\A'}$, so $M,N,L\in{\A}'$ and $\begin{tikzcd}
M \arrow[r, tail] & N \arrow[r, two heads] & L
\end{tikzcd}$ is a kernel-cokernel pair in $\A$. Also let, $M'\to N'\to L'$ be a kernel-cokernel pair in $\A'$ such that it is isomorphic to $\begin{tikzcd}
M \arrow[r, tail] & N \arrow[r, two heads] & L
\end{tikzcd}$. Hence by Lemma \ref{kerco} we get that, $M'\to N'\to L'$ is a kernel-cokernel pair in $\A$. Since $({\A},{\E})$ is an exact category, we have $\begin{tikzcd}
M' \arrow[r, tail] & N' \arrow[r, two heads] & L'
\end{tikzcd}$ is in ${\E}$. Hence by definition, $\begin{tikzcd}
M' \arrow[r, tail] & N' \arrow[r, two heads] & L'
\end{tikzcd}$ is in ${\E}|_{\A'}$. Thus ${\E}|_{\A'}$ is closed under isomorphisms. Next, we will show that, $({\A}',{\E}|_{\A'})$ satisfies the axiom [E0]. Let $A\in{\A}'$. Since $({\A},{\E})$ is an exact category, by \cite[Lemma 2.7]{Theo} we get that $\begin{tikzcd}
A \arrow[r, "1_A", tail] & A\oplus 0\cong A \arrow[r, "0", two heads] & 0
\end{tikzcd}$ is in ${\E}$. As $\A'$ is an additive subcategory of $\A$, so $0_{\A}=0_{\A'}$. Hence by definition $\begin{tikzcd}
A \arrow[r, "1_A", tail] & A\oplus 0\cong A \arrow[r, "0", two heads] & 0
\end{tikzcd}$ is in ${\E}|_{\A'}$, which implies $1_A$ is an admissible monic in ${\E}|_{\A'}$. So, $({\A}',{\E}|_{\A'})$ satisfies the axiom [E0]. Next, we will show that, $({\A}',{\E}|_{\A'})$ satisfies the axiom [E0$^{\text{op}}$]. Let $A\in{\A}'$. Since $({\A},{\E})$ is an exact category, by \cite[Lemma 2.7]{Theo} we get that $\begin{tikzcd}
0 \arrow[r, "0", tail] & 0\oplus A\cong A \arrow[r, "1_A", two heads] & A
\end{tikzcd}$ is in ${\E}$. As $\A'$ is an additive subcategory of $\A$, so $0_{\A}=0_{\A'}$. Hence by definition $\begin{tikzcd}
0 \arrow[r, "0", tail] & A\oplus 0\cong A \arrow[r, "1_A", two heads] & A
\end{tikzcd}$ is in ${\E}|_{\A'}$, which implies $1_A$ is an admissible epic in ${\E}|_{\A'}$. So, $({\A}',{\E}|_{\A'})$ satisfies the axiom [E0$^{\text{op}}$]. Next, we will show that $({\A}',{\E}|_{\A'})$ satisfies the axiom [E1$^{\text{op}}$]. Let $\begin{tikzcd}
B' \arrow[r, "e", two heads] & B
\end{tikzcd}$ and $\begin{tikzcd}
B \arrow[r, "p", two heads] & C
\end{tikzcd}$ be two admissible epics in $\E$ such that they are in ${\E}|_{\A'}$, which means that there exist two kernel-cokernel pairs $\begin{tikzcd}
A \arrow[r, "i", tail] & B\arrow[r, "p", two heads] & C
\end{tikzcd}$ and $\begin{tikzcd}
D \arrow[r, tail] & B'\arrow[r, "e", two heads] & B
\end{tikzcd}$ in $\E$ such that  $A,B,C,D,B'\in{\A}'$. Then we will show that
$\begin{tikzcd}
B' \arrow[r, "p\circ e", two heads] & C
\end{tikzcd}$ is an admissible epic in ${\E}|_{\A'}$. From Proposition \ref{compepic} we get that, $\begin{tikzcd}
A' \arrow[r, "i'", tail] & B'\arrow[r, "p\circ e", two heads] & C
\end{tikzcd}$ is a kernel-cokernel pair in $\E$. By the hypothesis of Lemma \ref{rest} and the diagram of Proposition \ref{compepic} we get that, $A'\in{\A}'$. Hence $\begin{tikzcd}
B' \arrow[r, "p\circ e", two heads] & C
\end{tikzcd}$ is an admissible epic in ${\E}|_{\A'}$. Thus $({\A}',{\E}|_{\A'})$ satisfies the axiom [E1$^{\text{op}}$]. So, $({\A}',{\E}|_{\A'})$ is an exact subcategory of $({\A},{\E})$. 
\end{proof}

From now on, given a subcategory $\A'$ of an exact category $({\A},{\E})$, the notation ${\E}|_{\A'}$ will stand for as defined in Lemma \ref{rest}. Using Lemma \ref{rest}, we now record two quick consequences, which give a sufficient condition on a subcategory $\A'$ such that $({\A}',{\E}|_{\A'})$ is an exact subcategory of $({\A},{\E})$. The first of which we state below now is well-known, see for instance \cite[Lemma 10.20]{Theo}.  However, due to the absence of a proof in \cite[Lemma 10.20]{Theo}, we give a proof using our Lemma \ref{rest}.

\begin{prop}\label{extclosed} (c.f. \cite[Lemma 10.20]{Theo})
Let $({\A},{\E})$ be an exact category. Let $\A'$ be a strictly full additive subcategory of $\A$. Assume that for every $X\to Y\to Z$ in ${\E}$, if $X,Z\in {\A}'$, then $Y\in{\A}'$ (i.e., $\A'$ is closed under extensions). Then $({\A}',{\E}|_{\A'})$ is an exact subcategory of $({\A},{\E})$. 
\end{prop}      

\begin{proof}
By Lemma \ref{rest}, it is enough to show that the pullback (resp. pushout) in $\A$ of every $\mathcal E$-deflation (resp. inflation) of sequences in $\mathcal E|_{\A'}$ along every morphism in $\mathcal A'$ is again in $\mathcal A'$. Let $\begin{tikzcd}
A \arrow[r, "i'", tail] & B
\end{tikzcd}$ be an inflation in $\mathcal E|_{\A'}$ and $f:A\to A'$ be a morphism in $\mathcal A'$. Then by \cite[Proposition 2.12(iv)]{Theo} we get the following pushout commutative diagram:
$$\begin{tikzcd}
A \arrow[r, "i", tail] \arrow[d, "f"] & B \arrow[r, "p", two heads] \arrow[d, "f'"] & C \arrow[d, equal] \\
A'\arrow[ur, phantom, "\scalebox{1.0}{$\text{PO}$}" description] \arrow[r, "i'", tail]              & B' \arrow[r, "p'"]                          & C               
\end{tikzcd}$$
Since $A',C\in{\A}'$, by assumption and the bottom row of the above diagram we get that $B'\in{\A}'$. The pullback case follows by a dual argument. Hence $({\A}',{\E}|_{\A'})$ is an exact subcategory of $({\A},{\E})$.  
\end{proof}  

\begin{theorem}\label{specext}
Let $({\A},{\E})$ be an exact category. Let $\A'$ be a strictly full additive subcategory of $\A$. Assume that $\A'$ is closed under kernels and co-kernels of admissible epics and monics in $\E$ respectively. Then, $({\A}',{\E}|_{\A'})$ is an exact subcategory of $({\A},{\E})$. 
\end{theorem}    

\begin{proof}   
By Lemma \ref{rest}, it is enough to show that the pullback (resp. pushout) in $\A$ of every $\mathcal E$-deflation (resp. inflation) of sequences in $\mathcal E|_{\A'}$ along every morphism in $\mathcal A'$ is again in $\mathcal A'$. Let $\begin{tikzcd}
A \arrow[r, "i'", tail] & B
\end{tikzcd}$ be an inflation in $\mathcal E|_{\A'}$ and $f:A\to A'$ be a morphism in $\mathcal A'$. Now we have the following pushout commutative diagram in $(A,{\E})$:
$$\begin{tikzcd}
A \arrow[r, "i", tail] \arrow[d, "f"] & B\arrow[d, "f'"]\\
A'\arrow[ur, phantom, "\scalebox{1.0}{$\text{PO}$}" description] \arrow[r, "i'", tail]              & B'             
\end{tikzcd}$$
Then by \cite[Proposition 2.12(ii)]{Theo} we have the following kernel-cokernel pair in ${\E}$: $\begin{tikzcd}
A \arrow[r, tail] & B\oplus A' \arrow[r, two heads] & B'
\end{tikzcd}$. As $B,A'\in{\A}'$ and $\A'$ is additive, so $B\oplus A'\in{\A}'$. Since $\begin{tikzcd}
A \arrow[r, tail] & B\oplus A'
\end{tikzcd}$ is an admissible monic in $\E$, we get the cokernel $B'\in{\A}'$. The pullback case follows by a dual argument. Hence $({\A}',{\E}|_{\A'})$ is an exact subcategory of $({\A},{\E})$.
\end{proof}   

Given any exact category $({\A},{\E})$ and $C,A\in {\A}$, one can define the Yoneda Ext group $\Ext_{\E}(C,A)$, which has an abelian group structure by Baer sum (see the beginning of \cite[Section 1.2]{vec} for a description of $\Ext_{\E}(-,-)$). When $({\A},{\E})$ is moreover an $R$-linear category, then $\Ext_{\E}(C,A)$ can be given an $R$-linear structure via either of the following constructions, both of which yield equivalent triples in $\Ext_{\E}(C,A)$: 

Given a kernel-cokernel pair $\begin{tikzcd} 
\sigma: A \arrow[r, tail] & B \arrow[r, two heads] & C
\end{tikzcd}$ in $\Ext_{\E}(C,A)$, the multiplication $r\cdot \sigma$ is either given by the following pullback diagram:

$$\begin{tikzcd}
r\sigma: & A \arrow[r, tail] \arrow[d, equal] & B'\arrow[d]\arrow[r, two heads] & C \arrow[d, "r\cdot \text{id}_C"]\\
\sigma: & A \arrow[r, tail]              & B \arrow[ur, phantom, "\scalebox{1.0}{$\text{PB}$}" description]  \arrow[r, two heads]               & C           
\end{tikzcd}$$

or by the following pushout diagram: 
$$\begin{tikzcd}
\sigma: & A \arrow[r, tail] \arrow[d, "r\cdot \text{id}_A"'] & B\arrow[d]\arrow[r, two heads] & C \arrow[d, equal]\\
r\sigma: & A\arrow[ur, phantom, "\scalebox{1.0}{$\text{PO}$}" description] \arrow[r, tail]  & B' \arrow[r, two heads]  & C           
\end{tikzcd}$$

That both of these yield equivalent triplet in $\Ext_{\E}(-,-)$ follows from \cite[Proposition 3.1]{Theo}. Moreover, this makes $\Ext_{\E}(-,-)$ into an $R$-linear functor in each component.  

Now we show that if $({\A}',{\E'})$ is a strictly full exact subcategory of an exact category $({\A},{\E})$, then $\Ext_{\E'}(-,-):{\A}'^{\operatorname{op}}\times{\A}'\to\mathbf{Ab}$ is naturally a subfunctor of $\Ext_{\E}(-,-)|_{{\A}'^{\operatorname{op}}\times{\A}'}:{\A}'^{\operatorname{op}}\times{\A}'\to\mathbf{Ab}$. For this, we first record a remark.

\begin{rem}\label{equi} Let $({\A}',{\E} ')$ be a strictly full exact subcategory of $(\A,{\E})$. Then, for $\sigma \in {\E} ' (\subseteq {\E})$, it is actually true that $[\sigma]_{\E '}=[\sigma]_{\E}$, hence the map $\Ext_{\E '}(-,-)\xrightarrow{[\sigma]_{\E '}\mapsto [\sigma]_{\E}} \Ext_{\E}(-,-)$ is the natural inclusion map.  Indeed, to see $[\sigma]_{\E '}=[\sigma]_{\E}$: Let $\sigma$ be a kernel-cokernel pair $\begin{tikzcd}
A \arrow[r, tail] & B \arrow[r, two heads] & C
\end{tikzcd}$ in ${\E}'$, so $A,B,C\in{\A}'$.  Let $\beta\in [\sigma]_{\E}$ be the kernel-cokernel pair $\begin{tikzcd}
A \arrow[r, tail] & B' \arrow[r, two heads] & C
\end{tikzcd}$ in $\E$, so $B'\in{\A}$. Hence there exists $f\in\Mor_{\A}(B,B')$ such that we have the following commutative diagram in $\A$:
$$\begin{tikzcd}
\sigma: & A \arrow[r, tail] \arrow[d, equal] & B\arrow[d, "f"]\arrow[r, two heads] & C \arrow[d,equal]\\
\beta: & A\arrow[r, tail]              & B'  \arrow[r, two heads]               & C           
\end{tikzcd}$$

By \cite[Corollary 3.2]{Theo} we get that, $f$ is an isomorphism. Since $B\in{\A}'$, $B'\in{\A}$ and $\A'$ is a strict subcategory of $\A$, $B\cong B'$ implies that $B'\in{\A}'$. Since $\A'$ is full, $f\in\Mor_{\A'}(B,B')$. Hence the above commutative diagram is in $\A'$. Since $({\A}',{\E}')$ is an exact category, ${\E}'$ is closed under isomorphisms. Hence $\beta\in{\E}'$, $\beta\in[\sigma]_{\E'}$. Thus $[\sigma]_{\E}\subseteq[\sigma]_{\E'}$. Now $({\A}',{\E}')$ is a subcategory of $({\A},{\E})$, so $[\sigma]_{\E'}\subseteq[\sigma]_{\E}$ as well. Hence $[\sigma]_{\E'}=[\sigma]_{\E}$.  
\end{rem}

\begin{prop}\label{exsub}
Let $({\A},{\E})$ be an exact category. Let $({\A}',{\E}')$ be a strictly full exact subcategory of $({\A},{\E})$. Then $\Ext_{\E'}(-,-):{\A}'^{\operatorname{op}}\times{\A}'\to\mathbf{Ab}$ is a subfunctor of $\Ext_{\E}(-,-)|_{{\A}'^{\operatorname{op}}\times{\A}'}:{\A}'^{\operatorname{op}}\times{\A}'\to\mathbf{Ab}$, where for every $C,A\in {\A}'$, the natural inclusion map $\Ext_{\E'}(C,A)\to \Ext_{\E}(C,A)$ is given by $[\sigma]_{\E'} \mapsto [\sigma]_{\E}$. If ${\A}$ is moreover $R$-linear, then the natural inclusion map is also $R$-linear, hence  $\Ext_{\E'}(-,-):{\A'}^{\operatorname{op}}\times{\A'}\to\Mod R$ is a subfunctor of $\Ext_{\E}(-,-)|_{{\A'}^{\operatorname{op}}\times{\A}'}:{\A'}^{\operatorname{op}}\times{\A}'\to \Mod R$.  
\end{prop}

\begin{proof} Let $C,A\in {\A}'$. Define $\phi_{C,A}:\Ext_{\E'}(C,A)\to \Ext_{\E}(C,A)$ by $\phi_{C,A}([\sigma]_{\E'})=[\sigma]_{\E}$. From now on, we will call this map $\phi$. The well-definedness and injectivity of $\phi$ follow from Remark \ref{equi}. Next, we will show that $\phi$ is a group homomorphism. Let $[\sigma]_{\E'},[\beta]_{\E'}\in\Ext_{\E'}(C,A)$, where $\sigma: \begin{tikzcd}
A \arrow[r, tail] & B \arrow[r, two heads] & C
\end{tikzcd}$ and $\beta:\begin{tikzcd}
A \arrow[r, tail] & B' \arrow[r, two heads] & C
\end{tikzcd}$ for some $B,B'\in{\A}'$. Then $[\sigma]_{\E'}+[\beta]_{\E'}=[\alpha]_{\E'}$ is given by the following commutative diagram in $({\A}',{\E}')$:
$$\begin{tikzcd}
\sigma\oplus\beta: & A\oplus A \arrow[r, tail] \arrow[d, "\Sigma"] & B\oplus B' \arrow[r, two heads] \arrow[d] & C\oplus C \arrow[d,equal]\\
\gamma: & A \arrow[r, tail]                             & E \arrow[r, two heads]                  & C\oplus C\\
{\alpha}: & A \arrow[r, tail] \arrow[u,equal] & B+B' \arrow[u] \arrow[r, two heads]       & C \arrow[u, "\Delta"]
\end{tikzcd}$$
We know that, $\gamma$ is the pushout of $\sigma\oplus\beta$ in $\E'$ by the sum map $A\oplus A\xrightarrow{\Sigma}A$. Since $({\A}',{\E}')$ is an exact subcategory of $({\A},{\E})$, by \cite[Proposition 5.2]{Theo} we get that $\gamma$ is also the pushout of $\sigma\oplus\beta$ in ${\E}$ by the sum map $A\oplus A\xrightarrow{\Sigma}A$. Next, $\alpha$ is the pullback of $\gamma$ in $\E'$ by the diagonal map $C\xrightarrow{\Delta} C\oplus C$, so by \cite[Proposition 5.2]{Theo} we get that $\alpha$ is also the pullback of $\gamma$ in $\E$ by the diagonal map $C\xrightarrow{\Delta} C\oplus C$. This implies that $\alpha$ is a representative of $[\sigma]_{\E}+[\beta]_{\E}$ in $\Ext_{\E}(C,A)$, i.e., $[\alpha]_{\E}=[\sigma]_{\E}+[\beta]_{\E}$. Since $\phi([\sigma]_{\E'}+[\beta]_{\E'})=\phi([\alpha]_{\E'})=[\alpha]_{\E}$ and $\phi([\sigma]_{\E'})+\phi([\beta]_{\E'})=[\sigma]_{\E}+[\beta]_{\E}$, we get $\phi$ is a group homomorphism.
\\
Let $A,B,C\in {\A}'$ and let $f\in\Mor_{\A'}(A,B)$. First, we will show that $\Ext_{\E'}(C,-)$ is a subfunctor of $\Ext_{\E}(C,-)$, so we need to prove that the following diagram commutes:
$$\begin{tikzcd}
{\Ext_{\E}(C,A)} \arrow[r, "f^*"]  & {\Ext_{\E}(C,B)}       \\
{\Ext_{\E'}(C,A)} \arrow[r, "\tilde{f}"] \arrow[u, "\phi_{C,A}" , hook] & {\Ext_{\E'}(C,B)} \arrow[u, "\phi_{C,B}" , hook]
\end{tikzcd}$$
where for every $[\sigma]_{\E} \in\Ext_{\E}(C,A)$, a representative of $f^*([\sigma]_{\E})$ is the pushout of $\sigma$ by $f$ in $\E$ and for every $[\beta]_{\E'} \in\Ext_{\E'}(C,A)$, a representative of $\tilde{f}([\beta]_{\E'})$ is the pushout of $\beta$ by $f$ in $\E'$.
\\
Let, $[\sigma]_{\E'}\in\Ext_{\E'}(C,A)$. Now  $\phi_{C,A}([\sigma]_{\E'})=[\sigma]_{\E}$ and $f^*([\sigma]_{\E})=[\gamma]_{\E}$, where $\gamma$ is the pushout of $\sigma$ by $f$ in $\E$. Next we have, $\tilde{f}([\sigma]_{\E'})=[\alpha]_{\E'}$, where $\alpha$ is the pushout of $\sigma$ by $f$ in $\E'$. Hence by \cite[Proposition 5.2]{Theo} we get that, $\alpha$ is also the pushout of $\sigma$ by $f$ in $\E$. Hence $[\alpha]_{\E}=[\gamma]_{\E}$, so $\phi_{C,B}(\tilde{f}([\sigma]_{\E'}))=\phi_{C,B}([\alpha]_{\E'})=[\alpha]_{\E}=[\gamma]_{\E}=f^*([\sigma]_{\E})=f^*(\phi_{C,A}([\sigma]_{\E'}))$, therefore the above diagram commutes. Hence $\Ext_{\E'}(C,-)$ is a subfunctor of $\Ext_{\E}(C,-)$ for any $C\in{\A}'$. Similarly, $\Ext_{\E'}(-,C)$ is a subfunctor of $\Ext_{\E}(-,C)$ for any $C\in{\A}'$. Thus $\Ext_{\E'}(-,-):{\A}'^{\operatorname{op}}\times{\A}'\to\mathbf{Ab}$ is a subfunctor of $\Ext_{\E}(-,-)|_{{\A}'^{\operatorname{op}}\times{\A}'}:{\A}'^{\operatorname{op}}\times{\A}'\to\mathbf{Ab}$.
\\  
Now let $\A$ be $R$-linear, hence so is $\A'$ because it is a full subcategory of $\A$. Let $C,A\in {\A}'$ and $[\sigma]_{\E'}\in \Ext_{\E'}(C,A)$. Let $[\gamma]_{\E'}=r\cdot[\sigma]_{\E'}$, so $\gamma$ is obtained from $\sigma$ by pullback in $\E'$ along the map $C\xrightarrow{r\cdot \text{id}_C} C$. Now $({\A}',{\E}')$ is an exact subcategory of $({\A},{\E})$, hence this is also the pullback in $\E$ along the map $C\xrightarrow{r\cdot \text{id}_C} C$ (\cite[Proposition 5.2]{Theo}). So, $[\gamma]_{\E}=r\cdot[\sigma]_{\E}=r\phi([\sigma]_{\E'})$. Also, $[\gamma]_{\E}=\phi([\gamma]_{\E'})=\phi(r\cdot[\sigma]_{\E'})$. This shows that the natural inclusion map $\phi$ is $R$-linear. 
\end{proof} 

In the following corollary, we denote $\Hom_{\A}(-,-)$ just by $\Hom(-,-)$ (we completely ignore the subcategory, since all our subcategories are full). In view of Definition \ref{numdef}, and the discussion following \ref{chnum}, the following corollary is crucial for recovering \cite[Proposition 1.37, Proposition 1.38]{thesis}.     

\begin{cor}\label{long} Let $({\A},{\E})$ be an exact category. Let $({\A}',{\E}')$ be a strictly full exact subcategory of $({\A},{\E})$. Let $\sigma:\begin{tikzcd}
M \arrow[r, "i", tail] & N \arrow[r, "p", two heads] & L
\end{tikzcd}$ be a kernel-cokernel pair in $\E'$.  Then, for every $A\in {\A}'$, we have the following commutative diagrams of long exact sequences: 

$$\begin{tikzcd}[sep=1.8em, font=\small]
0 \arrow[r] & {\Hom(A,M)} \arrow[r]                 & {\Hom(A,N)} \arrow[r]                 & {\Hom(A,L)} \arrow[r]                 & {\Ext_{\E}(A,M)} \arrow[r]                 & {\Ext_{\E}(A,N)} \arrow[r]                 & {\Ext_{\E}(A,L)}                 \\
0 \arrow[r] & {\Hom(A,M)} \arrow[r] \arrow[u, equal] & {\Hom(A,N)} \arrow[r] \arrow[u, equal] & {\Hom(A,L)} \arrow[r] \arrow[u, equal] & {\Ext_{\E'}(A,M)} \arrow[r] \arrow[u, hook] & {\Ext_{\E'}(A,N)} \arrow[r] \arrow[u, hook] & {\Ext_{\E'}(A,L)} \arrow[u, hook]
\end{tikzcd}$$  and  

$$\begin{tikzcd}[sep=1.8em, font=\small]
0 \arrow[r] & {\Hom(L,A)} \arrow[r]                 & {\Hom(N,A)} \arrow[r]                 & {\Hom(M,A)} \arrow[r]                 & {\Ext_{\E}(L,A)} \arrow[r]                 & {\Ext_{\E}(N,A)} \arrow[r]                 & {\Ext_{\E}(M,A)}                 \\
0 \arrow[r] & {\Hom(L,A)} \arrow[r] \arrow[u, equal] & {\Hom(N,A)} \arrow[r] \arrow[u, equal] & {\Hom(M,A)} \arrow[r] \arrow[u, equal] & {\Ext_{\E'}(L,A)} \arrow[r] \arrow[u, hook] & {\Ext_{\E'}(N,A)} \arrow[r] \arrow[u, hook] & {\Ext_{\E'}(M,A)} \arrow[u, hook]
\end{tikzcd}$$ 
\end{cor}   

\begin{proof}
We will only prove the covariant version (the first diagram above), since the proof of the contravariant version is given by the dual argument.
\\
In the first diagram above, the commutativity of the first two squares is obvious. Also, the commutativity of the last two squares follows directly from the proof of Proposition \ref{exsub} (i.e, $\Ext_{\E'}(-,-)$ is a subfunctor of $\Ext_{\E}(-,-)$). So, it is enough to prove that the following square is commutative:
$$\begin{tikzcd}[sep=1.8em, font=\small]
{\Hom(A,L)} \arrow[r, "f"]                 & {\Ext_{\E}(A,M)} \\
{\Hom(A,L)} \arrow[r, "f^*"] \arrow[u, equal] & {\Ext_{\E'}(A,M)} \arrow[u, "\phi_{A,M}", hook]
\end{tikzcd}$$
Now $f(g)=[\gamma]_{\E}$, where $\gamma$ is the pullback of $\sigma$ by $g$ in $\E$ and $f^*(g)=[\alpha]_{\E'}$, where $\alpha$ is the pullback of $\sigma$ by $g$ in $\E'$. Since $({\A}',{\E}')$ is an exact subcategory of $({\A},{\E})$, by \cite[Proposition 5.2]{Theo} we get that $[\gamma]_{\E}=[\alpha]_{\E}$. Also, by definition $\phi_{A,M}([\alpha]_{\E'})=[\alpha]_{\E}$. Thus $f(g)=\phi_{A,M}(f^*(g))$, so the above square commutes. 
\end{proof}

Let $({\A},{\E})$ be an exact category. Let $\A'$ be a strictly full subcategory of $\A$. If 
$\begin{tikzcd}
A_1 \arrow[r, "i", tail] & B \arrow[r, "p", two heads] & A_2\end{tikzcd}$ and $\begin{tikzcd}
A_1 \arrow[r, "i", tail] & C \arrow[r, "p", two heads] & A_2\end{tikzcd}$ are two kernel-cokernel pairs in $\E$ belonging to the same class in $\Ext_{\E}(A_2,A_1)$, then $B\cong C$ (by short five lemma), hence $B\in {\A}'$ if and only if $C\in {\A}'$. Hence, for every $M,N\in {\A}'$, the collection $\{[\sigma]_{\E}\in \Ext_{\E}(M,N) : \text{the middle object of } \sigma \text{ is in } {\A}'\}$ is well-defined.   

\begin{prop} Let $({\A},{\E})$ be an exact category. Let $\A'$ be a strictly full additive subcategory of $\A$. For every $M,N\in {\A}'$, define $F(M,N):=\{[\sigma]_{\E}\in \Ext_{\E}(M,N) : \text{ the middle object of } \sigma \text{ is in } {\A}'\}$. If $F(-,-)$ is a subfunctor of $\Ext_{\E}(-,-)|_{{\A}'^{\operatorname{op}}\times{\A}'}:{\A}'^{\operatorname{op}}\times{\A}'\to\mathbf{Ab}$ via the natural inclusion maps, then $({\A}',{\E}|_{\A'})$ is an exact subcategory of $({\A},{\E})$.
\end{prop}   

\begin{proof}
By Lemma \ref{rest}, it is enough to show that the pullback (resp. pushout) in $\A$ of every $\E$-deflation (resp. inflation) of sequences in ${\E}|_{\A'}$ along every morphism in $\A'$ is again in $\A'$. We will only prove the pullback case, since the proof of the pushout case can be given by a similar argument. Let $\begin{tikzcd}
B\arrow[r, two heads] & C
\end{tikzcd}$ be an admissible epic in $\E$ such that it is in ${\E}|_{\A'}$, which means that there exist a kernel-cokernel pair
$\begin{tikzcd}
\gamma: A \arrow[r, tail] & B\arrow[r, two heads] & C
\end{tikzcd}$ in $\E$ such that $A, B,C\in{\A}'$. Also, let $B'\xrightarrow{f} C$ be a morphism in $\A'$. Then we have a map $f^*:\Ext_{\E}(C,A)\to\Ext_{\E}(B',A)$ defined as follows: for every $[\sigma]_{\E} \in\Ext_{\E}(C,A)$, a representative of $f^*([\sigma]_{\E})$ is the pullback of $\sigma$ by $f$ in $({\A},{\E})$. Now by definition of $F(-,-)$, it is clear that $F(M,N)\subseteq\Ext_{\E}(M,N)$ for any $M,N\in{\A}'$. Since $F(-,-)$ is a subfunctor of $\Ext_{\E}(-,-)|_{{\A}'^{\operatorname{op}}\times{\A}'}:{\A}'^{\operatorname{op}}\times{\A}'\to\mathbf{Ab}$, we have the following commutative square:
$$\begin{tikzcd}
{\Ext_{\E}(C,A)} \arrow[r, "f^*"]  & {\Ext_{\E}(B',A)}       \\
{F(C,A)} \arrow[r, "\tilde{f}"] \arrow[u, hook] & {F(B',A)} \arrow[u, hook]
\end{tikzcd}$$
where the columns are natural inclusion maps. So, for every $[\beta]_{\E} \in F(C,A)$, a representative of $\tilde{f}([\beta]_{\E})$ is the pullback of $\beta$ by $f$ in $({\A},{\E})$. Now consider the pullback of $\gamma$ by $f$ in $({\A},{\E})$ as follows:
$$\begin{tikzcd}
A\arrow[r, tail] & A' \arrow[r, two heads] \arrow[d] & B'\arrow[d, "f"]\\
\gamma: A \arrow[r, tail] & B\arrow[ur, phantom, "\scalebox{1.0}{$\text{PB}$}" description] \arrow[r, two heads]              & C
\end{tikzcd}$$
Now $[\gamma]_{\E}\in F(C,A)$, so  $\tilde{f}([\gamma]_{\E})\in F(B',A)$ implies that $A'\in{\A}'$. Hence the pullback in $\A$ of every $\E$-deflation of sequences in ${\E}|_{\A'}$ along every morphism in $\A'$ is again in $\A'$. Hence by Lemma \ref{rest} we get that, $({\A}',{\E}|_{\A'})$ is an exact subcategory of $({\A},{\E})$. 
\end{proof} 

\section[subcategories and subfunctors of $\Ext^1$ from numerical functions and applications to module categories]{subcategories and subfunctors of $\Ext^1$ from numerical functions and applications to module categories}\label{sec4}    

In this section, we present tools to identify exact subcategories of an exact category coming from certain numerical functions. Consequently, due to Proposition \ref{exsub}, we are also able to identify subfunctors of $\Ext^1$ associated with certain numerical functions. Our first result in this direction is an application of Theorem \ref{specext}. 

\begin{prop}\label{funcspec} Let $({\A},{\E})$ be an exact category. Let $\phi:{\A}\to \mathbb {Z}_{\le 0}$ be a function such that $\phi$ is constant on isomorphism classes of objects in $\A$, $\phi$ is additive on finite biproducts, and $\phi$ is sub-additive on kernel-cokernel pairs in $\E$ (i.e., if $\begin{tikzcd}
M \arrow[r, tail] & N \arrow[r, two heads] & L
\end{tikzcd}$ is in $\E$, then $\phi(N)\le \phi(M)+\phi(L)$). Let $\A'$ be the strictly full subcategory of $\A$, whose objects are given by $\{M\in {\A}: \phi(M)=0\}$. Then $({\A}',{\E}|_{\A'})$ is an exact subcategory of $({\A},{\E})$. 
\end{prop}  

\begin{proof} Note that $\A'$ is a strict subcategory of $\A$ since $\phi$ is constant on isomorphism classes of objects in $\A$. Since $\phi$ is additive on finite biproduct, $\phi(0_{\A})=\phi(0_{\A}\oplus 0_{\A})=\phi(0_{\A})+\phi(0_{\A})$, hence $\phi(0_{\A})=0$. Hence $0_{\A}\in {\A}'$. Since $\A$ is an additive category and $\phi$ is additive on finite biproducts, by definition of $\A'$ we get that for every $X,Y\in {\A}'$, the biproduct of $X$ and $Y$ in $\A$ also belongs to $\A'$. Hence $\A'$ is an additive subcategory of $\A$. Let $B\in{\A}'$ and $\begin{tikzcd}
A_1 \arrow[r, "i", tail] & B \arrow[r, "p", two heads] & A_2
\end{tikzcd}$ be a kernel-cokernel pair in $\mathcal E$. We will show that $A_1, A_2 \in{\A}'$. Since $\phi$ is sub-additive on kernel-cokernel pairs in $\E$ and $B\in {\A}'$, we get $0=\phi(B)\leq\phi(A_1)+\phi(A_2)\le 0$. Hence $\phi(A_1)+\phi(A_2)=0$. Since we know $\phi$ always takes non-positive values, we get $\phi(A_1)=\phi(A_2)=0$, hence $A_1,A_2\in {\A}'$. Thus, $\A'$ is closed under kernels and co-kernels of admissible epics and monics in $\E$ respectively. Then by Theorem \ref{specext} we get that, $({\A}',{\E}|_{\A'})$ is an exact subcategory of $({\A},{\E}$.  
\end{proof}  

We now proceed to give our main example of Proposition \ref{funcspec}.  First, we recall some terminologies and prove some preliminary lemmas. For any unexpected concepts and notations, we refer the reader to \cite{bh} and \cite{lw}. 

Let $R$ be a commutative ring, and let ${\S}_R$ be the collection of all short exact sequences of $R$-modules, which gives the standard exact structure on $\Mod R$.  When the ring in question is clear, we drop the suffix $R$ and write only $\S$. Note that, since $\mod R$ is extension closed in $\Mod R$, hence the collection of all short exact sequences in $\mod R$ gives an exact subcategory of $\Mod R$, and we take this as the standard exact structure of $\mod R$. Note that, if $({\X},{\E})$ is an exact subcategory of $\Mod R$ and ${\X}\subseteq \mod R$, then $({\X},{\E})$ is also an exact subcategory of $\mod R$.

Now for a Noetherian local ring $(R,\m)$ of dimension $d$ and for an integer $s\ge 0$, let $\cm^s(R)$ denote the full subcategory of $\mod R$ consisting of the zero-module, and all Cohen--Macaulay $R$-modules (\cite[Definition 2.1.1]{bh}) of dimension $s$. Note that, when $s=d$, $\cm^d(R)$ is just the category of all maximal Cohen--Macaulay modules, which we will also denote by $\cm(R)$.

\begin{chunk}\label{cmchu} We quickly note that for each $s\ge 0$, $\cm^s(R)$ is closed under finite direct sums, direct summands and closed under extensions in $\Mod R$.  Indeed, let $0\to L \to M \to N \to 0$ be a short exact sequence with $L,N\in \cm^s(R)$. If $L$ or $N$ is zero, then there is nothing to prove. So, assume $L,N$ are non-zero. Now $M\in \mod R$, and we also have the following calculation for $\dim M$: $\dim M = \dim \supp(M)=\dim \left(\supp(L)\cup\supp(N)\right)=\max\{\dim \supp(L),\dim \supp(N)\}=\max\{\dim L, \dim N\}=s$. Consequently, $s=\inf\{\depth_R L, \depth_R N\}\le \depth_R M \le \dim M=s$, which shows $M\in \cm^s(R)$.  Thus, $\cm^s(R)$ is closed under extensions in $\Mod R$. To also show $\cm^s(R)$ is closed under direct summands, let $M,N$ be non-zero modules with $M\oplus N\in \cm^s(R)$. Then, $M,N\in \mod R$, and $s=\depth_R(M\oplus N)=\inf\{\depth_R M,\depth_R N\}\le \depth_R M\le \dim M\le \dim (M\oplus N)=s$, thus $M\in \cm^s(R)$. 
\end{chunk} 

Hence, if $\S$ is the standard exact structure on $\Mod R$, then $(\cm^s(R),{\S}|_{\cm^s(R)})$ is an exact subcategory of $\Mod R$ by Proposition \ref{extclosed}. Now let $I$ be an $\m$-primary ideal and $\phi_I:\cm^s(R)\to \mathbb Z$ be the function defined by $\phi_I(M):=\lambda_R(M/IM)-e_R(I,M)$, where $e_R(I,M)$ is the multiplicity of $M$ with respect to $I$ (\cite[Definition 4.6.1]{bh}).  We first show that $\phi_I$ satisfy all the hypothesis of Proposition \ref{funcspec}. To prove this, we will need to pass to the faithfully flat extension $S=R[X]_{\m[X]}$ whose unique maximal ideal is $\m S$. That we can harmlessly pass to this extension, is discussed in the following 

\begin{chunk}\label{infres} Let $(R,\m)$ be a local ring and consider $S=R[X]_{\m[X]}$ which is a faithfully flat extension of $R$ and the unique maximal ideal of $S$ is $\m S$, whose residue field $S/\m S$ is infinite (see \cite[Section 8.4]{HS}). Let $M\in \mod (R)$. By \cite[Theorem 2.1.7]{bh} (and the sentences following it), we have $M\in \cm^s(R)$ if and only if $M\otimes_R S \in \cm^s(S)$. Let $I$ be an $\m$-primary ideal of $R$. As $\m S$ is the maximal ideal of $S$, so $IS=I\otimes_R S$ is also primary to the maximal ideal of $S$. Now, $S\otimes_R \dfrac {M}{IM}\cong \dfrac{S\otimes_R M}{S\otimes_R (IM)}=\dfrac{S\otimes_R M}{(IS)(S\otimes_R M)}$ by Lemma \ref{tenfaith}(1). Hence, $\lambda_S\left(\dfrac{S\otimes_R M}{(IS)(S\otimes_R M)}\right)=\lambda_S\left(S\otimes_R \dfrac {M}{IM} \right)=\lambda_R\left(\dfrac{M}{IM}\right)$, where the last equality is by \cite[Tag 02M1]{st} remembering that $S/\m S$ is the residue field of $S$. Also, $e_R(I,M):=\lim_{n\to \infty}\dfrac{(\dim M)!}{n^{\dim M}}\lambda_R\left(\dfrac{M}{I^{n+1}M}\right)$, and $\dim_S(M\otimes_R S)=\dim_R(M)$ (\cite[Theorem A.11(b)]{bh}), so a similar argument as the previous one shows $e_R(I,M)=e_S(IS, S\otimes_R M)$. Thus $\phi_I(M)=\phi_{IS}(S\otimes_R M)$.  
\end{chunk}

\begin{lem}\label{ineq} The function $\phi_I:\cm^s(R)\to \mathbb Z$ satisfies $\phi_I(M)\le 0$, is constant on isomorphism classes of modules in $\cm^s(R)$, additive on finite biproducts, and subadditive on short exact sequences of modules in $\cm^s(R)$.  
\end{lem} 

\begin{proof} It is obvious that $\phi_I$ is constant on isomorphism classes of modules in $\cm^s(R)$, and additive on finite biproducts (direct sums). Now, $e_R(I,-)$ is additive on short exact sequences of modules in $\cm^s(R)$ by \cite[Corollary 4.7.7]{bh}. Since $\lambda_R((-)\otimes_R R/I)$ is always subadditive on short exact sequences, this proves $\phi_I$ is also subadditive on short exact sequences of modules in $\cm^s(R)$. Now we finally prove  that $\phi_I(M)\le 0$ for all $M\in \cm^s(R)$. If $\dim M=0$, then $e_R(I,M)=\lambda_R(M)\ge \lambda_R(M/IM)$, so $\phi_I(M)\le 0$. Now assume that $s=\dim M>0$. We may assume that $R$ has infinite residue field due to \ref{infres}. By \cite[Corollary 4.6.10]{bh}, we have $e_R(I,M)=e_R((\mathbf x),M)$ for some system of parameters $\mathbf x=x_1,...,x_s$ on $M$, which is a reduction of $I$ with respect to $M$. Then, $\mathbf x$ is a $M$-regular sequence, since $M$ is Cohen--Macaulay (\cite[Theorem 2.1.2(d)]{bh}). Let $J=(\mathbf x)$.  Since $\mathbf x$ is a $M$-regular sequence, by \cite[Theorem 1.1.8]{bh} we have $J^nM/J^{n+1}M\cong (M/JM)^{\oplus \dfrac{(s+n-1)!}{n!(s-1)!}}$. Hence, $e_R(I,M)=e_R(J,M)=(s-1)!\lim_{n\to \infty}\dfrac{\lambda_R(J^nM/J^{n+1}M)}{n^{s-1}}=\lambda_R(M/JM)\ge \lambda_R(M/I M)$ (the last inequality follows by noticing that $J\subseteq I$). Hence $\phi_I(M)\le 0$. 
\end{proof}  

\begin{dfn}\label{uls} Let $(R,\m)$ be a Noetherian local ring, and $I$ be an $\m$-primary ideal. Let $s\ge 0$ be an integer. We denote by $\Ul^s_I(R)$ the full subcategory of all modules $M\in \cm^s(R)$ such that $\phi_I(M)=0$, i.e., $\lambda_R(M/IM)=e_R(I,M)$. When $s=\dim R$, we will denote this subcategory simply by $\Ul_I(R)$. When $I=\m$, $\Ul^s_\m(R)$  will be denoted by $\Ul^s(R)$. We will also denote $\Ul^{\dim R}_\m(R)$ simply by $\Ul(R)$. Note that, when $s=\dim R=1$ and $R$ is Cohen--Macaulay, $\Ul_I(R)$ is exactly the collection of all $I$-Ulrich modules as defined in \cite[Definition 4.1]{dms}.  
\end{dfn}   

\begin{rem} In terms of \cite[Definition 2.1]{agl}, a non-zero $R$-module $M$ is Ulrich if $M\in \Ul^s(R)$ for some $s\ge 0$ in our notation. When $R$ is Cohen--Macaulay, then the modules in $\Ul(R)$ are simply the maximally generated modules as studied in \cite{ul}.  
\end{rem}

\begin{cor}\label{iulrich} For each integer $s\ge 0$ and $\m$-primary ideal $I$, $(\Ul^s_I(R), {\S}|_{{\Ul}^s_I(R)})$ is an exact subcategory of $\Mod R$ (hence of $\mod R$). Hence, $\Ext^1_{\Ul^s_I(R)}(-,-): \Ul^s_I(R)^{op}\times \Ul^s_I(R)\to \mod R$ is a subfunctor of $\Ext^1_{R}(-,-): \Ul^s_I(R)^{op}\times \Ul^s_I(R)\to \mod R$. 
\end{cor} 

\begin{proof} Since it has been noticed that $(\cm^s(R),{\S}|_{\cm^s(R)})$ is an exact subcategory of $\Mod R$, by Lemma \ref{ineq} and Proposition \ref{funcspec} it follows that $(\Ul^s_I(R), {\S}|_{\Ul^s_I(R)})$ is an exact subcategory of $(\cm^s(R),{\S}|_{\cm^s(R)})$, and hence of $\Mod R$.  The subfunctor part now follows from Proposition \ref{exsub}. 
\end{proof}

Next, we record a proposition for constructing a special kind of exact substructure of an exact structure $\E$ on a category $\A$, without shrinking the category, that also comes from certain kinds of numerical functions. This will be used in the next section to recover and improve some results on module categories.  

\begin{theorem}\label{subadd} Let $({\A},{\E})$ be an exact category. Let $\phi:{\A} \to \mathbb Z$ be a function such that $\phi$ is constant on isomorphism classes of objects in $\A$, $\phi$ is additive on finite biproducts, and $\phi$ is sub-additive on kernel-cokernel pairs in $\E$ (i.e., if $\begin{tikzcd}
M \arrow[r, tail] & N \arrow[r, two heads] & L
\end{tikzcd}$ is in $\E$, then $\phi(N)\le \phi(M)+\phi(L)$). Set ${\E}^{\phi}:=\{\text{kernel-cokernel pairs in } {\E} \text{ on which } \phi \text { is additive} \}$.  Then, $({\A},{\E}^{\phi})$ is an exact subcategory of $({\A},{\E})$.
\end{theorem}  

\begin{proof}
By \cite[Definition 2.1]{Theo} and \cite[Remark 2.4]{Theo}, it is enough to show that ${\E}^{\phi}$ is closed under isomorphisms and $({\A},{\E}^{\phi})$ satisfies the axioms [E0], [E0$^{\text{op}}$], [E1$^{\text{op}}$], [E2] and [E2$^{\text{op}}$]. First, we will show that ${\E}^{\phi}$ is closed under isomorphisms. Let $\begin{tikzcd}
M \arrow[r, tail] & N \arrow[r, two heads] & L
\end{tikzcd}$ be a kernel-cokernel pair in ${\E}^{\phi}$. Also let, $M'\to N'\to L'$ be a kernel-cokernel pair in $\A$ such that it is isomorphic to $\begin{tikzcd}
M \arrow[r, tail] & N \arrow[r, two heads] & L
\end{tikzcd}$, which implies $\begin{tikzcd}
M' \arrow[r, tail] & N' \arrow[r, two heads] & L'
\end{tikzcd}$ is in $\E$ and $M\cong M'$, $N\cong N'$, $L\cong L'$. Since $\phi(N)=\phi(M)+\phi(L)$ and $\phi$ is constant on isomorphism classes of objects in $\A$, we have $\phi(N')=\phi(N)=\phi(M)+\phi(L)=\phi(M')+\phi(L')$. So, $\begin{tikzcd}
M' \arrow[r, tail] & N' \arrow[r, two heads] & L'
\end{tikzcd}$ is in ${\E}^{\phi}$. Hence ${\E}^{\phi}$ is closed under isomorphisms. Next, we will show that, $({\A},{\E}^{\phi})$ satisfies the axiom [E0]. Let $A\in{\A}$. Since $({\A},{\E})$ is an exact category, by \cite[Lemma 2.7]{Theo} we get that $\begin{tikzcd}
A \arrow[r, "1_A", tail] & A\oplus 0\cong A \arrow[r, "0", two heads] & 0
\end{tikzcd}$ is in $\E$. Since $\phi$ is additive on finite biproducts, $\phi$ is additive on the kernel-cokernel pair $\begin{tikzcd}
A \arrow[r, "1_A", tail] & A\oplus 0\cong A \arrow[r, "0", two heads] & 0
\end{tikzcd}$. Hence by definition $\begin{tikzcd}
A \arrow[r, "1_A", tail] & A\oplus 0\cong A \arrow[r, "0", two heads] & 0
\end{tikzcd}$ is in ${\E}^{\phi}$, which implies $1_A$ is an admissible monic in ${\E}^{\phi}$. So, $({\A},{\E}^{\phi})$ satisfies the axiom [E0]. Next we will show that, $({\A},{\E}^{\phi})$ satisfies the axiom [E0$^{\text{op}}$]. Let $A\in{\A}$. Since $({\A},{\E})$ is an exact category, by \cite[Lemma 2.7]{Theo} we get that $\begin{tikzcd}
0 \arrow[r, "0", tail] & A\oplus 0\cong A \arrow[r, "1_A", two heads] & A
\end{tikzcd}$ is in $\E$. Since $\phi$ is additive on finite biproducts, $\phi$ is additive on the kernel-cokernel pair $\begin{tikzcd}
0 \arrow[r, "0", tail] & A\oplus 0\cong A \arrow[r, "1_A", two heads] & A
\end{tikzcd}$. Hence by definition $\begin{tikzcd}
0 \arrow[r, "0", tail] & A\oplus 0\cong A \arrow[r, "1_A", two heads] & A
\end{tikzcd}$ is in ${\E}^{\phi}$, which implies $1_A$ is an admissible epic in ${\E}^{\phi}$. So, $({\A},{\E}^{\phi})$ satisfies the axiom [E0$^{\text{op}}$]. Next, we will show that $({\A},{\E}^{\phi})$ satisfies the axiom [E1$^{\text{op}}$]. Let $\begin{tikzcd}
B' \arrow[r, "e", two heads] & B
\end{tikzcd}$ and $\begin{tikzcd}
B \arrow[r, "p", two heads] & C
\end{tikzcd}$ be two admissible epics in ${\E}^{\phi}$. Then we will show that
$\begin{tikzcd}
B' \arrow[r, "p\circ e", two heads] & C
\end{tikzcd}$ is an admissible epic in ${\E}^{\phi}$. Let $\begin{tikzcd}
A \arrow[r, "i", tail] & B \arrow[r, "p", two heads] & C
\end{tikzcd}$ be the complete kernel-cokernel pair of $p$, and consider a pullback square as in Proposition \ref{compepic}. From Proposition \ref{compepic} we get that, $\begin{tikzcd}
A' \arrow[r, "i'", tail] & B'\arrow[r, "p\circ e", two heads] & C
\end{tikzcd}$ is a kernel-cokernel pair in $\E$, so $\phi(B')\leq\phi(C)+\phi(A')$. 

%%%%%Since $\begin{tikzcd} A \arrow[r, "i", tail] & B \arrow[r, "p", two heads] & C \end{tikzcd}$ is in ${\E}^{\phi}$, so $\phi(B)=\phi(A)+\phi(C)$. 
Now consider the pullback square of Proposition \ref{compepic}

$$\begin{tikzcd}
A' \arrow[r, "e'", two heads] \arrow[d, "i'", tail] & A \arrow[d, "i", tail] &   \\
B' \arrow[ur, phantom, "\scalebox{1.0}{$\text{PB}$}" description]\arrow[r, "e", two heads]                        & B 
\end{tikzcd}
$$

By the dual (pullback version) of \cite[Proposition 2.12(i)$\implies$(iv)]{Theo} we get a  commutative diagram in $\E$  as follows:

$$ \begin{tikzcd}
D \arrow[r, tail] \arrow[d, equal] & A' \arrow[r, "e'", two heads] \arrow[d, "i'"]                                   & A \arrow[d, "i"] \\
D \arrow[r, tail]                       & B' \arrow[ru, "\scalebox{1.0}{$\text{PB}$}", phantom] \arrow[r, "e", two heads] & B               
\end{tikzcd}
$$ 

 The top row gives $\phi(A')\leq \phi(A)+\phi(D)$. Thus $\phi(C)+\phi(A')\leq \phi(A)+\phi(C)+\phi(D)$. Since  $\begin{tikzcd}
B \arrow[r, "p", two heads] & C
\end{tikzcd}$ is an admissible epic in ${\E}^{\phi}$ and $\begin{tikzcd}
A \arrow[r, "i", tail] & B \arrow[r, "p", two heads] & C
\end{tikzcd}$ is in $\E$ and $\phi(-)$ is constant on isomorphism classes of objects, we have $\phi(A)+\phi(C)=\phi(B)$. Thus $\phi(C)+\phi(A')\leq \phi(A)+\phi(C)+\phi(D)$ implies $\phi(C)+\phi(A')\leq \phi(B)+\phi(D)$. Since $\begin{tikzcd} B'  \arrow[r, "e", two heads]  & B \end{tikzcd}$ is an admissible epic in ${\E}^{\phi}$, the bottom row of the above diagram gives $\phi(B')=\phi(B)+\phi(D)$. Hence $\phi(C)+\phi(A')\leq \phi(B')$. Thus $\begin{tikzcd}
A' \arrow[r, "i'", tail] & B'\arrow[r, "p\circ e", two heads] & C
\end{tikzcd}$ in $\E$ is actually in ${\E}^{\phi}$. This shows $({\A},{\E}^{\phi})$ satisfies the axiom [E1$^{\text{op}}$].  

\if0
Now by an argument dual to \cite[Proposition 2.12(ii)]{Theo} applied to the pullback square of Proposition \ref{compepic} we get a kernel-cokernel pair in $\E$ as follows: $\begin{tikzcd}
B' \arrow[r, tail] & B\oplus A'\arrow[r, two heads] & A
\end{tikzcd}$, which implies $\phi(B)+\phi(A')=\phi(B\oplus A')\leq\phi(A)+\phi(B')$. Hence $\phi(A)+\phi(C)+\phi(A')\leq\phi(A)+\phi(B')$, so $\phi(C)+\phi(A')\leq\phi(B')$. Thus $\phi(C)+\phi(A')=\phi(B')$, so $\begin{tikzcd}
A' \arrow[r, "i'", tail] & B'\arrow[r, "p\circ e", two heads] & C
\end{tikzcd}$ is a kernel-cokernel pair in ${\E}^{\phi}$. This implies that $\begin{tikzcd}
B' \arrow[r, "p\circ e", two heads] & C
\end{tikzcd}$ is an admissible epic in ${\E}^{\phi}$. Thus $({\A},{\E}^{\phi})$ satisfies the axiom [E1$^{\text{op}}$].
\fi

Now we will show that $({\A},{\E}^{\phi})$ satisfies the axiom [E2]. Let  $\begin{tikzcd}
A \arrow[r, "i", tail] & B
\end{tikzcd}$ be an admissible monic in ${\E}^{\phi}$ and $A\xrightarrow{f}A'$ be an arbitrary morphism in $\A$. Then we have the following pushout commutative square in $\E$:
$$\begin{tikzcd}
A \arrow[r, "i", tail] \arrow[d, "f"] & B\arrow[d, "f'"]\\
A'\arrow[ur, phantom, "\scalebox{1.0}{$\text{PO}$}" description] \arrow[r, "i'", tail]              & B'
\end{tikzcd}$$
Then by \cite[Proposition 2.12(ii)]{Theo} we have the following kernel-cokernel pair in $\E$: $\begin{tikzcd}
A \arrow[r, tail] & B\oplus A' \arrow[r, two heads] & B'
\end{tikzcd}$, so $\phi(B)+\phi(A')=\phi(B\oplus A')\leq\phi(A)+\phi(B')$. Also, by \cite[Proposition 2.12(iv)]{Theo} we have the following commutative diagram with rows being kernel-cokernel pairs in $\E$:
$$\begin{tikzcd}
A \arrow[r, "i", tail] \arrow[d, "f"] & B\arrow[d, "f'"]\arrow[r, "p", two heads] & C \arrow[d,equal]\\
A'\arrow[ur, phantom, "\scalebox{1.0}{$\text{PO}$}" description] \arrow[r, "i'", tail]              & B'  \arrow[r, "p'", two heads]               & C           
\end{tikzcd}$$
Since $\begin{tikzcd}
A \arrow[r, "i", tail] & B
\end{tikzcd}$ is an admissible monic in ${\E}^{\phi}$, we have $\phi(B)=\phi(A)+\phi(C)$. This implies $\phi(A)+\phi(C)+\phi(A')\leq\phi(A)+\phi(B')$, so $\phi(C)+\phi(A')\leq\phi(B')$. Also, $\begin{tikzcd}
A' \arrow[r, "i'", tail] & B' \arrow[r, "p'", two heads] & C
\end{tikzcd}$ is a kernel-cokernel pair in $\E$, so $\phi(B')\leq\phi(C)+\phi(A')$. Thus $\phi(C)+\phi(A')=\phi(B')$. Hence $\begin{tikzcd}
A' \arrow[r, "i'", tail] & B' \arrow[r, "p'", two heads] & C
\end{tikzcd}$ is a kernel-cokernel pair in ${\E}^{\phi}$, so $\begin{tikzcd}
A' \arrow[r, "i'", tail] & B'
\end{tikzcd}$ is an admissible monic in ${\E}^{\phi}$. Thus $({\A},{\E}^{\phi})$ satisfies the axiom [E2]. A dual argument will show that $({\A},{\E}^{\phi})$ satisfies the axiom [E2$^{\text{op}}$]. Hence $({\A},{\E}^{\phi})$ is an exact subcategory of $({\A},{\E})$.  
\end{proof}

%\subsection{Identifying Subfunctors of Ext in Mod R} 

Moving forward, we record some applications of Theorem \ref{subadd}.

\begin{chunk}\label{subst} Let $\X$ be a subcategory of $\Mod (R)$ such that $({\X},{\S}|_{\X})$ is an exact subcategory of $\Mod R$. In this case, for $M,N\in {\X}$, by $\Ext^1_{\X}(M,N)$ we will mean $\Ext_{{\S}|_{\X}}(M,N)$, i.e., $\Ext^1_{\X}(M,N)=\{[\sigma]: \text{the middle object of }\sigma \text{ is in } {\X}\}$ (Note that, it does not matter whether the equivalence class is taken in ${\S}|_{\X}$ or $\S$ by Remark \ref{equi}). Note that, $\Ext^1_{\X}(-,-): {\X}^{op}\times {\X}\to \Mod R$ is a subfunctor of $\Ext^1_R(-,-): {\X}^{op}\times {\X}\to \Mod R$ by Proposition \ref{exsub}. For example, we can apply this discussion to ${\X}=\Ul^s_I(R)$ due to Corollary \ref{iulrich}.   
\end{chunk}    

Note that, if $\X$ is closed under taking extensions in $\Mod R$, then $({\X},{\S}|_{\X})$ is exact subcategory of $\Mod R$ by Proposition \ref{extclosed}. In this case, $\Ext^1_{\X}(M,N)=\Ext^1_{R}(M,N)$ for all $M,N\in {\X}$.  

\begin{dfn}\label{numdef} Let $\X$ be a subcategory of $\Mod R$ such that $({\X},{\S}|_{\X})$ is an exact subcategory of $\Mod R$, and $\phi:{\X}\to \mathbb Z$ be a function satisfying the hypothesis of Theorem \ref{subadd}, i.e., $\phi$ is constant on isomorphism classes of modules in $\X$, $\phi$ is additive on finite direct sums, and sub-additive on short exact sequences of modules in $\X$. Then, in the notation of Theorem \ref{subadd}, $({\X},{\S}|_{{\X}^{\phi}})$ is an exact subcategory of $({\X},{\S}|_{\X})$, hence an exact subcategory of $\Mod R$. For $M,N\in{\X}$, define $\Ext_{\X}^1(M,N)^{\phi}:=\Ext_{{\S}|_{{\X}^{\phi}}}(M,N)=\{[\alpha]\in \Ext_{\X}^1(M,N)\text{ }:\text{ }\phi\text{ is additive on }\alpha\}$. Again, we notice that if ${\X}\subseteq \Mod R$ is closed under taking extensions, then $\Ext_{\X}^1(M,N)^{\phi} =\{[\alpha]\in \Ext^1_{R}(M,N)\text{ }:\text{ }\phi\text{ is additive on }\alpha\}$, and in this case, we denote it just by $\Ext^1_R(M,N)^{\phi}$.
\end{dfn}   

\begin{chunk}\label{chnum}  Note that, $\Ext_{\X}^1(-,-)^{\phi}:{\X}^{op}\times {\X}\to \Mod R$ is a subfunctor of $\Ext_{\X}^1(-,-):{\X}^{op}\times {\X}\to \Mod R$, which in turn is a subfunctor of $\Ext^1_{R}(-,-):{\X}^{op}\times {\X}\to \Mod R$ (Proposition \ref{exsub}) and we have corresponding commutative diagram of long exact sequences by Corollary \ref{long}. 
\end{chunk}   

With the extension-closed subcategory ${\X}=\mod R$, where $(R,\m,k)$ is a Noetherian local ring, and the subadditive function $\phi(-)=\mu(-):\mod R\to \mathbb Z$ being the number of generators function in Definition \ref{numdef}, we see that  $\Ext^1_R(M,N)^{\mu}=\langle M, N\rangle$ in the notation of \cite[Definition 1.35]{thesis}. So by our discussion, Theorem \ref{subadd}, Proposition \ref{exsub} and Corollary \ref{long} recovers \cite[Corollary 1.36, Proposition 1.37, Proposition 1.38]{thesis}. 
\\

Before we apply our discussion to more certain special cases, we record one preliminary lemma.

\begin{lem}\label{31} Let $({\X},{\E})$ be an exact subcategory of $\Mod R$. Let $G:{\X} \to \operatorname{fl}(R)$ be an additive half-exact functor, where $\operatorname{fl}(R)$ denotes the full subcategory of $\mod R$ consisting of finite length modules. Then, $\phi(-):=\lambda_R(G(-)): {\X} \to \mathbb Z$ satisfies the hypothesis of Theorem \ref{subadd}. Moreover, given a short exact sequence $\sigma$ with objects in $\X$, $G(\sigma)$ is a short exact sequence of modules if and only if $\phi$ is additive on $\sigma$. 
\end{lem}   

\begin{proof} We will only prove the statement when $G$ is a covariant functor, since the proof for a contravariant functor is similar. 
\\
Since $G$ is a functor and $\lambda_R$ is constant on isomorphism classes of objects in $\Mod R$, $\phi$ is constant on isomorphism classes of objects in $\X$. Let $A,B\in{\X}$. Then $\phi(A\oplus B)=\lambda_R(G(A\oplus B))=\lambda_R(G(A)\oplus G(B))=\lambda_R(G(A))+\lambda_R(G(B))=\phi(A)+\phi(B)$, so $\phi$ is additive on finite biproducts. Next let, $\sigma:0\to M\xrightarrow{f} N\xrightarrow{g} L\to 0$ be a short exact sequence with objects in $\X$. Then $G(\sigma)$ is the exact sequence $G(\sigma):0\to K\to G(M)\xrightarrow{G(f)} G(N)\xrightarrow{G(g)} G(L)\to P\to 0$, where $K=\operatorname{Ker}(G(f))$ and $P=\operatorname{Coker}(G(g))$. Then we have $\phi(M)+\phi(L)-\phi(N)=\lambda_R(G(M))+\lambda_R(G(L))-\lambda_R(G(N))=\lambda_R(K)+\lambda_R(P)\geq 0$, so $\phi(M)+\phi(L)\ge\phi(N)$. Hence $\phi$ is subadditive on a short exact sequence with objects in $\X$. Thus $\phi$ satisfies the hypothesis of Proposition \ref{subadd}.
\\
Next let, $\sigma:0\to M\xrightarrow{f} N\xrightarrow{g} L\to 0$ be a short exact sequence with objects in $\X$. Then $G(\sigma)$ is short exact if and only if $K=\operatorname{Ker}(G(f))=0$ and $P=\operatorname{Coker}(G(g))=0$ if and only if $\lambda_R(K)=\lambda_R(P)=0$ if and only if $\phi(M)+\phi(L)-\phi(N)=\lambda_R(G(M))+\lambda_R(G(L))-\lambda_R(G(N))=\lambda_R(K)+\lambda_R(P)= 0$ if and only if $\phi(M)+\phi(L)=\phi(N)$ if and only if $\phi$ is additive on $\sigma$.  
\end{proof} 

\begin{cor}\label{half} Let $({\X},{\E})$ be an exact subcategory of $\Mod R$. Let $G:{\X} \to \operatorname{fl}(R)$ be an additive half-exact functor, where $\operatorname{fl}(R)$ denotes the full subcategory of $\mod R$ consisting of finite length modules. Set ${\E}^G:=\{\sigma \in {\E}: G(\sigma) \text{ is a short exact sequence of modules}\}$. Then, $({\X}, {\E}^G)$ is an exact subcategory of $\Mod R$. 
\end{cor}  

\begin{proof} From Lemma \ref{31} it follows that, ${\E}^G={\E}^{\phi}$, where $\phi(-):=\lambda_R(G(-)): {\X} \to \mathbb Z$. Hence the claim now follows from Theorem \ref{subadd}.   
\end{proof}

Now again, let ${\X}=\mod R$, where $R$ is Noetherian. Basic examples of additive half-exact functors $G:{\X} \to \operatorname{fl}(R)$ are $G(-):=\Ext^i_R(-,X),\Ext^i_R(X,-), \tor^R_i(X,-)$, where  $X$ is a fixed $R$-module of finite length and $i\ge 0$ is an integer.

Given a collection ${\C}\subseteq \mod R$, in \cite{aus}, the authors considered the following:
\begin{align*}
F_{\C}(C,A)&:=\{\sigma: 0\to A \to B \to C\to 0\text{ } |\text{ }\Hom_R(X,\sigma) \text{ is short exact for all } X \in {\C}\}\\
F^{\C}(C,A)&:=\{\sigma: 0\to A \to B \to C\to 0\text{ }|\text{ }\Hom_R(\sigma,X) \text{ is short exact for all } X \in {\C}\}
\end{align*}
and showed in \cite[Proposition 1.7]{aus} that these define subfunctors of $\Ext^1_R(-,-):\mod R^{op}\times \mod R \to \textbf{Ab}$.  Using Corollary \ref{half}, we recover this result when ${\C}\subseteq \operatorname{fl}(R)$ as follows:

\begin{cor}\label{half2} Let $R$ be a commutative Noetherian ring. Let ${\C}\subseteq \operatorname{fl}(R)$. Then, $F_{\C}(-,-)$ and $F^{\C}(-,-)$ are subfunctors of $\Ext^1_R(-,-):\mod R^{op}\times \mod R \to \mod R$.
\end{cor}  

\begin{proof} We only prove the case of $F_{\C}(-,-)$, since the proof of $F^{\C}(-,-)$ is similar. 

Let ${\S}^{\C}$ be the collection of all short exact sequences $\sigma$ of finitely generated $R$-modules such that $\Hom_R(X,\sigma)$ is short exact for all $X \in {\C}$. We first prove that $(\mod R,{\S}^{\C})$ is an exact subcategory of $\mod R$. Since ${\S}^{\C}=\cap_{C \in {\C}} {\S}^{C}$, it is enough to prove that for each $C\in {\C}$, ${\S}^C$ gives an exact substructure on $\mod R$ (intersection of exact substructures is again an exact substructure by Lemma \ref{inter}). Now note that, ${\S}^C=\{\sigma: G(\sigma) \text{ is short exact}\}$, where $G:\mod R \to \operatorname{fl}(R)$ is given by $G(-):=\Hom_R(C,-)$, so by Corollary \ref{half} we get that $(\mod R,{\S}^C)$ is an exact subcategory of $\mod R$. Since $\Ext_{{\S}^{\C}}(-,-)=F_{\C}(-,-)$, by Proposition \ref{exsub} we are done.
\end{proof}  

Next, we recover and improve upon  \cite[Theorem 3.11, Theorem 3.13]{tony}. 

For a Noetherian local ring $(R,\m)$, let $\MD(R)$ denote the full subcategory of $\mod R$ consisting of all modules $M$ such that either $M=0$ or $\dim M=\dim R$. We recall that, if $I$ is an $\m$-primary ideal, then $e_R(I,-)$ is additive on $\MD(R)$. We also recall that, a finitely generated module $M$ satisfy Serre's condition $(S_n)$ if $\depth M_{\p}\ge \inf\{n,\dim R_{\p}\}$ for all $\p\in\operatorname{Spec} R$. The full subcategory of $\mod R$ consisting of all modules satisfying $(S_n)$ is denoted by $S_n(R)$. Note that, $S_n(R)$ is closed under extensions in $\mod R$.   

\begin{lem}\label{md} Let $R$ be an equidimensional local ring. Let $S_1(R)$ be the collection of all modules in $\mod R$ satisfying Serre's condition $(S_1)$. Then, $S_1(R)\subseteq \MD(R)$.
\end{lem}

\begin{proof}  The hypothesis on $R$ implies that, $\dim(R/\p)=\dim R$ for all $\p \in \Min(R)$. So, if  $0\ne M$ satisfy $(S_1)$, then there exists $\p \in  \Ass(M)\subseteq \Min(R)$. Hence $\dim(M)\ge \dim(R/\p)=\dim R$, and we are done. 
\end{proof} 

Now inspecting the proof of \cite[Proposition 17]{tonyh}, the only place, where $R$ is Cohen--Macaulay and $M$ is maximal Cohen--Macaulay are required, is to ensure that $e_R(I,-)$ is additive on $0\to \syz_R M \to F \to M \to 0$ and that $M,\syz_R M \in \MD(R)$. But if we assume $R$ is equidimensional and satisfy $(S_1)$, and $M\in S_1(R)\subseteq \MD(R)$, then since $\syz_R M\in S_1(R)\subseteq \MD(R)$ (Lemma \ref{md}) and $e_R(I,-)$ is additive on short exact sequence of modules of the same dimension (\cite[Corollary 4.7.7]{bh}), hence we get the following lemma by following the same proof as in \cite[Proposition 17]{tonyh}: 

\begin{lem}\label{quasi} Let $(R,\m)$ be an equidimensional local ring satisfying $(S_1)$. Let $M\in S_1(R)$. Let $d=\dim R\ge 1$, and $I$ be an $\m$-primary ideal. Then, the function $n \mapsto \lambda(\tor^R_1(M,R/I^{n+1}))$ is given by a polynomial of degree $\le d-1$ for $n\gg 0$. So, in particular, the limit $\lim_{n\to \infty}\dfrac{\lambda(\tor^R_1(M,R/I^{n+1}))}{n^{d-1}}$ exists.  
\end{lem}

If $R$ and $M$ are as in Lemma \ref{quasi}, then let us denote $e^T_I(M):=(d-1)!\lim_{n\to \infty}\dfrac{\lambda(\tor^R_1(M,R/I^{n+1}))}{n^{d-1}}$. (What \cite{tony} denotes by $e^T_R(M)$ is exactly same as $e^T_{\m}(M)$ in our notation). 

The following theorem  recovers and improves upon \cite[Theorem 3.11, Theorem 3.13]{tony} (since a local Cohen--Macaulay ring $R$ is equidimensional, satisfies $(S_1)$ and $\cm(R)\subseteq S_1(R)$). 

\begin{theorem}\label{thmTony} Let $(R,\m)$ be an equidimensional local ring satisfying $(S_1)$. Let ${\X}=S_1(R)$ be the subcategory of $\mod R$ consisting of all modules satisfying $(S_1)$. Let $I$ be an $\m$-primary ideal. Then, $\Ext^1_R(-,-)^{e_I^T}:S_1(R)^{op}\times S_1(R)\to \mod R$  is a subfunctor of $\Ext^1_R(-,-):S_1(R)^{op}\times S_1(R)\to \mod R$.  
\end{theorem} 

\begin{proof} Since $S_1(R)$ is an extension closed subcategory of $\mod R$, we have $(S_1(R),{\S}|_{S_1(R)})$ is an exact subcategory of $\mod R$. Since for each $n$, the function $\mod R \to \mathbb Z$ given by $M\mapsto\lambda(\tor^R_1(M,R/I^{n+1}))$ is subadditive on short exact sequences of $\mod R$, we have $e^T_I:S_1(R)\to \mathbb Z$ is also subadditive on short exact sequences of $S_1(R)$. Thus $(S_1(R),{\S}|_{S_1(R)}^{e^T_I})$ is an exact subcategory of $(S_1(R),{\S}|_{S_1(R)})$ by Theorem \ref{subadd}. Hence we are done by Proposition \ref{exsub}.  
\end{proof}   

\section[Special subfunctors of $\Ext^1$ and applications]{Special subfunctors of $\Ext^1$ and applications}\label{sec5}  

Throughout this section, $R$ will denote a Noetherian local ring with unique maximal ideal $\m$ and residue field $k$. We will denote by $\mu_R(-)$ the minimal number of generators function for finitely generated modules, i.e., $\mu_R(M)=\lambda_R(M\otimes_R R/\m)$ for all $M\in \mod R$. We drop the subscript $R$ when the ring in question is clear. We shall study properties and applications of a number of subfunctors of $\Ext^1$, whose existence follow from results in previous sections.  

\subsection[Some computations and applications of the subfunctor $\Ext^1_R(-,-)^{\mu}$]{Some computations and applications of the subfunctor $\Ext^1_R(-,-)^{\mu}$}   
\text{ }
\\
We start this subsection with a characterisation of regularity among Cohen--Macaulay rings of positive dimension $d$ in terms of vanishing of certain $\Ext^1_R(-,-)^{\mu}$ (see Definition \ref{numdef} and the discussion after \ref{chnum} for notation).

\begin{subthm}\label{regdchar} Let $R$ be a local Cohen--Macaulay ring of dimension $d\ge 1$. Then, $R$ is regular if and only if  $\Ext^1_R(\syz^{d-1}_R k, N)^{\mu}=0$ for some finitely generated non-zero $R$-module $N$ of finite injective dimension.  
\end{subthm}

When $R$ is regular, we show $\Ext^1_R(\syz^{d-1}_R k, R)^{\mu}=0$, which shows one direction of Theorem \ref{regdchar}. For this, we first need two preliminary lemmas.

\begin{sublemma}\label{one} Let $Y$ be a submodule of an $R$-module $X$ such that $\m X\subseteq Y$. If $X$ is cyclic and $Y\ne X$, then $Y=\m X$.   
\end{sublemma}

\begin{proof} We prove that if $Y\ne \m X$, then $Y=X$. If $Y\ne \m X$, then pick $y\in Y\setminus \m X\subseteq X\setminus \m X$. Then, $y$ is part of a minimal system of generators of $X$, but $X$ is cyclic, so $X=Ry$. Also, $Ry\subseteq Y\subseteq X$. Hence, $Y=X$.  
\end{proof}

\begin{sublemma}\label{lemM} Let $M$ be a finitely generated $R$-module with first Betti number $1$.  Then, $\Ext^1_R(M,\syz_R M)^{\mu}=\m \Ext^1_R(M,\syz_R M)$.  
\end{sublemma}   

\begin{proof} $\Ext^1_R(M,\syz_R M)^{\mu}$ is a submodule of $\Ext^1_R(M,\syz_R M)$ such that $\m \Ext^1_R(M,\syz_R M)\subseteq \Ext^1_R(M,\syz_R M)^{\mu}$ by Lemma \ref{jane}.  By hypothesis, we have $\syz_R M\cong R/I$ for some ideal $I\ne R$, and we have an exact sequence $ 0\to R/I \to F \to M \to 0$, for some free $R$-module $F$. So, in particular, $\mu(F)=\mu(M)\ne \mu(R/I)+\mu(M)$. Hence this exact sequence is not $\mu$-additive. So, $\Ext^1_R(M,\syz_R M)^{\mu}\ne \Ext^1_R(M,\syz_R M)$.  Also, applying $\Hom_R(-,R/I)$ to the exact sequence  $0\to R/I \to F \to M \to 0$ we get an exact sequence $0\to \Hom_R(M,R/I)\to \Hom_R(F,R/I)\to \Hom_R(R/I,R/I)\cong R/I \to \Ext^1_R(M,R/I)\cong \Ext^1_R(M,\syz_RM)\to 0$. Hence $\Ext^1_R(M,\syz_R M)$ is a cyclic module. Now applying Lemma \ref{one} to $Y=\Ext^1_R(M,\syz_R M)^{\mu}$ and $X=\Ext^1_R(M,\syz_R M)$, we get $\Ext^1_R(M,\syz_R M)^{\mu}=\m\Ext^1_R(M,\syz_R M)$.  
\end{proof}

Now, one direction of Theorem \ref{regdchar} is the following.

\begin{subcor}\label{regu} Let $(R,\m,k)$ be a regular local ring of dimension $d\ge 1$. Then, $\Ext^1_R(\syz^{d-1}_R k, R)^{\mu}=0$. 
\end{subcor}

\begin{proof} We know $\syz_R \syz^{d-1}_R k\cong R$ by looking at the Koszul complex of $R/\m$, so the first Betti number of $\syz^{d-1}_R k$ is $1$. Hence by Lemma \ref{lemM} we get that, $\Ext^1_R(\syz^{d-1}_Rk,R)^{\mu}=\m\Ext^1_R(\syz^{d-1}_Rk,R)\cong \m\Ext^d_R(k,R)\cong \m\cdot k=0$.  
\end{proof}   

The other direction of Theorem \ref{regdchar} requires more work. First, we begin with the following setup.

\begin{subchunk}\label{eR}

Let $(R,\m,k)$ be a local Cohen--Macaulay ring of dimension $d\ge 1$ admitting a canonical module $\omega_R$. Then, $\Ext^1_R(\syz^{d-1}_Rk,\omega_R)\cong k$. Hence for any two non-split exact sequences $\alpha, \beta \in \Ext^1_R(\syz^{d-1}_Rk,\omega_R)$, we have $[\beta]=r[\alpha] \in \Ext^1_R(\syz^{d-1}_Rk,\omega_R)$ for some unit $r\in R$. So, we have the following pushout diagram: 

$$\begin{tikzcd}
\alpha: 0 \arrow[r] & \omega_R \arrow[r] \arrow[d, "\cdot r"'] & X_{\alpha} \arrow[r] \arrow[d] & \syz^{d-1}_Rk \arrow[r] \arrow[d, equal] & 0 \\
\beta: 0 \arrow[r] & \omega_R \arrow[r]                       & X_{\beta} \arrow[r]           & \syz^{d-1}_Rk \arrow[r]           & 0
\end{tikzcd}$$ 
Since $r\in R$ is a unit, we have $\omega_R \xrightarrow{\cdot r} \omega_R$ is an isomorphism. Hence by five lemma, $X_{\alpha}\cong X_{\beta}$. So, there exists a unique module (up to isomorphism), call it $E^R$, such that the middle term of every non-split exact sequence $0\to \omega_R \to X \to \syz^{d-1} k \to 0$ is isomorphic to $E^R$. Since $\syz^{d-1}_R k$ has co-depth $1$, by \cite[Proposition 11.21]{lw} we get that $0\to \omega_R \to E^R \to \syz^{d-1}_Rk \to 0$ is the minimal MCM approximation of $\syz^{d-1}_R k$.
\end{subchunk}     
Moving forward, we denote $\Hom_R(-,\omega_R)$ by $(-)^{\dagger}$.

We need to collect some properties of $E^R$ to prove the other direction of Theorem \ref{regdchar}. 

\begin{sublemma}\label{3.12} Let $(R,\m,k)$ be a local Cohen--Macaulay ring of dimension $1$ admitting a canonical module $\omega_R$. Then, $E^R\cong \m^{\dagger}$. 
\end{sublemma}   

\begin{proof} Dualizing $0\to \m \to R \to k \to 0$ by $\omega_R$ we get, $0\to \omega_R \to \m^{\dagger} \to \Ext^1_R(k,\omega_R)\cong k \to 0$. So, we have the exact sequence $0\to \omega_R \to \m^{\dagger} \to k \to 0$, which is clearly non-split, since $\m^{\dagger}$ has positive depth. But $k\oplus \omega_R$ has depth $0$. Thus $E^R\cong \m^{\dagger}$.  
\end{proof}

We record the following lemma, which will be used to deduce further properties of the module $E^R$.  

\begin{sublemma}\label{3.10} Let $(R,\m,k)$ be a local ring, and $I$ an ideal of $R$, which is not principal. Then, for every $x\in I\setminus \m I$, we have $(x\m:_{\m} I)=(x\m:_R I)=((x) :_R I)$.  
\end{sublemma}   

\begin{proof} Since trivially $(x\m:_{\m}I)\subseteq (x\m:_R I)\subseteq ((x) :_R I)$ always holds, it is enough to prove the inclusion $((x) :_R I)\subseteq (x\m:_{\m} I)$. So, let $y\in ((x) :_R I)$, which means $yI \subseteq (x)$. If $y$ were a unit, then we would have $I \subseteq y^{-1}(x)=(x)\subseteq I$, implying $I=(x)$, contradicting our assumption that $I$ is not principal. Thus, we must have $y\in \m$. Now pick an arbitrary element $r\in I$. Then $yr\in y I\subseteq (x)$, so $yr=xs$ for some $s\in R$. If $s\notin \m$, then $x=s^{-1}yr\in \m I$, which is a contradiction. Hence,  $s\in \m$. So, $yr=xs\in x\m$. Since $r\in I$ was arbitrary, we get $yI \subseteq x\m$. Hence $y\in (x\m:_{\m} I)$.      
\end{proof} 

In the following, $\r(-)$ will denote the type of a module, i.e., $\r(M)=\dim_k \Ext^{\depth_R M}_R(k,M)$. 

\begin{sublemma}\label{cano} Let $(R,\m,k)$ be a non-regular complete local Cohen--Macaulay ring of dimension $d\ge 1$, with canonical module $\omega_R$. Then, we have $\mu(E^R)=\operatorname{r}(R)+\mu(\syz^{d-1}_Rk)$.  
\end{sublemma}   

\begin{proof}  We prove this by induction on $d$. First, let $d=1$. By Lemma \ref{3.12} we have $E^R\cong\m^{\dagger}$. Pick $x\in \m\setminus \m^2$ to be $R$-regular. Then $\mu_R(E^R)=\mu_{R/xR}(E^R/xE^R)$. Now, $E^R/x E^R\cong \m^{\dagger}/x\m^{\dagger}\cong \Hom_{R/xR}(\m/x\m,\omega_R/x\omega_R)\cong (\m/x\m)^{\dagger}$ (where the last two isomorphisms follow from \cite[Proposition 3.3.3(a), Theorem 3.3.5(a)]{bh}). Now $ \Hom_{R/xR}((\m/x\m)^{\dagger},k)\cong \Hom_{R/xR}(k^{\dagger}, \m/x\m)\cong \Hom_{R/xR}(k, \m/x\m)\cong (x\m:_{\m}\m)/x\m=((x):_R\m)/x\m$ (where the first  isomorphism holds because $R/xR$ is Artinian, so one can invoke \cite[Lemma 3.14, Remark 3.15]{clus}, and the last equality follows from Lemma \ref{3.10} since $R$ is not regular). Hence $\mu_R(E^R)=\mu_{R/xR}(E^R/xE^R)=\dim_k \Hom_{R/xR}((\m/x\m)^{\dagger},k)=\dim_k ((x):_R\m)/x\m=\dim_k \left((xR:_R \m)/xR \right)+ \dim_k (xR/x\m)=\dim_k \Soc(R/xR)+\dim_k (xR/x\m)=\text{r}(R)+\dim_k (xR/x\m)$ (since $xR/x\m$ is annihilated by $\m$, hence it is a $k$-vector space). Since $xR/x\m$ is generated by $x$ and it is a $k$-vector space, we have $xR/x\m\cong k$. Hence $\dim_k (xR/x\m)=1$, so $\mu(E^R)=\text{r}(R)+1=\text{r}(R)+\mu(k)$. This concludes the $d=1$ case. 

Now let, $\dim R=d>1$ and suppose the claim has been proved for all non-regular local Cohen--Macaulay rings of dimension $1,...,d-1$ admitting a canonical module. Since $\syz^{d-1}_R k$ has co-depth $1$, by \cite[Proposition 11.21]{lw}(see \ref{eR}) we get that $0\to \omega_R \to E^R \to \syz^{d-1}_Rk \to 0$ is the minimal MCM approximation of $\syz^{d-1}_R k$. Now pick an $R$-regular element $x\in \m\setminus \m^2$ (which is also $\syz^{d-1}_R k$-regular, since $d-1>0$). By \cite[Corollary 2.5]{yo}, we have $0\to \omega_{R/xR}\to E^R/xE^R\to \syz^{d-1}_R k/x(\syz^{d-1}_Rk)\to 0$ is the minimal MCM approximation of $\syz^{d-1}_R k/x(\syz^{d-1}_Rk) \cong \syz^{d-1}_{R/xR}k \oplus \syz^{d-2}_{R/xR} k$ over $R/xR$ (the isomorphism follows from \cite[Corollary 5.3]{summ}). Since $0\to 0 \to \syz^{d-1}_{R/xR}k\to \syz^{d-1}_{R/xR}k \to 0$ is the minimal MCM approximation of $\syz^{d-1}_{R/xR}k$ over $R/xR$ and $0\to \omega_{R/xR}\to E^{R/xR} \to \syz^{d-2}_{R/xR} k \to 0$ is the minimal MCM approximation of $\syz^{d-2}_{R/xR} k$ over $R/xR$ (by \ref{eR}, \cite[Proposition 11.21]{lw}), their direct sum $0\to \omega_{R/xR}\to E^{R/xR} \oplus \syz^{d-1}_{R/xR}k \to \syz^{d-2}_{R/xR} k \oplus \syz^{d-1}_{R/xR}k \to 0$ is the minimal MCM approximation of $\syz^{d-2}_{R/xR} k \oplus \syz^{d-1}_{R/xR}k\cong \syz^{d-1}_R k/x(\syz^{d-1}_Rk)$ over $R/xR$ (direct sum preserves minimal MCM approximation by \cite[Theorem 1.4]{yo}). By uniqueness of minimal MCM approximation, we have $E^R/xE^R\cong E^{R/xR} \oplus \syz^{d-1}_{R/xR}k$. So, $\mu_R(E^R)=\mu_{R/xR}(E^R/xE^R)=\mu_{R/xR}(E^{R/xR})+\mu_{R/xR}(\syz^{d-1}_{R/xR}k)$. Since $R$ is not regular, $R/xR$ is not regular. Hence by induction hypothesis, we have $\mu_{R/xR}(E^{R/xR})=\text{r}(R/xR)+\mu_{R/xR}(\syz^{d-2}_{R/xR} k)=\text{r}(R)+\mu_{R/xR}(\syz^{d-2}_{R/xR} k)$. So, $\mu_R(E^R)=\text{r}(R)+\mu_{R/xR}(\syz^{d-2}_{R/xR} k)+\mu_{R/xR}(\syz^{d-1}_{R/xR} k)=\text{r}(R)+\mu_{R/xR}(\syz^{d-2}_{R/xR} k \oplus \syz^{d-1}_{R/xR}k)=\text{r}(R)+\mu_{R/xR}\left( \syz^{d-1}_R k/x(\syz^{d-1}_Rk) \right)=\text{r}(R)+\mu_R(\syz^{d-1}_R k)$. This finishes the inductive step, and hence the proof.  
\end{proof}  

Since $\mu(\omega_R)=\text{r}(R)$, Lemma \ref{cano} says that the sequence $0\to \omega_R\to E^R\to \Omega_R^{d-1}k\to 0$ is $\mu$-additive. As a consequence of this, we get the following:  

\begin{subproposition}\label{injd}
Let $(R,\m,k)$ be a non-regular local Cohen--Macaulay ring of dimension $d\ge 1$, and $N$ be a finitely generated non-zero $R$-module with finite injective dimension. Then $\Ext^1_R(\Omega_R^{d-1}k,N)^{\mu}=\Ext_{R}^1(\Omega_R^{d-1}k,N)\ne 0$. 
\end{subproposition}  

\begin{proof} That $\Ext_{R}^1(\Omega_R^{d-1}k,N)\cong \Ext^d_R(k,N)\ne 0$ follows from \cite[Exercise 3.1.24]{bh}. So, it is enough to prove that $\Ext^1_R(\Omega_R^{d-1}k,N)^{\mu}=\Ext_{R}^1(\Omega_R^{d-1}k,N)$. 

We first consider the case where $R$ is complete, hence admitting a canonical module $\omega_R$. Consider the $\mu$-additive (by Lemma \ref{cano}) exact sequence $\sigma:0\to \omega_R\to E^R\to \Omega_R^{d-1}k\to 0$. Since $\sigma\in {\S}|_{\mod R}^{\mu}$, by applying $\Hom_R(-,N)$ to $\sigma$, we get the following part of a commutative diagram of exact sequences by Corollary \ref{long}:    

$$\begin{tikzcd}[sep=1.8em, font=\small]
{\Hom_R(\omega_R,N)} \arrow[r]                 & {\Ext_{R}^1(\Omega_R^{d-1}k,N)} \arrow[r]                 & {\Ext_{R}^1(E^R,N)}                 \\
{\Hom_R(\omega_R,N)} \arrow[r] \arrow[u, equal] & {\Ext_{R}^1(\Omega_R^{d-1}k,N)^{\mu}} \arrow[r] \arrow[u, hook] & {\Ext_{R}^1(E^R,N)^{\mu}} \arrow[u, hook]
\end{tikzcd}$$  
Since $E^R$ is maximal Cohen--Macaulay, we have $\Ext_{R}^1(E^R,N)=0$ by \cite[Exercise 3.1.24]{bh}. So, $\Ext_{R}^1(E^R,N)^{\mu}=0$ as well. Hence we get the following commutative diagram:
$$\begin{tikzcd}[sep=1.8em, font=\small]
{\Hom_R(\omega_R,N)} \arrow[r, "f"]                 & {\Ext_{R}^1(\Omega_R^{d-1}k,N)} \arrow[r]                 & 0                 \\
{\Hom_R(\omega_R,N)} \arrow[r,"g"] \arrow[u, equal] & {\Ext_{R}^1(\Omega_R^{d-1}k,N)^{\mu}} \arrow[r] \arrow[u, hook, "h"] & 0 
\end{tikzcd}$$ 
Thus $h\circ g=f$ is surjective, so $h$ is surjective. Since $h$ is the natural inclusion map, we have $\Ext^1_R(\Omega_R^{d-1}k,N)^{\mu}=\Ext^1_R(\Omega_R^{d-1}k,N)$.  

Now we consider the general case. Since $\Ext^1_R(\Omega_R^{d-1}k,N)^{\mu}\subseteq \Ext_{R}^1(\Omega_R^{d-1}k,N)$, it is enough to prove the other inclusion. So let, $\sigma: 0\to N\to X\to \syz^{d-1}_R k\to 0 $ be a short exact sequence. We need to show that, $\sigma$ is $\mu$-additive. Now consider the completion $\widehat \sigma: 0\to \widehat N\to \widehat X\to \widehat{\syz^{d-1}_R k}\cong \syz^{d-1}_{\widehat R} k \to 0$. Since $\widehat R$ is non-regular, Cohen--Macaulay of dimension $d$ and admits a canonical module and $\widehat N \in \mod \widehat R$ has finite injective dimension over $\widehat R$, by the first part of the proof we get $ \Ext^1_{\widehat R}(\syz^{d-1}_{\widehat R} k,\widehat N)=\Ext^1_{\widehat R}(\syz^{d-1}_{\widehat R} k,\widehat N)^{\mu}$. Thus $[\widehat \sigma ]\in \Ext^1_{\widehat R}(\syz^{d-1}_{\widehat R} k,\widehat N)=\Ext^1_{\widehat R}(\syz^{d-1}_{\widehat R} k,\widehat N)^{\mu}$. Hence $\widehat \sigma$ is $\mu$-additive. So, $\mu_{\widehat R}(\widehat N)+\mu_{\widehat R}( \syz^{d-1}_{\widehat R} k)=\mu_{\widehat R}(\widehat X)$. Since number of generators does not change under completion, we get $\mu_R(N)+\mu_R(\syz^{d-1}_R k)=\mu_R(X)$. Thus $\sigma$ is $\mu$-additive, which is what we wanted to prove. 
\end{proof}

\begin{proof}{\textit{of Theorem \ref{regdchar}}}: Follows by combining Corollary \ref{regu} and Proposition \ref{injd}.  
\end{proof}   

For an arbitrary local ring of positive depth, we give a characterization of the ring being regular in terms of vanishing of $\Ext^1_R(k,M)^{\mu}$ for some $M\in \mod R$ of finite projective dimension. For this, we first recall the definitions of weakly $\m$-full and Burch submodules of a module from \cite[Definition 3.1, 4.1]{dk} and subsequently, we relate that property to the vanishing of certain $\Ext^1_R(k,-)^{\mu}$. 

\begin{subdfn}
Let $(R,\m,k)$ be a local ring and let, $N$ be an $R$-submodule of a finitely generated $R$-module $M$. Then $N$ is called a weakly $\m$-full submodule of $M$ if $(\m N:_M\m)=N$. Also, $N$ is called a Burch submodule of $M$ if $\m (N:_M \m)\ne \m N$. 
\end{subdfn}  

%First we record that vanishing of $\Ext^1_R(k,J)^{\mu}$ have some consequences for the weakly $\m$-full property.  

\begin{subproposition}\label{3.9}
Let $(R,\m,k)$ be a local ring and let, $N$ be an $R$-submodule of a finitely generated $R$-module $M$ such that $\Ext^1_R(k,N)^{\mu}=0$. Then $(\m N:_M\m)=N+\Soc(M)$, i.e., $N+\Soc(M)$ is a weakly $\m$-full submodule of $M$. So, in particular, if we moreover have $\depth(M)>0$, then $N$ is a weakly $\m$-full submodule of $M$.   
\end{subproposition}   

\begin{proof}  
Clearly, $N\subseteq (\m N:_M \m)$ and $\Soc(M)=(0:_M \m)\subseteq (\m N:_M\m)$, so $N+\Soc(M)\subseteq (\m N:_M\m)$. Now choose an element $f\in (\m N:_M\m)$, and we aim to show $f\in N+\Soc(M)$. If $f\in N$, then we are done. Otherwise, say $f\notin N$. As $f\in (\m N:_M\m)$, so $\m f\subseteq\m N\subseteq N$, hence $\m(N+Rf)=\m N$. Thus $\dfrac{Rf+N}{N}$ is a non-zero (as $f\not\in N$) $k$-vector space (as $\m(N+Rf)\subseteq N$). Moreover, $\dfrac{Rf+N}{N}$ is cyclically generated by the image of $f$ in $M/N$. Hence $\dfrac{Rf+N}{N}\cong k$, so we have a short exact sequence $\sigma: 0\to N\to Rf+N \to k\to 0$. This sequence is $\mu$-additive, since  $\mu(N+Rf)=\lambda\left(\frac{N+Rf}{\m(N+Rf)}\right)=\lambda\left(\frac{N+Rf}{\m N}\right)=\lambda\left(\frac{N+Rf}{N}\right)+\lambda\left(\frac{N}{\m N}\right)=\lambda(k)+\mu(N)=\mu(k)+\mu(N)$. So, $\sigma\in\Ext^1_R(k,N)^{\mu}=0$. Hence $\sigma$ splits. Thus $N+Rf=N\oplus g(k)$, where $g:k\to N+Rf$ is the splitting map, so $k\cong g(k)$. Now $f\in N+Rf=N\oplus g(k)$, so $f=x+y$ for some $x\in N$ and $y\in g(k)\subseteq M$. Since $k\cong g(k)$, we have $\m g(k)=0$. Now $\m y\in\m g(k)=0$, so $y\in (0:_M\m)=\Soc(M)$. Hence $f=x+y\in N+\Soc(M)$. This finally shows that $(\m N:_M\m)\subseteq N+\Soc(M)$, so $(\m N:_M\m)=N+\Soc(M)$. Since $\m(N+\Soc(M))=\m N$, we get $(\m (N+\Soc(M)):_M\m)=N+\Soc(M)$, which implies $N+\Soc(M)$ is a weakly $\m$-full submodule $M$. Now note that, if $\depth(M)>0$, then $\Soc(M)=0$. Hence $N+\Soc(M)=N$ is a weakly $\m$-full submodule of $M$. 

%%%Next let, $R$ be Gorenstein. We are already done if $\depth(R)>0$. So, we moreover assume that $\depth(R)=0$. Hence $R$ is an Artinian Gorenstein ring. Now since $J$ is a non-zero ideal of $R$, so $\Soc(R)\subseteq J$, so $J+\Soc(R)=J$. Hence $J$ is a weakly $\m$-full ideal of $R$. 
\end{proof}

Now we give another characterization of regular local rings in terms of vanishing of certain $\Ext^1_R(k,-)^{\mu}$.  

\begin{subthm}\label{reg} Let $(R,\m,k)$ be a local ring of depth $t>0$. Then, the following are equivalent:

\begin{enumerate}[\rm(1)]
    \item $R$ is regular. 
    
    \item $\Ext^1_R(k,R/(x_1,...,x_{t-1})R)^{\mu}=0$ for some $R$-regular sequence $x_1,...,x_{t-1}$.
    
    \item $\Ext^1_R(k,M)^{\mu}=0$ for some finitely generated $R$-module $M$ of projective dimension $t-1$. 
\end{enumerate}

\end{subthm}  

\begin{proof} $(1) \implies (2)$ Since $R$ is regular, we have $t=\dim R$. Since $R$ is regular, we get that $\m$ is generated by a regular sequence $x_1,...,x_t$. Now $R/(x_1,...,x_{t-1})R$ has projective dimension $t-1$ over $R$ and the $(t-1)$-th Betti number of this module is $1$ (by looking at the Koszul complex). Hence, $\Ext^1_R(k,R/(x_1,...,x_{t-1})R)\cong \Tor^R_{t-1}(k,R/(x_1,...,x_{t-1})R)\cong k$ by \cite[Exercise 3.3.26]{bh}. Now, we have an exact sequence $0\to R/(x_1,...,x_{t-1})R\xrightarrow{\cdot x_t} R/(x_1,...,x_{t-1})R\to R/(x_1,...,x_{t-1},x_t)R\cong k\to 0$, which is clearly not $\mu$-additive. Hence, $\Ext^1_R(k,R/(x_1,...,x_{t-1})R)\ne \Ext^1_R(k,R/(x_1,...,x_{t-1})R)^{\mu}$. Since $\Ext^1_R(k,R/(x_1,...,x_{t-1})R)\cong k$ is cyclic, by Lemma \ref{one} we get that $\Ext^1_R(k,R/(x_1,...,x_{t-1})R)^{\mu}=\m \Ext^1_R(k,R/(x_1,...,x_{t-1})R)=0$. 

$(2) \implies (3)$ Obvious.   

$(3) \implies (1)$ By Auslander-Buchsbaum formula, $\depth M=1$, so $\Soc(M)=0$. Hence by prime avoidance, we can choose $x\in \m$, which is both $R$ and $M$-regular. Then $xM\cong M$, so $\Ext^1_R(k, xM)^{\mu}\cong \Ext^1_R(k,M)^{\mu}=0$. By Proposition \ref{3.9} we get that, $xM$ is a weakly $\m$-full submodule of $M$. Since $\depth (M/xM)=0$, $xM$ is a Burch submodule of $M$ by \cite[Lemma 4.3]{dk}. Moreover, $\pd_{R/xR}(M/xM)=\pd_R M<\infty$, and $\pd_R R/xR <\infty$, so $\pd_R M/xM<\infty$. Hence $\Tor^R_{\gg 0}(k,M/xM)=0$. Thus $\pd_R k<\infty$ by \cite[Theorem 1.2]{dk}, so $R$ is regular.  %%%%%In view of Corollary \ref{regu}, we only need to prove that $\Ext^1_R(k,R)^{\mu}=0$ implies $R$ is regular.Pick a non-zero divisor $x\in \m\setminus \m^2$. Then, $(x)\cong R$, and so $\Ext_R(k,(x))^{\mu}\cong \Ext_R(k,R)^{\mu}=0$. By Proposition \ref{3.9} we have $(x\m:_R\m)=(x)$. If $R$ is not regular, then $\m$ is not principal, so by Lemma \ref{3.10} (applied to $I=\m$) we get $(x\m:_R \m)=((x):_R \m)$. Hence we get $(x)=((x):_R \m)$. So, $\text{soc}(R/(x))=\dfrac{((x):_R \m)}{(x)}=0$, contradicting $\depth R/(x)=\depth R-1=0$. Thus, $R$ must be regular.  
\end{proof} 

Next, we try to relate $\Ext^1_R(M,N)^{\mu}$ to $\Ext^1_R(M,N)$ for all $M,N\in \mod R$, when $(R,\m)$ is a regular local ring of dimension $1$ (i.e., a local PID). For this, we first record a preliminary lemma.

\begin{sublemma}\label{cycquot} Let $(R,\m)$ be a local ring and $x\in \m$ be a non-zero-divisor. Then, $\Ext^1_R(R/xR, R/I)^{\mu}=\m\Ext^1_R(R/xR, R/I)\cong \dfrac{\m}{I+xR}$ for every proper ideal $I$ of $R$.   
\end{sublemma}     

\begin{proof}  Calculating $\Ext^1_R(R/xR, R/I)$ from the minimal free-resolution  $0\to R \xrightarrow{\cdot x} R\to R/xR\to 0$ of $R/xR$, we see that $\Ext^1_R(R/xR,R/I)\cong R/(xR+I)$ is a cyclic $R$-module. By Lemma \ref{jane} we have $\m\Ext^1_R(R/xR,R/I)\subseteq \Ext^1_R(R/xR,R/I)^{\mu}$. We finally claim that $\Ext^1_R(R/xR,R/I)^{\mu}\neq \Ext^1_R(R/xR,R/I)$. Indeed, we have an exact sequence $\sigma: 0\to xR/xI \to R/xI\to R/xR\to 0$. Now, we have a natural surjection $R \xrightarrow{r\mapsto rx+xI} xR/xI$, whose kernel is $\{r\in R: rx\in xI\}$. Since $x$ is a non-zero-divisor, $rx\in xI$ if and only if $r\in I$. Hence, the kernel is $I$. Hence $R/I\cong xR/xI$.  So, we get the exact sequence $\sigma: 0\to R/I\to R/xI \to R/xR\to 0$. Moreover, $\sigma$ is not $\mu$-additive, since  $\mu(R/xI)=1\ne 1+1= \mu(R/I)+\mu(R/xR)$. Thus, $[\sigma]\in \Ext^1_R(R/xR, R/I)\setminus \Ext^1_R(R/xR, R/I)^{\mu}$.  Hence $\Ext^1_R(R/xR,R/I)^{\mu}=\m \Ext^1_R(R/xR,R/I)$ by Lemma \ref{one}. 
\end{proof}

Now using this lemma, we can compare the structure of $\Ext^1_R(M,N)^{\mu}$ to $\Ext^1_R(M,N)$ for every pair of finitely generated modules $M,N$ over a DVR.  

\begin{subproposition}\label{dvr} Let $(R,\m)$ be a regular local ring of dimension $1$. Then, $\Ext^1_R(M,N)^{\mu}=\m\Ext^1_R(M,N)$ for all finitely generated $R$-modules $M$ and $N$. So, in particular, $\Ext^1_R(k,N)^{\mu}=0$ for all finitely generated $R$-modules $N$. 
\end{subproposition}  

\begin{proof} Let $\m=xR$. Then, for every finitely generated $R$-module $X$, we have $X\cong R^{\oplus a}\oplus \left(\oplus_{i=1}^n R/x^{a_i}R\right)$ for some non-negative integers (depending on $X$) $a,a_i$. Now, fix arbitrary $N\in \mod R$. Applying Lemma \ref{extended} to the subfunctor $\Ext^1_R(-,N)^{\mu}$ of $\Ext^1_R(-,N)$ and taking $I=\m$, it is enough to prove that $\Ext^1_R(R,N)^{\mu}=\m\Ext^1_R(R,N)$ and $\Ext^1_R(R/x^lR,N)^{\mu}=\m\Ext^1_R(R/x^lR,N)$ for every integer $l\ge 1$. Now $\Ext^1_R(R,N)^{\mu}=\m\Ext^1_R(R,N)$ is obvious, as both sides are zero. Now to prove $\Ext^1_R(R/x^lR,N)^{\mu}=\m\Ext^1_R(R/x^lR,N)$ for every integer $l\ge 1$, first fix an $l\ge 1$, and look at the subfunctor $\Ext^1_R(R/x^lR,-)^{\mu}$ of $\Ext^1_R(R/x^lR,-)$. Again by the structure of finitely generated $R$-modules and Lemma \ref{extended}, it is enough to prove that $\Ext^1_R(R/x^lR,R)^{\mu}=\m \Ext^1_R(R/x^lR,R)$ and $\Ext^1_R(R/x^lR,R/x^bR)^{\mu}=\m \Ext^1_R(R/x^lR,R/x^bR)$ for every integer $b\ge 1$. Now these equalities follow from Lemma \ref{cycquot}, since $x$ is a non-zero-divisor. 
\end{proof}   

%Next, we aim to establish a converse of Proposition \ref{reg}. for which we first record a Lemma. 

When restricting to short exact sequences in $\Ext^1$, on which certain subadditive function, other than $\mu(-)$, is additive, one obtains vanishing of the corresponding submodule of $\Ext^1_R(M,F)$ for any free $R$-module $F$. We make this precise in Proposition \ref{loewy}, whose proof uses the following lemma. 

\begin{sublemma}\label{lowey}
Let $(R,\m)$ be a local Cohen--Macaulay ring of dimension $1$. Let $x\in \m$ be $R$-regular. Let $\sigma:0\to F_1\to F_2\oplus \frac{R}{x^aR}\to \frac{R}{x^bR}\to 0$ be a short exact sequence, where $a,b$ are non-negative integers and $F_1,F_2$ are free $R$-modules. Let $c$ be an integer such that $c\ge b$ and $\sigma\otimes \frac{R}{x^cR}$ is short exact. Then $a=b$ and $\rnk(F_2)= \rnk(F_1)$. So, in particular, $\sigma$ is split exact.   
\end{sublemma}  

\begin{proof} Since $x\in R$ is a non-zero-divisor, so $\frac{R}{x^aR}, \frac{R}{x^bR}$ are torsion modules, i.e., have constant rank $0$. Hence calculating rank along $\sigma$, we get $\rnk(F_2)= \rnk(F_1)$.   Next, we will show that $a\le b$. Dualizing $\sigma$ by $R$, we get the following part of a long exact sequence: $\Ext^1_R\left(\frac{R}{x^bR},R\right)\to\Ext^1_R\left(\frac{R}{x^aR},R\right)\to 0$. Now since $x^a\in R$ is a non-zero-divisor, by taking the resolution $0\to R\xrightarrow{\cdot x^a}R\to \frac{R}{x^aR}\to 0$ of $\frac{R}{x^aR}$, and dualizing by $R$, and calculating the cohomology we get that, $\Ext^1_R\left(\frac{R}{x^aR},R\right)\cong \frac{R}{x^aR}$. Similarly, $\Ext^1_R\left(\frac{R}{x^bR},R\right)\cong \frac{R}{x^bR}$. Hence we have the exact sequence  $\frac{R}{x^bR}\to\frac{R}{x^aR}\to 0$. Since $x^b$ annihilates $\frac{R}{x^bR}$, we have $x^b$ annihilates $\frac{R}{x^aR}$ as well, which implies $x^bR\subseteq x^aR$. Since $x\in\m$ is $R$-regular, so $x^bR\subseteq x^aR$ implies that $a\leq b$. Since $\sigma\otimes \frac{R}{x^cR}$ is short exact, we have the following short exact sequence $$\sigma\otimes \frac{R}{x^cR}:0\to \left(\frac{R}{x^cR}\right)^{\oplus l_1}\to\left(\frac{R}{x^cR}\right)^{\oplus l_2}\oplus\frac{R}{x^aR}\to \frac{R}{x^bR}\to 0$$ where we have used that $\frac{R}{x^cR}\otimes\frac{R}{x^aR}\cong \frac{R}{x^cR+x^aR}\cong \frac{R}{x^aR}$ and $\frac{R}{x^cR}\otimes\frac{R}{x^bR}\cong \frac{R}{x^cR+x^bR}\cong \frac{R}{x^bR}$, as $a\leq b\leq c$. Since $R$ is local Cohen--Macaulay of dimension $1$ and $x\in \m$ is $R$-regular, we have $x^aR, x^b R, x^c R$ are $\m$-primary ideals.  Hence by calculating the length along the short exact sequence $\sigma\otimes \frac{R}{x^cR}$ we get that, $\lambda\left(\frac{R}{x^bR}\right)-\lambda\left(\frac{R}{x^aR}\right)=(l_2-l_1)\lambda\left(\frac{R}{x^cR}\right)=0$. Since $x^bR\subseteq x^aR$, so $\lambda\left(\frac{R}{x^bR}\right)=\lambda\left(\frac{R}{x^aR}\right)$ now implies $x^aR=x^bR$ by calculating length along the short exact sequence $0\to \frac{x^aR}{x^bR}\to \frac{R}{x^bR}\to \frac{R}{x^aR}\to 0$. Since $x\in\m$ is $R$-regular, we have $a=b$. Finally, since $F_1\cong F_2$ and $a=b$, we get $\sigma$ is split exact. 
\end{proof}

In the following, $\H^0_{\m}(-)$ denotes the zero-th local cohomology module. For a finite length $R$-module $M$, $\ell\ell(M)$ will stand for the smallest integer $n\ge 0$ such that $\m^n M=0$. 

\begin{subproposition}\label{loewy}
Let $(R,\m)$ be a regular local ring of dimension $1$. Let $L\in \mod(R)$. Let $c$ be an integer such that $c\ge\ell\ell(\H^0_{\m}(L))$. Consider the function $\phi_L(-):=\lambda(\frac{R}{\m^{c}}\otimes -):\mod(R)\to\mathbb{N}\cup\{0\}$. Then for any free $R$-module $F$ we have, $\Ext_R^1(L,F)^{\phi_L}=0$. 
\end{subproposition}    

\begin{proof} Note that, $R$ is not a field, since $\dim R=1$. Since $R$ is a regular local ring, we have $L\cong G\oplus L'$, where $G$ is a finite free $R$-module and $L'$ is an $R$-module of finite length. We have $\H^0_\m(L)\cong L'$. Since $\Ext^1_R(L,F)^{\phi_L}\cong \Ext^1_R(G,F)^{\phi_L}\oplus \Ext^1_R(L',F)^{\phi_L}\cong \Ext^1_R(L',F)^{\phi_L}$, we may replace $L$ by $L'$, and assume without loss of generality that $L$ has finite length, and $c\ge \ell\ell(L)$. For simplicity, we denote $\phi_L=\phi$. Now it is enough to show that, $\Ext_R^1(L,R)^{\phi}=0$. Since $(R,\m)$ is a regular local ring of dimension $1$, $R$ is a PID. Hence $\m=xR$ for some $x\neq 0$. Since $L$ is a finite length module, by the structure theorem of modules over a PID we have, $L\cong \oplus_{i=1}^{n} \frac{R}{x^{b_i}R}$, where $b_i>0$. So, $\ell\ell(L)=\max_{1\le i\le n}{b_i}$. Hence it is enough to show that, $\Ext_R^1\left(\frac{R}{x^{b_i}R},R\right)^{\phi}=0$ for all $i=1,\cdots,n$. Fix an $i\in\{1,\cdots,n\}$ and consider a short exact sequence $\sigma:0\to R\to X\to\frac{R}{x^{b_i}R}\to 0$ in $\Ext_R^1\left(\frac{R}{x^{b_i}R},R\right)^{\phi}$. By the structure theorem of modules over a PID, we also have $X\cong R^{\oplus s}\oplus \left(\oplus_{j=1}^t\frac{R}{x^{a_j}R}\right)$, where $a_j>0$. As $R$ is an integral domain and $\frac{R}{x^{b_i}R}$ and $\frac{R}{x^{a_j}R}$ are all torsion $R$-modules, so calculating rank along the short exact sequence $\sigma$, we obtain $s=1$. Moreover, $1+t=s+t=\mu(X)\leq 2$, since $\mu $ is subadditive. Thus $t\leq 1$. If $t=0$, then we obtain $\sigma:0\to R\to R \oplus \frac{R}{x^0R}\to\frac{R}{x^{b_i}R}\to 0$. Since $\sigma$ is $\phi$-additive, by using Lemma \ref{31} with the functor $\lambda_R\left(\frac{R}{\m^{c}} \otimes_R -\right)$ we obtain that $\sigma\otimes\frac{R}{\m^{c}}=\sigma\otimes\frac{R}{x^{c}R}$ is short exact. As $c\geq b_i>0$, we obtain a contradiction from Lemma \ref{lowey}.  Thus $t=1$. So, we get $X\cong R\oplus\frac{R}{x^{a}R}$ for some $a=a_j$. Thus we have the short exact sequence $\sigma:0\to R\to R\oplus\frac{R}{x^{a}R}\to\frac{R}{x^{b_i}R}\to 0$. %by calculating the length along the exact sequence $\sigma\otimes\frac{R}{\m^{c}}=\sigma\otimes\frac{R}{x^{c}R}$ we get that, the sequence $\sigma\otimes\frac{R}{x^{c}R}$ is short exact (by Lemma \ref{31}). 

Since $c\ge\ell\ell(L)=\max_{1\le i\le n}{b_i}$, we have $c\ge b_i$ for all $i=1,\cdots,n$. Hence by applying Lemma \ref{lowey} on $\sigma$ we get that, $\sigma$ is split exact. Since $\sigma$ is an arbitrary element of $\Ext_R^1\left(\frac{R}{x^{b_i}R},R\right)^{\phi}$, we get $\Ext_R^1\left(\frac{R}{x^{b_i}R},R\right)^{\phi}=0$ for all $i=1,\cdots,n$. This implies that, $\Ext_R^1\left(L,R\right)^{\phi}=0$, so $\Ext_R^1\left(L,F\right)^{\phi}=0$.
\end{proof}   

Taking $L=k$, $c=1=\ell\ell(k)$, we see that $\phi_L(-)=\mu(-)$ in Proposition \ref{loewy}. So, we also get another proof of Corollary \ref{regu} in dimension $1$. 

Next, we compare $\Ext^1_R(M,R)^{\mu}$ with $\Ext^1_R(M,R)$. In arbitrary dimension, we only consider this for local Cohen--Macaulay rings of minimal multiplicity.  We first record some general preliminary lemmas.    

\begin{sublemma}\label{artincan}
Let $(R,\m,k)$ be a local ring such that $\m^2=0$, $\m\neq 0$. Let $e:=\mu(\m)$. Then  $\mu(\omega_R)=e$, $\mu(\Omega_R\omega_R)=e^2-1$ and $\Omega_R\omega_R\cong k^{\oplus(e^2-1)}$.
\end{sublemma}  

\begin{proof}
Since $\m^2=0$ and $\m\neq 0$, we have $\m\subseteq(0:_R\m)\subsetneq R$. Hence $\m=(0:_R\m)$. Then we have $\mu(\omega_R)=\r(R)=\dim_k(0:_R\m)=\dim_k\m=\dim_k\left(\frac{\m}{\m^2}\right)=\mu(\m)=e$. Also note that, $\lambda(R)=\lambda\left(\frac{R}{\m^2}\right)=\lambda\left(\frac{R}{\m}\right)+\lambda\left(\frac{\m}{\m^2}\right)=1+\mu(\m)$. Next, consider the short exact sequence $0\to\Omega_R\omega_R\to R^{\oplus\mu(\omega_R)}\to\omega_R\to 0$, so $\Omega_R\omega_R\subseteq\m R^{\oplus\mu(\omega_R)}$. Hence $\m\Omega_R\omega_R=0$. This implies that, $\mu(\Omega_R\omega_R)=\lambda\left(\frac{\Omega_R\omega_R}{\m\Omega_R\omega_R}\right)=\lambda(\Omega_R\omega_R)=\lambda(R^{\oplus\mu(\omega_R)})-\lambda(\omega_R)=\mu(\omega_R)\lambda(R)-\lambda(R)=\lambda(R)(\mu(\omega_R)-1)=\lambda(R)(\mu(\m)-1)=(1+\mu(\m))(\mu(\m)-1)=e^2-1$, where $\lambda(\omega_R)=\lambda(R)$ follows from Matlis duality.
\end{proof}

\begin{sublemma}\label{syzreg}
Let $(R,\m,k)$ be a local ring and let, $M$ be a finitely generated $R$-module. If $x$ is $M$-regular, then $\frac{\Omega_R M}{x\Omega_R M}\cong\Omega_{R/xR}\left(\frac{M}{xM}\right)$.
\end{sublemma} 

\begin{proof}
Since $x$ is $M$-regular, we have $\tor_1^R(M,\frac{R}{xR})=0$. Hence tensoring the short exact sequence $0\to\Omega_R M\to R^{\oplus\mu_R(M)}\to M\to 0$ with $R/xR$, we get the exact sequence $0\to\frac{\Omega_R M}{x\Omega_R M}\to (\frac{R}{xR})^{\oplus\mu_R(M)}\to\frac{M}{xM}\to 0$. Since $\mu_R(M)=\mu_{\frac{R}{xR}}(\frac{M}{xM})$, we have $\frac{\Omega_R M}{x\Omega_R M}\cong\Omega_{R/xR}\left(\frac{M}{xM}\right)$.
\end{proof}  

\begin{subchunk}\label{inftype} Since $\mu_R(-)=\lambda_R\left(\dfrac{(-)}{\m(-)}\right)$, it follows by \ref{infres} that $\mu_R((-))=\mu_S(S\otimes_R(-))$, where $S=R[X]_{\m[X]}$. Since tensoring with $S$ preserves exactness, so tensoring a minimal free resolution $(F_{\bullet},\partial_{\bullet})$ of an $R$-module $M$ with $S$ and remembering $S\otimes\partial$ now have entries in $\m S$, the maximal ideal of $S$, we see that $S\otimes_R \syz_R M\cong\syz_S(S\otimes_S M)$. Also, if $\omega_R$ exists, then owing to the fact that $S/\m S$ is a field, we see that $\omega_S$ also exists and $\omega_S\cong S\otimes_R \omega_R$ by \cite[Theorem 3.3.14(a)]{bh}. Hence, $S\otimes_R \Hom_R(-,\omega_R)\cong \Hom_S\left(S\otimes_R(-),\omega_S\right)$.   Finally, we also have $\r(R)=\r(S)$ by \cite[Proposition 1.2.16(b)]{bh}.
\end{subchunk}   

\begin{subproposition}\label{mintype}
Let $(R,\m,k)$ be a local Cohen--Macaulay ring of minimal multiplicity admitting a canonical module $\omega_R$ and also let, $\m\neq 0$. Then $\mu((\Omega_R\omega_R)^{\dagger})=\r(R)^2-1$. 
\end{subproposition}

\begin{proof} Due to \ref{inftype}, we may pass to the faithfully flat extension $S:=R[X]_{\m[X]}$ and assume the residue field is infinite.  

We will prove the claim by induction on $\dim R=d$. First let, $d=0$. Since $R$ has minimal multiplicity, we have $\m^2=0$. Hence by Lemma \ref{artincan} we get that, $\Omega_R\omega_R\cong k^{\oplus(\r(R)^2-1)}$. This implies $(\Omega_R\omega_R)^{\dagger}\cong k^{\oplus(\r(R)^2-1)}$, so $\mu((\Omega_R\omega_R)^{\dagger})=\r(R)^2-1$. Now let, $\dim R=d\ge 1$ and let, the claim be true for rings with dimension $d-1$. Let $x\in\m$ be such that $\m^2=(x,x_1,...,x_{d-1})\m$ (see \cite[4.6.14(c)]{bh}). So, $\frac{R}{xR}$ has minimal multiplicity. Now we have
\begin{align*}
\frac{(\Omega_R\omega_R)^{\dagger}}{x(\Omega_R\omega_R)^{\dagger}} &\cong\left(\frac{\Omega_R\omega_R}{x\Omega_R\omega_R}\right)^{\dagger}\cong \left(\Omega_{\frac{R}{xR}}\left(\frac{\omega_R}{x\omega_R}\right)\right)^{\dagger}\cong \left(\Omega_{\frac{R}{xR}}\omega_{\frac{R}{xR}}\right)^{\dagger}
\end{align*} where the first isomorphism follows from \cite[Proposition 3.3.3(a)]{bh}, and the second isomorphism is by Lemma \ref{syzreg}.  
So, 
\begin{align*}
\mu_R((\Omega_R\omega_R)^{\dagger}) =\mu_{\frac{R}{xR}}\left(\frac{(\Omega_R\omega_R)^{\dagger}}{x(\Omega_R\omega_R)^{\dagger}}\right)&=\mu_{\frac{R}{xR}}\left(\left(\Omega_{\frac{R}{xR}}\omega_{\frac{R}{xR}}\right)^{\dagger}\right)\\
&=r\left(\frac{R}{xR}\right)^2-1\text{ [By induction hypothesis]}\\
&=\r(R)^2-1
\end{align*}
\end{proof}

\begin{subproposition}\label{muadd}
Let $(R,\m,k)$ be a local Cohen--Macaulay ring of minimal multiplicity admitting a canonical module $\omega_R$. Then the exact sequence $0\to R\to\omega_R^{\oplus\mu(\omega_R)}\to (\Omega_R\omega_R)^{\dagger}\to 0$ is $\mu$-additive.
\end{subproposition}

\begin{proof}
Consider the exact sequence $0\to\Omega_R\omega_R\to R^{\oplus\mu(\omega_R)}\to\omega_R\to 0$. Since $\Ext^1_R(\omega_R,\omega_R)=0$, we have the exact sequence $0\to R\to \omega_R^{\oplus\mu(\omega_R)}\to(\Omega_R\omega_R)^{\dagger}\to 0$. From Proposition \ref{mintype} we have, $\mu((\Omega_R\omega_R)^{\dagger})=\r(R)^2-1=\mu(\omega_R)^2-\mu(R)=\mu(\omega_R^{\oplus\mu(\omega_R)})-\mu(R)$. Hence $0\to R\to\omega_R^{\oplus\mu(\omega_R)}\to (\Omega_R\omega_R)^{\dagger}\to 0$ is $\mu$-additive.
\end{proof}

\begin{subproposition}\label{mr}
Let $(R,\m,k)$ be a local Cohen--Macaulay ring of minimal multiplicity. Then $\Ext^1_R(M,F)=\Ext^1_R(M,F)^{\mu}$ for any maximal Cohen--Macaulay $R$-module $M$ and any finitely generated free $R$-module $F$.    
\end{subproposition}  

\begin{proof} By Lemma \ref{extended}, it is enough to prove the claim for $F=R$.  

We first consider the case when $R$ has a canonical module $\omega_R$. Consider the $\mu$-additive (by Proposition \ref{muadd}) exact sequence $\sigma:0\to R\to\omega_R^{\oplus\mu(\omega_R)}\to (\Omega_R\omega_R)^{\dagger}\to 0$. Applying $\Hom_R(M,-)$ to $\sigma$, we get the following part of a commutative diagram of exact sequences by Corollary \ref{long}:   

$$\begin{tikzcd}[sep=1.8em, font=\small]
{\Hom_R(M,(\Omega_R\omega_R)^{\dagger})} \arrow[r]                 & {\Ext_{R}^1(M,R)} \arrow[r]                 & {\Ext_{R}^1(M,\omega_R^{\oplus\mu(\omega_R)})}                 \\
{\Hom_R(M,(\Omega_R\omega_R)^{\dagger})} \arrow[r] \arrow[u, equal] & {\Ext_{R}^1(M,R)^{\mu}} \arrow[r] \arrow[u, hook] & {\Ext_{R}^1(M,\omega_R^{\oplus\mu(\omega_R)})^{\mu}} \arrow[u, hook]
\end{tikzcd}$$  
Since $\omega_R$ has finite injective dimension, we have $\Ext_{R}^1(M,\omega_R^{\oplus\mu(\omega_R)})=0$ by \cite[Exercise 3.1.24]{bh}. So, $\Ext_{R}^1(M,\omega_R^{\oplus\mu(\omega_R)})^{\mu}=0$ as well. Hence we get the following commutative diagram:
$$\begin{tikzcd}[sep=1.8em, font=\small]
{\Hom_R(M,(\Omega_R\omega_R)^{\dagger})} \arrow[r, "f"]                 & {\Ext_{R}^1(M,R)} \arrow[r]                 & 0             \\
{\Hom_R(M,(\Omega_R\omega_R)^{\dagger})} \arrow[r, "g"] \arrow[u, equal] & {\Ext_{R}^1(M,R)^{\mu}} \arrow[r] \arrow[u, hook, "h"] & 0 
\end{tikzcd}$$ 
Thus $h\circ g=f$ is surjective, so $h$ is surjective. Since $h$ is the natural inclusion map, we have $\Ext^1_R(M,R)=\Ext^1_R(M,R)^{\mu}$.

Now we consider the general case. Since $\Ext^1_R(M,R)^{\mu}\subseteq \Ext_{R}^1(M,R)$, it is enough to prove the other inclusion. So let, $\sigma: 0\to R\to X\to M\to 0 $ be an exact sequence. We need to show $\sigma$ is $\mu$-additive. Now consider the completion $\widehat \sigma: 0\to \widehat R\to \widehat X\to \widehat M \to 0$. Since $\widehat R$ is Cohen--Macaulay, having minimal multiplicity, admitting a canonical module, and $\widehat M$ is maximal Cohen--Macaulay over $\widehat R$, by the first part of the proof we get $ \Ext^1_{\widehat R}(\widehat M,\widehat R)=\Ext^1_{\widehat R}(\widehat M,\widehat R)^{\mu}$. Thus $[\widehat \sigma ]\in \Ext^1_{\widehat R}(\widehat M,\widehat R)=\Ext^1_{\widehat R}(\widehat M,\widehat R)^{\mu}$. Hence $\widehat \sigma$ is $\mu$-additive. Since number of generators does not change under completion, we get $\sigma$ is $\mu$-additive, which is what we wanted to prove.  
\end{proof}

\begin{subcor}\label{hyper}
Let $(R,\m,k)$ be a local Cohen--Macaulay ring of dimension $d$ and minimal multiplicity. If $\Ext^1_R(\syz^{i+d}_R k,R)^{\mu}=0$ for some $i\geq 0$, then $R$ is a hypersurface. 
\end{subcor} 

\begin{proof}
By Proposition \ref{mr} we have, $\Ext^{d+i+1}_R(k,R)\cong \Ext^1_R(\syz^{d+i}_R k, R)=\Ext^1_R(\syz^{d+i}_R k,R)^{\mu}=0$. Hence $R$ has finite injective dimension by \cite[II. Theorem 2]{roberts}, so $R$ is Gorenstein. Since $R$ has minimal multiplicity, $R$ is a hypersurface.
\end{proof}  

One can also detect when a local ring has depth $0$ by comparing $\Ext^1_R(M,R)$ to $\Ext^1_R(M,R)^{\mu}$ as shown in the following proposition. In the following, $\Tr(-)$ stands for Auslander transpose (see \cite[Definition 12.3]{lw}). 

\begin{subproposition} Let $I$ be an ideal of a local ring $(R,\m,k)$. Then, $\Ext^1_R(\Tr(R/I),R/\ann_R(I))^{\mu}=\m \Ext^1_R(\Tr(R/I),R/\ann_R(I))$. Moreover, the following are equivalent: 

\begin{enumerate}[\rm(1)]

\item $\depth R=0$.
    \item $\Ext^1_R(M,F)^{\mu}=\Ext^1_R(M,F)$ for all finitely generated $R$-modules $M$ and $F$, where $F$ is free.  
    
    \item $\Ext^1_R(\Tr k,R)^{\mu}= \Ext^1_R(\Tr k,R)$.
\end{enumerate}
\end{subproposition}     

\begin{proof} Since everywhere in this proposition the Auslander transpose is in the first component of $\Ext^1$, our claim does not depend on the choice of $\Tr$. First we prove, $\Ext^1_R(\Tr(R/I),R/\ann_R(I))^{\mu}=\m \Ext^1_R(\Tr(R/I),R/\ann_R(I))$. We may assume $\Tr(R/I)$ is non-free. Dualizing the exact sequence $R^{\oplus \mu(I)} \to R \xrightarrow{\pi} R/I \to 0$ by $R$, we get $0\to \Hom_R(R/I,R)\xrightarrow{\pi^*} R\to R^{\oplus \mu(I)}\to \Tr(R/I)\to 0$. Under the natural identification  $\Hom_R(R/I,R)\cong \ann_R(I)$, the map $\Hom_R(R/I,R)\xrightarrow{\pi^*} R$ can be identified with the inclusion map $\ann_R(I)\to R$, giving us $0\to \ann_R(I) \to R \to R^{\oplus \mu(I)}\to \Tr(R/I)\to 0$. Hence, we get an exact sequence $0\to R/\ann_R(I) \to R^{\oplus \mu(I)}\to \Tr(R/I)\to 0$. Since $\Tr(R/I)$ is non-free, this sequence gives us $\syz_R \Tr(R/I)\cong R/\ann_R(I)$ and the first Betti number of $\Tr(R/I)$ is $1$. Hence, $\Ext^1_R(\Tr(R/I),R/\ann_R(I))^{\mu}=\m \Ext^1_R(\Tr(R/I),R/\ann_R(I))$ by Lemma \ref{lemM}.

Now we prove the equivalence of the three conditions as follows:

$(1) \implies (2)$: Let $\sigma: 0\to F\to X\to M\to 0$ be an exact sequence, where $F$ is a finitely generated free $R$-module. Since $\depth R=0$, we get that $k$ embeds inside $R$, i.e., $k$ is torsionless. Hence $\Ext^1_R(\Tr k, R)=0$ by \cite[Proposition 12.5]{lw}. Hence, $\Ext^1_R(\Tr k, F)=0$. So, we get an exact sequence $0\to \Hom_R(\Tr k,F)\to \Hom_R(\Tr k, X) \to \Hom_R(\Tr k, M)\to 0$. Then by \cite[Exercise 13.36]{lw} we get that, the sequence $0\to F \otimes_R k \to X \otimes_R k \to M \otimes_R k \to 0$ is also exact, which means $\sigma$ is $\mu$-additive. Thus $\Ext^1_R(M,F)^{\mu}=\Ext^1_R(M,F)$.

$(2) \implies (3)$ Obvious.  

$(3) \implies (1)$: Assume $\Ext^1_R(\Tr k,R)^{\mu}=\Ext^1_R(\Tr k,R)$. Now if possible let, $\depth R>0$. Then $\m$ contains a non-zero-divisor, so $\ann_R(\m)=0$. Then by the first part of this proposition, we get  $\Ext^1_R(\Tr k,R)^{\mu}=\m \Ext^1_R(\Tr k,R)$. Hence $\Ext^1_R(\Tr k,R)=\m \Ext^1_R(\Tr k,R)$. Then by Nakayama's lemma, $\Ext^1_R(\Tr k,R)=0$.  Hence by \cite[Proposition 12.5]{lw}, we have an embedding $k \to k^{**}$. Consequently, $k^*\ne 0$, i.e., $\depth R=0$.
\end{proof}   

%Now dualizing the exact sequence $R^{\oplus \mu(\m)} \to R \to k \to 0$ by $R$, we get an exact sequence $0=\Hom_R(k,R)\to R\to R^{\oplus \mu(\m)}\to \Tr k\to 0$. Since $\Ext^1_R(\Tr k,R)=0$, so this sequence splits. Hence $\Tr k$ is projective, so $\Tr k$ is free. Hence $k$ is also free, which means $R$ is a field, contradicting $\depth R>0$. Thus we must have $\depth R=0$.  

For general local Cohen--Macaulay rings of dimension $1$, we now show that if $\Ext^1_R(M,R)^{\mu}=0$ for some $I$-Ulrich module (\cite[Definition 4.1]{dms}) $M\subseteq Q(R)$ containing a non-zero-divisor of $R$, then $I$ is principal. For this, we first need the following lemma. For the remainder of this section, given $R$-submodules $M,N$ of $Q(R)$, by $(M:N)$ we will mean $\{x\in Q(R): xN\subseteq M\}$.

\begin{sublemma}\label{endo} Let $(R,\m)$ be a local Cohen--Macaulay ring of dimension $1$. Let $I$ be an $\m$-primary ideal of $R$ admitting a principal reduction $a\in I$. Assume $(I:I)=R$. If $M\subseteq Q(R)$ is an $I$-Ulrich module, and contains a non-zero-divisor of $R$, then the natural inclusion  map $\Hom_R\left(M,(a)\right)\to \Hom_R(M,I)$, induced by the inclusion $(a) \to I$, is an isomorphism.    
\end{sublemma}     

\begin{proof} The natural inclusion $0\to (a) \xrightarrow{i} I$ induces the following commutative diagram:

$$\begin{tikzcd}
0 \arrow[r] & {\Hom_R(M,(a))} \arrow[r]                                      & {\Hom_R(M,I)}                                            \\
0 \arrow[r] & ((a):M) \arrow[u, "\alpha\to\{x\mapsto \alpha x\}"] \arrow[r] & (I:M) \arrow[u, "\alpha\to\{x\mapsto \alpha x\}"']
\end{tikzcd}$$
where the rows are natural inclusion maps, and the vertical arrows are isomorphisms due to \cite[Proposition 2.4(1)]{trace}. So, it is enough to show that $(I:M)\subseteq ((a):M)$. Indeed, if $x\in (I:M)$, then $xM\subseteq I$. Since $M$ is $I$-Ulrich, $M$ is a $B(I)=R\left[\dfrac Ia\right]$-module (\cite[Remark 4.4, Theorem 4.6]{dms}). Hence $\dfrac I a M\subseteq M$, so $\dfrac I a xM\subseteq xM\subseteq I$. Thus $\dfrac 1 a xM\subseteq (I:I)=R$, so $xM\subseteq (a)$. Therefore $x\in \left((a):M\right)$.  

\end{proof}

\begin{subproposition} Let $(R,\m)$ be a local Cohen--Macaulay ring of dimension $1$. Let $I$ be an $\m$-primary ideal of $R$ admitting a principal reduction $a\in I$. Assume $(I:I)=R$. If there exists an $I$-Ulrich module $M\subseteq Q(R)$, containing a non-zero-divisor of $R$ such that $\Ext^1_R(M,R)^{\mu}=0$, then $I\cong R$.  
\end{subproposition}  

\begin{proof}   Consider the short exact sequence $0\to (a) \xrightarrow{i} I \to I/(a)\to 0$. Since $I^{n+1}=aI^n$ for all $n\gg 0$ and $I$ contains a non-zero-divisor, we have $a$ is a non-zero-divisor. So, $(a)\cong R$. If $a\in \m I$, then $I^{n+1}\subseteq \m II^n=\m I^{n+1}$, which implies $I^{n+1}=0$ by Nakayama's lemma. This contradicts the fact that $I$ is $\m$-primary. So, $a\notin \m I$, hence $\mu(I/(a))=\mu(I)-1$. Thus the above short exact sequence is $\mu$-additive.  So, by Corollary \ref{long} we get the following induced exact sequence:
$$0\to \Hom_R(M,(a))\to \Hom_R(M,I)\to \Hom_R(M,I/(a))\to \Ext^1_R(M,(a))^{\mu}\cong \Ext^1_R(M,R)^{\mu}=0$$ 
where the induced map  $ \Hom_R(M,(a))\to \Hom_R(M,I)$ is an isomorphism by Lemma \ref{endo}. Hence we get $\Hom_R(M,I/(a))=0$. Now $I/(a)$ has finite length. So, if $I/(a)\ne 0$, then $\Ass(I/(a))=\{\m\}$. Hence $\Ass\left(\Hom_R(M,I/(a))\right)=\supp(M)\cap \Ass(I/(a))=\supp(M)\cap \{\m\}=\{\m\}$, contradicting $\Hom_R(M,I/(a))=0$. Thus, we must have $I/(a)=0$, i.e., $I=(a)\cong R$.  
\end{proof}

\begin{subproposition}
    Let $(R,\m,k)$ be a Gorenstein, non-regular, local ring of dimension $d>0$. Let $M$ be an $R$-module of projective dimension $d-1$ such that the $(d-1)$-th Betti number of $M$ is $1$. Then $\Ext^1_R(k,M)^{\mu}\cong k$. 
\end{subproposition} 

\begin{proof} By \cite[Exercise 3.3.26]{bh} we have $\Ext^1_R(k,M)\cong \Tor^R_{d-1}(k,M)$, which in turn is isomorphic to $k$, since the $(d-1)$-th Betti number of $M$ is $1$.  By Lemma \ref{jane} and Lemma \ref{one}, we get either $\Ext^1_R(k,M)^{\mu}=\m\Ext^1_R(k,M)$ or $\Ext^1_R(k,M)^{\mu}=\Ext^1_R(k,M)$. However, if $\Ext^1_R(k,M)^{\mu}=\m\Ext^1_R(k,M)$, then $\Ext^1_R(k,M)^{\mu}=0$, so $R$ would be regular by Theorem \ref{reg}, contradicting our assumption that $R$ is not regular. Thus $\Ext^1_R(k,M)^{\mu}=\Ext^1_R(k,M)\cong k$.  
\end{proof}

\subsection{Some applications of the subfunctor $\Ext^1_{\Ul_I(R)}(-,-)$}   
\text{ }
\\
Let $I$ be an $\m$-primary ideal. In this subsection, we give various applications of the subfunctor $\Ext^1_{\Ul_I^s(R)}(-,-):\Ul_I^s (R)^{op}\times \Ul_I^s(R) \to \mod R$ (see \ref{subst} for notation), where we recall from Corollary \ref{iulrich} that, $\Ul_I^s(R)$ along with its all short exact sequences form an exact subcategory of $\mod R$. Hence $\Ext^1_{\Ul_I^s(R)}(-,-):\Ul_I^s (R)^{op}\times \Ul_I^s(R) \to \mod R$ is a subfunctor of $\Ext^1_R(-,-):\Ul_I^s (R)^{op}\times \Ul_I^s(R) \to \mod R$ by Proposition \ref{exsub}. 

We begin by observing a general connection between $\Ext^1_R(M,N)$ and $\Ext^1_R(M,N)^{\nu_I}$, where $M,N\in \mod R$ and $\nu_I(-):=\lambda((-)\otimes_R R/I):\mod R\to \mathbb Z$ and  (see Definition \ref{numdef} for notation of $\Ext^1_R(-,-)^{\phi}$).     

\begin{sublemma}\label{jane} Let $(R,\m)$ be a Noetherian local ring, and $I$ be an $\m$-primary ideal. Then, $I\Ext^1_R(M,N)\subseteq \Ext^1_R(M,N)^{\nu_I}$ for all $M,N\in \mod R$.  
\end{sublemma}   

\begin{proof}  Let $\sigma:0\to N \to X \to M \to 0$ be a short exact sequence in $I\Ext^1_R(M,N)$. Then $\sigma\otimes_R R/I$ splits by \cite[Theorem 1.1]{jan}, so $M/IM \oplus N/IN \cong X/IX$. Hence taking length we get $\lambda(M/IM)+\lambda(N/IN)=\lambda(X/IX)$, so $\sigma$ is $\nu_I$-additive.  
\end{proof}  

This allows us to prove a general connection between $\Ext^1_{\Ul^s_I(R)}(M,N)$ and $\Ext^1_R(M,N)^{\nu_I}$, when $M,N\in \Ul^s_I(R)$ (recall the definition of $\Ul^s_I(R)$ from Definition \ref{uls}).   

\begin{sublemma}\label{uladd} Let $(R,\m)$ be a local  ring, $s\ge 0$ be an integer and $I$ an $\m$-primary ideal. Then, $\Ext^1_R(M,N)^{\nu_I}=\Ext^1_{\Ul^s_I(R)}(M,N)$ for all $M,N\in \Ul^s_I(R)$. So, in particular, $I\Ext^1_R(M,N)\subseteq \Ext^1_{\Ul_I^s(R)}(M,N)$ for all $M,N\in \Ul^s_I(R)$.
\end{sublemma}  

\begin{proof} Let $M,N\in \Ul^s_I(R)$, and $\sigma: 0\to N \to X \to M \to 0$ be a short exact sequence of $R$-modules. Then $X\in \cm^s(R)$ (see \ref{cmchu}), and $\lambda(M/IM)+\lambda(N/IN)=e_R(I,M)+e_R(I,N)=e_R(I,X)$. Now, $\sigma\in \Ext^1_R(M,N)^{\nu_I}$, if and only if $\lambda(X/IX)=\lambda(M/IM)+\lambda(N/IN)$, if and only if $\lambda(X/IX)=e_R(I,X)$, if and only if $X\in \Ul^s_I(R)$ if and only if $\sigma\in \Ext^1_{\Ul_I^s(R)}(M,N)$. This proves the desired equality $\Ext^1_R(M,N)^{\nu_I}=\Ext^1_{\Ul^s_I(R)}(M,N)$. The last inclusion in the statement now follows from this and Lemma \ref{jane}.  
\end{proof}

When $R$ is Cohen--Macaulay of dimension $1$ and $M,N\in \Ul_I(R)$, then one can improve the inclusion $I\Ext^1_R(M,N)\subseteq \Ext^1_{\Ul_I(R)}(M,N)$ of Lemma \ref{uladd} quite a bit. For this, we first record a general lemma about trace ideals. In the following, we say that $\dfrac{a}{b}\in Q(R)$ is a non-zero-divisor if $a$ is a non-zero-divisor in $R$. Note that, this does not depend on the choice of representative, since $\frac ab=\frac {a'} {b'}$ in $Q(R)$ implies $ab'=a'b$, and since $b,b'$ are non-zero-divisors in $R$, so $a$ is a non-zero-divisor in $R$ if and only if $a'$ is a non-zero-divisor in $R$.  

\begin{sublemma}\label{trgen} Let $I$ be an ideal of $R$ containing a non-zero-divisor.  Then the following holds:

\begin{enumerate}[\rm(1)]
    \item There exist non-zero-divisors $x_1,...,x_n\in R$ such that $I=(x_1,...,x_n)$.
    
    \item There exist non-zero-divisors $y_1,...,y_n \in (R:I)$ such that $\tr_R(I)=\sum_{i=1}^n y_iI$. 
\end{enumerate}
\end{sublemma}  

    \begin{proof} (1) Let $S$ be the collection of all non-zero-divisors of $R$ that are in $I$. Let $\langle S\rangle$ be the ideal of $R$ generated by $S$. Then $I\subseteq (\cup_{\p\in \Ass(R)} \p) \cup \langle S\rangle$. By prime avoidance, either $I \subseteq \p$ for some $\p \in \Ass(R)$, or $I\subseteq \langle S\rangle$. But $I$ contains a non-zero-divisor, so $I\nsubseteq \p$ for every $\p\in \Ass(R)$. Hence  $I\subseteq \langle S\rangle$. Since $S$ is a subset of $I$, we conclude $I= \langle S\rangle$. Since $R$ is Noetherian, there exist finitely many elements $x_1,...,x_n\in S$ such that $I=(x_1,...,x_n)$. 

(2) We know that $\tr_R(I)=(R:I)I$ in $Q(R)$ (see \cite[Proposition 2.4(2)]{trace}). Pick a non-zero-divisor $a\in I$, so $J=a(R:I)\subseteq R$ is an ideal of $R$. Then $\tr_R(I)=J(\frac 1 aI)$. Now $J$ contains a non-zero-divisor, so by part (1) we have $J=(x_1,...,x_n)$ for some non-zero-divisors $x_1,...,x_n$. Then $\tr_R(I)=\sum_{i=1}^n \frac 1 a x_iI$. Denoting $y_i:=\frac 1 a x_i$, we see that each $y_i\in Q(R)$ is a non-zero-divisor, and $y_i\in (R:I)$ as $x_i\in a(R:I)$. Hence the claim.
\end{proof}

In the proof of the next result, for an $\m$-primary ideal $I$, $B(I)$ denotes blow-up of $I$ in the sense of \cite[Definition 4.3, Remark 4.4]{dms}, namely $B(I) :=\cup_{n>0}(I^n:_{Q(R)} I^n)$. 

\begin{subproposition}\label{trset}  Let $(R,\m)$ be a local Cohen--Macaulay ring of dimension $1$, and $I$ be an $\m$-primary ideal of $R$. Then $\tr_R(I)\Ext^1_R(M,N)\subseteq \Ext^1_{\Ul_I(R)}(M,N)$ for all $M,N\in \Ul_I(R)$. 
\end{subproposition}

\begin{proof}  Let $M,N\in\Ul_I(R)$ and $a\in (R:I)$ be a non-zero-divisor. Then $aI$ is an $\m$-primary ideal of $R$ and $B(I)=B(aI)$. So, $\Ul_I(R)=\Ul_{aI}(R)$ by \cite[Proposition 4.24]{dms}. Hence for every non-zero-divisor $a\in (R:I)$, we have  $aI\Ext^1_R(M,N)\subseteq \Ext^1_{\Ul_I(R)}(M,N)$ by Lemma \ref{uladd}. Now by Lemma \ref{trgen}, there exist non-zero-divisors $y_1,...,y_n\in (R:I)$ such that $\tr_R(I)=\sum_{i=1}^n y_iI$. Thus $\tr_R(I)\Ext^1_R(M,N)= \sum_{i=1}^n y_iI\Ext^1_R(M,N)\subseteq \Ext^1_{\Ul_I(R)}(M,N)$.  
\end{proof}   

%Now, for the rest of the section, $R$ will denote a local Cohen--Macaulay ring.   

When $s=1$, $\depth R>0$ and $I=\m$ (so $\nu_I(-)=\mu(-)$), the inclusion $I\Ext^1_R(M,N)\subseteq \Ext^1_{\Ul^s_I(R)}(M,N)$ of Lemma \ref{uladd} is actually an equality as we prove next. For this, we first record an easy lemma about flat extensions.

\begin{sublemma}\label{tenfaith} Let $R\to S$ be a flat extension of  rings. Let $M$ be an $R$-module and $I$ an ideal of $R$. Then, the following holds:  

\begin{enumerate}[\rm(1)]
    \item $S \otimes_R (IM)=(IS)(S \otimes_R M)$ when identified as submodules of $S\otimes_R M$.  
    
    \item If $N\subseteq M$ is an $R$-submodule, $S$ is a faithfully flat extension of $R$ and $S \otimes_R M=S \otimes_R N$, then $M=N$.    
\end{enumerate}

\end{sublemma}  

\begin{proof} (1) Consider the exact sequence $0\to IM \to M \to M/IM\to 0$, which after tensoring with $S$ gives $0\to S \otimes_R (IM) \to S\otimes_R M \to S\otimes_R M/IM \to 0$. Now $S \otimes_R M/IM \cong S \otimes_R (M\otimes_R R/I)\cong (S\otimes_R M)\otimes_R R/I\cong (S\otimes_R M)/I(S \otimes_R M)$. Now by the natural $S$-module structure on $S\otimes_R M$, we see that $I(S \otimes_R M)=(IS)(S \otimes_R M)$. Thus we get the exact sequence $0\to S \otimes_R (IM) \to S\otimes_R M \to (S \otimes_R M)/(IS)(S \otimes_R M)\to 0$, and by naturality of the isomorphisms, we see that the map $S\otimes_R M \to (S \otimes_R M)/(IS)(S \otimes_R M)$ in the exact sequence has kernel $(IS)(S \otimes_R M)$. Hence $S \otimes_R (IM)=(IS)(S \otimes_R M).$  

(2) Tensoring the exact sequence $0\to N \to M \to M/N \to 0$ with $S$ and using $S \otimes_R M=S \otimes_R N$, we get $(M/N)\otimes_R S=0$. Since $S$ is faithfully flat, we have $M/N=0$. Hence $M=N$. 
\end{proof}    

Now we prove the desired equality between $\m\Ext^1_R(M,N)$ and $\Ext^1_{\Ul^1(R)}(M,N)$, when $R$ has positive depth. 

\begin{subproposition}\label{1} Let $(R,\m,k)$ be a local ring of positive depth. Let $M,N\in \Ul^1(R)$. Then, the following holds:  

\begin{enumerate}[\rm(1)]
    \item We always have $\m\Ext^1_R(M,N)=\Ext^1_{\Ul^1(R)}(M,N)$.
    
    \item If $R$ is moreover Cohen--Macaulay of dimension $1$ (so $\Ul^1(R)=\Ul(R)$) and if $x\in \m$ is a minimal reduction of $\m$, then $x\Ext^1_R(M,N)=\Ext^1_{\Ul(R)}(M,N)$.
\end{enumerate} 
\end{subproposition}  

\begin{proof} (1) Due to Lemma \ref{uladd}, we only need to prove $\Ext^1_{\Ul^1(R)}(M,N) \subseteq \m\Ext^1_R(M,N)$. We may assume $M,N\ne 0$. 

First, we assume that the residue field is infinite.
Let $\sigma: 0\to N\to X\to M\to 0$ be a short exact sequence such that $X\in \Ul^1(R)$. Choose $x\in \m$ to be $R\oplus M \oplus N\oplus X$-superficial (which exists by \cite[Proposition 8.5.7]{HS}, since we are assuming $R$ has infinite residue field). Now $\depth_R(R\oplus M\oplus N\oplus X)=\inf\{\depth R,\depth_R M,\depth_R N,\depth_R X\}>0$, so $x$ is $R\oplus M\oplus N\oplus X$-regular by \cite[Lemma 8.5.4]{HS}. Hence $x\in \m$ is regular on $R,M$ and $N$, and superficial on $M,N$ and $X$. Thus $M/xM,N/xN,X/xX$ are $0$-dimensional Ulrich modules by \cite[Proposition 2.2(4)]{agl}, so these are $k$-vector spaces by \cite[Proposition 2.2(1)]{agl}. Hence $\sigma \otimes_R R/xR$, being a short exact sequence of $k$-vector spaces, is split exact. Since $x$ is $R,M,N$-regular, by \cite[Proposition 2.8]{jan} we have $\sigma \in x\Ext^1_R(M,N)\subseteq \m\Ext^1_R(M,N)$.  

Now we prove the general case. This part will use extensively that, $\Ext^1_{\Ul^1(R)}(M,N)$ is a submodule of $\Ext^1_R(M,N)$, where $M,N\in \Ul^1(R)$.
Consider the faithfully flat extension of $R$ as follows: $S :=R[X]_{\m[X]}$, with maximal ideal $\m S$, and infinite residue field. Since we already know, $\m\Ext^1_R(M,N)\subseteq \Ext^1_{\Ul^1(R)}(M,N)$ by Lemma \ref{uladd} and since $S$ is faithfully flat, it is enough to prove $S\otimes_R \m\Ext^1_R(M,N)=S \otimes_R \Ext^1_{\Ul^1(R)}(M,N)$ (by Lemma \ref{tenfaith}(2)). Now $\m\Ext^1_R(M,N)\subseteq \Ext^1_{\Ul^1(R)}(M,N)$, with flatness of $S$, already implies $S \otimes_R \m\Ext^1_R(M,N)\subseteq S \otimes_R \Ext^1_{\Ul^1(R)}(M,N)$. So, to prove equality, it is enough to prove that $S \otimes_R \Ext^1_{\Ul^1(R)}(M,N) \subseteq S\otimes_R \m\Ext^1_R(M,N)$. Now due to Lemma \ref{tenfaith}(1), it is enough to prove $S \otimes_R \Ext^1_{\Ul^1(R)}(M,N) \subseteq (\m S)(S \otimes_R \Ext^1_R(M,N))$, and the latter object here is naturally identified with $(\m S)\Ext^1_S(S \otimes_R M, S \otimes_R N)$.  Since the $S$-module $S \otimes_R \Ext^1_{\Ul^1(R)}(M,N)$ is generated by $1\otimes_R \sigma $ as $\sigma$ runs over all elements of $\Ext^1_{\Ul^1(R)}(M,N)$, we need to prove that $1 \otimes_R \sigma \in (\m S)\Ext^1_S(S \otimes_R M, S \otimes_R N)$ for all $\sigma \in \Ext^1_{\Ul^1(R)}(M,N)$. So, let $\sigma: 0\to N\to X\to M\to 0$ be a short exact sequence such that $X\in \Ul^1(R)$. Then, $1\otimes_R \sigma \in S \otimes_R \Ext^1_{\Ul^1(R)}(M,N)\subseteq S \otimes_R \Ext^1_R(M,N)\cong \Ext^1_S(S \otimes_R M, S \otimes_R N)$ is naturally identified with the exact sequence $S \otimes_R \sigma: 0\to S \otimes_R N  \to S \otimes_R X \to S \otimes_R M \to 0$ in $\Ext^1_S(S \otimes_R M, S \otimes_R N)$. Now $S \otimes_R \sigma$ is a short exact sequence of modules in $\Ul^1(S)$ by \cite[Proposition 2.2(3)]{agl}. Hence, $1\otimes_R \sigma \in \Ext^1_{\Ul^1(S)}(S \otimes_R M, S \otimes_R N)$. Since $S$ also has positive depth and infinite residue field, by the proof of the infinite residue field case we get $1\otimes_R \sigma \in \Ext^1_{\Ul^1(S)}(S \otimes_R M, S \otimes_R N) \subseteq (\m S)\Ext^1_S(S \otimes_R M, S \otimes_R N)$, which is what we wanted to prove.  

(2)  Due to Lemma \ref{uladd}, we only need to prove $\Ext^1_{\Ul(R)}(M,N) \subseteq x\Ext^1_R(M,N)$. So, let $\sigma: 0\to N\to X\to M\to 0$ be a short exact sequence such that $X\in\Ul(R)$. Since $xR$ is a reduction of $\m$, we have $xR$ is $\m$-primary. Hence $x$ is $R$-regular, so $x$ is $M,N,X$-regular. Also, $\m M=xM,\m N=xN,\m X=xX$. So, $\sigma \otimes_R R/xR: 0\to N/xN \to X/xX \to M/xM\to 0$ is an exact sequence of $k$-vector spaces. Hence, $\sigma \otimes_R R/xR$ is split exact. Since $x$ is $R,M,N,X$-regular, by \cite[Proposition 2.8]{jan} we have $\sigma \in x\Ext^1_R(M,N)$.  Thus we get $\Ext^1_{\Ul(R)}(M,N) \subseteq x\Ext^1_R(M,N)$.
\end{proof}  

Following is an interesting consequence of Proposition \ref{1}. Here, for $R$-modules $X,Y$, we denote by $X*Y$, the collection of all $R$-modules $Z$ that fit into a short exact sequence $0\to X \to Z \to Y \to 0 $.  

\begin{subcor}\label{ulfaith}  
Let $(R,\m)$ be a local $1$-dimensional Cohen--Macaulay ring. Then, the following are equivalent:  
\begin{enumerate}[\rm(1)]

\item $R$ is regular.

\item $\Ul(R)$ is closed under taking extensions.

\item There exists $M,N\in \Ul(R)$ such that $N$ is faithful, $M\ne 0$ and $N*M\subseteq \Ul(R)$. 
\end{enumerate}  
\end{subcor}     

\begin{proof} $(1)\implies (2)$: If $R$ is regular, then $\Ul(R)=\cm(R)$, which is closed under taking extensions.  

$(2)\implies (3)$: One can take $M=N=\m^n$ for some $n\gg 0$. 

$(3) \implies (1)$: By hypothesis, we get $\Ext^1_{R}(M,N)=\Ext^1_{\Ul(R)}(M,N)$. Then from Proposition \ref{1} (remembering $\Ul^1(R)=\Ul(R)$ in our case) we get, $\Ext^1_{R}(M,N)=\m\Ext^1_{R}(M,N)$. So by Nakayama's lemma, we get $\Ext^1_{R}(M,N)=0$. By \cite[Theorem 4.6]{dms} (also, see \cite[5.2]{dk}) we get, $M\cong \m M$ and $N\cong \m N$. So, $\Ext^1_R(\m M,\m N)=0$. Hence $\pd_R (\m M)<\infty$ by \cite[Proposition 2.6]{dk}. Thus $\Tor^R_{\gg 0}(\m, \m M)=0$. Since $0\ne M$, we have $\m M\cong M\ne 0$. Thus $\pd_R(\m)<\infty$ by \cite[Corollary 3.17]{dk}. Hence $R$ is regular.  
\end{proof}      
\if0
Since $\dim R=1$, by our notation (Definition \ref{uls}), $\Ul^1(R)=\Ul(R)$. First, we will show that $\Ext_R^1(M,N)=\m \Ext_R^1(M,N)$ for all $M,N\in\Ul(R)$. Let $\alpha\in\Ext_R^1(M,N)$. Since $M,N\in\Ul(R)$ and $\Ul(R)$ is closed under taking extensions, we get that the middle term of $\alpha$ is Ulrich, i.e., $\alpha \in \Ext^1_{\Ul(R)}(M,N)$. Hence by Proposition \ref{1} we get, $\alpha\in\m\Ext_R^1(M,N)$. Thus $\Ext_R^1(M,N)=\m \Ext_R^1(M,N)$ for all $M,N\in\Ul(R)$, so by Nakayama's lemma we get $\Ext_R^1(M,N)=0$ for all $M,N\in\Ul(R)$. Next, note that $\m^n\in\Ul(R)$ for large enough $n$. So, by taking $M=N=\m^n$ for large enough $n$, it follows from \cite[Corollary 1.4]{power} that $R$ is regular.
\fi

We mention in passing that when   $R$ is local Cohen--Macaulay of minimal multiplicity, $(2)\implies (1)$ of Corollary \ref{ulfaith}  also follows from \cite[Proposition 3.2(5)]{restf}. 
  
%\begin{subremark}  When $R$ is local Cohen--Macaulay of minimal multiplicity, $(2)\implies (1)$ of Corollary \ref{ulfaith}  also follows from \cite[Proposition 3.2(5)]{restf}.
%\end{subremark}  

Before proceeding further with applications to $\Ul_I(R)$, we outline an alternative proof that $\Ul_I(R)$ is an exact subcategory of $\mod R$ when $R$ is a local $1$-dimensional Cohen--Macaulay ring, via birational extensions. For this, we first record some preliminary results.   

The following is easy to see from the definition of an exact category.  

\begin{sublemma}\label{ringex}
%Let $R,S$ be two commutative rings such that $R\to S$ be a ring extension. Let ${\S}_S$ and ${\S}_R$ be the standard exact structures on $\Mod(S)$ and $\Mod(R)$ respectively. Let ${\X}\subseteq\Mod(S)$ be a strictly full subcategory such that $({\X},{\S}_S|_{\X})$ is an exact subcategory of $(\Mod(S),{\S}_S)$. Assume that, $\Hom_S(M,N)=\Hom_R(M,N)$ for all $M,N\in{\X}$. and also, $\X$ is closed under isomorphism   

Let $R,S$ be two commutative rings such that $R\to S$ is a ring extension. Let $({\X},{\E})$ be a strictly full exact subcategory of $\Mod S$. Assume that, ${\X}$ is closed under $R$-linear isomorphism of $R$-modules. Also, assume that every $R$-linear map between any two modules in ${\X}$ is $S$-linear, i.e., $\Hom_R(M,N)=\Hom_S(M,N)$ for all $M,N\in {\X}$. Then, $({\X},{\E})$ is a strictly full exact subcategory of $\Mod R$.  
\end{sublemma}   

%\begin{proof}  Since ${\X}$ is closed under $R$-linear isomorphism of $R$-modules and $\Hom_R(M,N)=\Hom_S(M,N)$ for all $M,N\in {\X}$ and ${\X}$ is a strictly full subcategory of $\Mod S$, so ${\X}$ is a strictly full subcategory of $\Mod R$. Since every element in $\E$ is a short exact sequence of $S$-modules and $S$-linear maps, so every element in $\E$ is a short exact sequence of $R$-modules and $R$-linear maps. Hence every element in $\E$ is a kernel-cokernel pair in $\Mod(R)$.
%Hence every element in $\E$ is a kernel-cokernel pair in ${\X}$ considered as a subcategory of $\Mod(R)$, since ${\X}$ is a strictly full subcategory of $\Mod R$.

%Now by \cite[Definition 2.1]{Theo} and \cite[Remark 2.4]{Theo}, it is enough to show that $\E$ is closed under isomorphisms and $({\X},{\E})$ satisfies the axioms [E0], [E0$^{\text{op}}$], [E1$^{\text{op}}$], [E2] and [E2$^{\text{op}}$]. First we will show that, $\E$ is closed under isomorphisms. Let $M\to N\to L$ be an kernel-cokernel pair in ${\E}^{\phi}$. Also let, $M'\to N'\to L'$ be an kernel-cokernel pair in ${\A}$ such that it is isomorphic to $M\to N\to L$, which implies $M'\to N'\to L'$ is in $\E$ and $M\cong M'$, $N\cong N'$, $L\cong L'$. Since
%\end{proof}  

For the consequences of Lemma \ref{ringex}, we need the following lemma which is essentially \cite[Proposition 4.14(i)]{lw}. Before proceeding, we note that for a birational extension $R\subseteq S\subseteq Q(R)$ and an $S$-module $N$, being torsion-free as an $S$-module and as an $R$-module are the same. 

\begin{sublemma}\label{lw4.14} Let $R$ be a commutative Noetherian ring and let $R\subseteq S \subseteq Q(R)$ be a ring extension. Then, $\Hom_R(M,N)=\Hom_S(M,N)$ for all $M,N\in \Mod (S)$ where $N$ is torsion-free $S$-module.  
\end{sublemma} 

\begin{proof} $\Hom_S(M,N)\subseteq \Hom_R(M,N)$ is clear. Hence it is enough to show $\Hom_R(M,N)\subseteq \Hom_S(M,N)$. Let $\dfrac a b\in S \subseteq Q(R)$, where $a,b\in R$, so $b\in R$ is a non-zero-divisor. Let $m\in M$ and $f\in \Hom_R(M,N)$. Then $b\left(f\left(\dfrac a b m\right)-\dfrac a b f(m)\right)=bf\left(\dfrac a b m\right)-af(m)=f\left(b\dfrac a b m\right)-af(m)=f(am)-af(m)=0$. Since $N$ is torsion-free and $b$ is a non-zerodivisor on $R$, it is a non-zerodivisor on $N$. So, $f\left(\dfrac a b m\right)=\dfrac a b f(m)$. As $\dfrac a b\in S, m\in M$ and $f\in \Hom_R(M,N)$ were arbitrary, we conclude $\Hom_R(M,N)\subseteq \Hom_S(M,N)$. 
\end{proof}  

Now as a consequence, we can deduce the following:

\begin{subproposition}\label{tf}  Let $R\subseteq S$ be a birational extension ($S\subseteq Q(R)$) of commutative rings. Let ${\S}_R$ be the standard exact structure on $\Mod R$. Let ${\X}$ be the strictly full subcategory of $\Mod S$ consisting of all torsion-free $S$-modules. Then it holds that ${\S}_R|_{\X}={\S}_S|_{\X}$, and $({\X}, {\S}_R|_{\X})$ is a strictly full exact subcategory of $\Mod R$, and $({\X}\cap \mod R, {\S}_R|_{{\X}\cap\mod R})$ is an exact subcategory of $\Mod R$ (hence, also of $\mod R$). 
\end{subproposition}  

\begin{proof}  By definition, ${\S}_R|_{\mod R\cap {\X}}={\S}_R|_{{\mod R}}\cap {\S}_R|_{\X}$. Since $(\mod R,{\S}_R|_{{\mod R}})$ is an exact subcategory of $\Mod R$, the second part of the statement would readily follow from the first part of the statement and Lemma \ref{inter}. First, we show that $\X$ is strictly full in $\Mod R$. If $M,N$ are two $R$-modules and $f:M\to N$ is an $R$-linear isomorphism and $N$ is moreover a torsion-free $S$-module extending its $R$-module structure, then $M$ has an $S$-module structure, extending its $R$-module structure, given by $s\cdot m:=f^{-1}(sf(m))$ for all $s\in S, m \in M$. Note that, with this structure, $M$ is moreover a torsion-free $S$-module. Indeed, let $s\in S$ be a non-zero-divisor such that $s\cdot m=0$. Then $sf(m)=f(s\cdot m)=0$, so $f(m)=0$ as $N$ is $S$-torsionfree. Thus $m=0$ as $f$ is an isomorphism. Moreover, $\Hom_R(M,N)=\Hom_S(M,N)$ for all torsion-free $S$-modules $M,N$ by Lemma \ref{lw4.14}. Consequently, we notice that ${\S}_R|_{{\X}}={\S}_S|_{{\X}}$. So, to show $({\X}, {\S}_R|_{\X})$ is an exact subcategory of $\Mod R$, it is enough to show that $({\X}, {\S}_S|_{\X})$ is an exact subcategory of $\Mod S$ (by Lemma \ref{ringex}). Now it is well known that for any ring $S$, the subcategory of all $S$-torsion-free modules is closed under extensions in $\Mod S$. Hence $({\X}, {\S}_S|_{\X})$ is an exact subcategory of $\Mod S$ by Proposition \ref{extclosed}.  
\end{proof}       

\begin{subcor} Let $R$ be a local Cohen--Macaulay ring of dimension $1$. Let $R\subseteq S$ be a finite birational extension. Then, $(\cm(S),{\S}_R|_{{\cm(S)}})$ is a strictly full exact subcategory of $\mod R$.  
\end{subcor}   

\begin{proof} We note that, $\mod S=\Mod S\cap \mod R$. Since $S$ is also $1$-dimensional Cohen--Macaulay, we have $\cm(S)=$ collection of all finitely generated torsion-free $S$-modules $ = $ collection of all torsion-free $S$-modules $\cap \mod R$. Hence the conclusion follows from Proposition \ref{tf}.  
\end{proof}   

When $(R,\m)$ is a local Cohen--Macaulay ring of dimension $1$ and $I$ an $\m$-primary ideal, then considering the finite birational extension $R\subseteq B(I)$, where $B(I)$ denotes blow-up of $I$ (\cite[Definition 4.3, Remark 4.4]{dms}), it is shown in \cite[Theorem 4.6]{dms} that $\cm(B(I))=\Ul_I(R)$. Thus, in this particular case, we get another proof of the fact that $\Ul_I(R)$ is an exact subcategory of $\mod R$, which is very different from Corollary \ref{iulrich}.  
\text{ }

For further applications, we first record a lemma connecting $\Ext^1_{\Ul_I(R)}(-,-)$ to $\Ext^1_{B(I)}(-,-)$. Here, we will keep in mind that, when $\dim R=1$, we write $\Ul^1_I(R)=\Ul_I(R)$ in our notation (see Definition \ref{uls}).    

\begin{sublemma}\label{uliso} Let $(R,\m)$ be a local Cohen--Macaulay ring of dimension $1$. Let $I$ be an $\m$-primary ideal of $R$. Then, for all $I$-Ulrich modules $M,N$, we have a natural map $\Ext^1_{\Ul_I(R)}(M,N)\xrightarrow{[\sigma]_R \to [\sigma]_{B(I)}} \Ext^1_{B(I)}(M,N)$, which is an isomorphism of $R$-modules.  
\end{sublemma} 

\begin{proof}  We recall that, $\Ext^1_{\Ul_I(R)}(M,N)$ is the equivalence class of all short exact sequences $0\to N \to X \to M\to 0$ of $I$-Ulrich modules and $R$-linear maps. Since $\Ul_I(R)=\cm(B(I))$ (\cite[Theorem 4.6]{dms}), any $R$-linear map between two Ulrich modules are $B(I)$-linear (\cite[Proposition 4.14(i)]{lw}). Hence any short exact sequence $0\to N \to X \to M\to 0$ of $I$-Ulrich modules and $R$-linear maps is a short exact sequence of maximal Cohen--Macaulay $B(I)$-modules and $B(I)$-linear maps, and any $R$-linear morphism between two short exact sequences of $I$-Ulrich $R$-modules is actually a $B(I)$-morphism. This proves the well-definedness of the map. Injectivity is similarly obvious. To prove surjectivity, we note that if $Y$ is a $B(I)$-module, and $0\to N \to Y \to M \to 0$ is a short exact sequence in $\Mod (B(I))$, then $M,N\in \cm(B(I))$ implies $Y\in \cm(B(I))$. Hence $Y\in \Ul_I(R)$. Since $B(I)$-linear maps are $R$-linear, this gives a pre-image in $\Ext^1_{\Ul_I(R)}(M,N)$. Finally, to show the map is $R$-linear, we recall that the Baer-sum structure and multiplication by $R$ on $\Ext^1$ are given by certain pullback and pushout diagrams. So, it is enough to prove that, given an exact sequence $\sigma: 0\to N \to X \to M \to 0$ in $\Ul_I(R)$ and $R$-linear maps $f: N\to N', g:M'\to M$ with $M',N'\in \Ul_I(R)$, the pushout and pullback of $\sigma$ by $f$ and $g$ in $\mod R$ respectively are actually pushout and pullback in $\mod (B(I))$. Let us look at the pushout case, as the pullback case is similar. If we have the following pushout diagram: 

$$\begin{tikzcd}
\sigma: 0 \arrow[r] & N \arrow[r] \arrow[d, "f"'] & X \arrow[r] \arrow[d] & M \arrow[d, equal] \arrow[r] & 0 \\
0 \arrow[r]         & N' \arrow[r]                & Y \arrow[r]           & M \arrow[r]                       & 0
\end{tikzcd}$$
then in the bottom row we again have $Y\in \Ul_I(R)$, since $\Ul_I(R)$ along with all its short exact sequences is an exact subcategory of $\mod R$ by Corollary \ref{iulrich}. Hence in the above diagram, all the modules are in $\Ul_I(R)=\cm(B(I))$, so all the $R$-linear maps are moreover $B(I)$-linear. Since $\cm(B(I))$ is an extension closed subcategory of $\mod(B(I))$, it is an exact subcategory of $\mod (B(I))$ by Proposition \ref{extclosed}. Thus this is a pushout diagram in $\mod(B(I))$ as well by \cite[Proposition 2.12]{Theo}. This is what we wanted to show. 
\end{proof}

\begin{subcor}\label{trann} Let $(R,\m)$ be a local Cohen--Macaulay ring of dimension $1$, and $I$ be an $\m$-primary ideal of $R$. If $M,N\in \Ul_I(R)$ are such that $\Ext^1_{B(I)}(M,N)=0$, then $\tr_R(I)\Ext^1_R(M,N)=0$.
\end{subcor}

\begin{proof} This follows from Lemma \ref{uliso} and Proposition \ref{trset}. 
\end{proof}

In the following, for a module $X$ over a ring $R$, by $\add_R(X)$ we denote the collection of all $R$-modules $Y$ such that there exists an $R$-module $Z$ and an isomorphism of $R$-modules $Y\oplus Z\cong X^{\oplus n}$ for some integer $n\ge 0$.   

\begin{subthm}\label{projgor} Let $(R,\m)$ be a local Cohen--Macaulay ring of dimension $1$. Let $I$ be an $\m$-primary ideal of $R$. Then the following hold: 

\begin{enumerate}[\rm(1)]
    \item If $M\in \add_R(B(I))$, then $\tr_R(I)\Ext^1_R(M, \Ul_I(R))=0$. 
    
    \item  Let $M\in \Ul(R)$. Then $M\in \add_R(B(\m))$ if and only if $\m \Ext^1_R(M, \Ul(R))=0$. 
    
    \item If $B(I)$ is a Gorenstein ring, then $\tr_R(I)\Ext^1_R(\Ul_I(R),B(I))=0$. The converse holds when $I=\m$.   
\end{enumerate}
\end{subthm} 

\begin{proof} (1) Denote $S:=B(I)$. If $M\in \add_R(S)(\subseteq \Ul_I(R))$, then there exists $X\in \mod(R)$ and an isomorphism of $R$-modules $M\oplus X\cong B(I)^{\oplus n}$ for some $n\ge 0$. Then $M,X$ are $I$-Ulrich, so $M,X\in \cm(S)$. Hence the isomorphism is also $S$-linear by \cite[Proposition 4.14(i)]{lw}, so $M$ is a projective $S$-module. Thus $\Ext^1_S(M,N)=0$ for all $N\in \cm(S)=\Ul_I(R)$. Now the conclusion follows from Corollary \ref{trann}.  

(2) If $M\in \add_R(B(\m))$, then the conclusion follows by (1) as $\m\subseteq \tr_R(\m)$. So, assume $M\in \Ul(R)$ and $\m \Ext^1_R(M,N)=0$ for all $N\in \Ul(R)$. By Proposition \ref{1}, we then have $\Ext^1_{\Ul(R)}(M,N)=0$ for all $N\in \Ul(R)$. Now denote $S=B(\m)$. Then by Lemma \ref{uliso} we have, $\Ext^1_{S}(M,N)=0$ for all $N\in \Ul(R)=\cm(S)$. We know that, for any Cohen--Macaulay ring $S$, and $M\in \cm(S)$, we have $\Ext^1_S(M,\cm(S))=0$ if and only if $M\in \add_S(S)$. Indeed, one direction is clear since $M\in \add_S(S)$ implies $M$ is projective. For the converse, we notice that as $M$ is finitely generated over $S$, so we have a short exact sequence $\sigma: 0\to M'\to S^{\oplus n}\to M \to 0$ for some $n>0$ and $S$-module $M'$. Moreover, $S$ being Cohen--Macaulay implies $M'\in \cm(S)$, and now the vanishing of $\Ext^1_S(M,M')$ gives that the sequence $\sigma$ splits. Hence $M\in \add_S(S)$. Now applying this observation to our scenario with $S=B(\m)$, we get $M\in \add_S(B(\m))=\add_R(B(\m))$, where this last equality follows from the fact that direct summand of $B(\m)$ is in $\cm(B(\m))$ and consequently, $R$-linear maps are $S$-linear (\cite[Proposition 4.14(i)]{lw}).  

(3) Nothing to prove if $R$ is regular, so we assume $R$ is singular. If $B(I)$ is a Gorenstein ring, then $\Ext^1_{B(I)}(M,B(I))=0$ for all $M \in \cm(B(I))=\Ul_I(R)$. Since $B(I)$ is $I$-Ulrich, by Lemma \ref{uliso} we get, $\Ext^1_{\Ul_I(R)}(M,B(I))=0$ for all $M \in \Ul_I(R)$. Then by Corollary \ref{trann} we get, $\tr_R(I)\Ext^1_R(M,B(I))=0$ for all $M \in \Ul_I(R)$. For the converse part, assume $I=\m$ and $\m\Ext^1_R(M,B(\m))=0$ for all $M \in \Ul(R)$. Then by Proposition \ref{1} we have, $\Ext^1_{\Ul(R)}(M,B(\m))=0$ for all $M\in \Ul(R)$. Hence by Lemma \ref{uliso} we have, $\Ext^1_{B(\m)}(M,B(\m))=0$ for all $M\in \Ul(R)=\cm(B(\m))$. Now we note that, for a finite-dimensional Cohen--Macaulay ring $S$,  $\Ext^1_S(M,S)=0$ for all $M \in \cm(S)$ if and only if $S$ is Gorenstein. Indeed, if $S$ is Gorenstein, then the vanishing is clear by \cite[Definition 3.1.18, Proposition 3.1.9, 3.1.24]{bh}. Conversely, let $d=\dim S$. Then for every $\p \in \spec(S)$, we have $\syz^d_S(S/\p) \in \cm(S)$ (by localizing and applying depth lemma). So, now the vanishing of $\Ext^{d+1}_S(S/\p, S)\cong \Ext^1_S(\syz^d_S(S/\p),S)$ for every $\p\in \spec(S)$ implies $\injdim_S S <\infty$. Thus $S$ is Gorenstein by \cite[Definition 3.1.18, Proposition 3.1.9]{bh}. Now applying this observation to our scenario with $S=B(\m)$, we get that $B(\m)$ is Gorenstein.  
\end{proof}  

Since for $1$-dimensional local Cohen--Macaulay rings $(R,\m)$ of minimal multiplicity it holds that $B(\m)=(\m:\m)\cong \m$, the following corollary is an immediate consequence of Theorem \ref{projgor}(3) and \cite[Theorem 5.1]{gmp}.

\begin{subcor}\label{algor}  Let $(R,\m)$ be a local Cohen--Macaulay ring of dimension $1$ and minimal multiplicity. Then, $R$ is almost Gorenstein if and only if $\m\Ext^1_R(\Ul(R),\m)=0$.  
\end{subcor}      

Finally, we give one more application of Theorem \ref{projgor}(3), for which we first record an easy observation: 

\begin{sublemma}\label{redul} Let $(R,\m)$ be a local Cohen--Macaulay ring of dimension $1$. Let $I$ be an $\m$-primary ideal of $R$ admitting a principal reduction $x\in I$. Then, $\m$ is $I$-Ulrich if and only if $\m\subseteq ((x):I)$.  
\end{sublemma} 

\begin{proof} If $I$ is principal, then all MCM modules are $I$-Ulrich. So, it is enough to assume $I$ is not principal. If $\m$ is $I$-Ulrich, then $I\m =x\m$ (\cite[Proposition 4.5]{dms}). So, $\m \subseteq ((x):I)$. Conversely, let  $\m\subseteq ((x):I)$. Then, $\m\subseteq ((x):_RI)$. We claim that, $x\notin \m I$. Since $x\in I$ is a principal reduction of $I$, we have $I^{n+1}=xI^n$ for some $n\ge 1$. Hence if $x\in \m I$, then $I^{n+1}\subseteq \m I^{n+1}$. So, $I^{n+1}=\m I^{n+1}$. Hence by Nakayama's lemma, $I^{n+1}=0$, contradicting $I$ is $\m$-primary. Thus $x\in I\setminus \m I$. Since $\m\subseteq ((x):_R I)$ and $I$ is not principal, we have $\m \subseteq (x\m :_R I)$ by Lemma \ref{3.10}. So, $\m I \subseteq x \m$. Hence $\m I=x\m$, so $\m$ is $I$-Ulrich. 
\end{proof}     

\begin{subcor}\label{5.2.16} Let $(R,\m)$ be a local Cohen--Macaulay ring of dimension $1$, with infinite residue field, and minimal multiplicity, admitting a canonical module.  If $\m\Ext^1_R(\Ul(R),\m)=0$, then $R$ admits a canonical ideal $\omega$ and $B(\omega)$ is Gorenstein. Consequently, $\m\Ext^1_{R}(\Ul_{\omega}(R),B(\omega))=0$. 
\end{subcor}     

\begin{proof} Let $\omega$ be a canonical module of $R$. We show that $\omega$ can be identified with an ideal of $R$, $\m\subseteq \tr_R(\omega)$, and $B(\omega)$ is Gorenstein. Since $R$ has minimal multiplicity and $\m\Ext^1_R(\Ul(R),\m)=0$, $R$ is almost Gorenstein by Corollary \ref{algor}. Hence $\m\subseteq \tr_R(\omega)$ from \cite[Definition 2.2, Proposition 6.1]{trcan}, so $R$ is generically Gorenstein (\cite[Lemma 2.1]{trcan}). Hence $\omega$ can be identified with an ideal of $R$ by \cite[Proposition 3.3.18]{bh}. Since $R$ is almost Gorenstein, we have $\m \subseteq ((x):\omega)$ for some principal reduction $x$ of $\omega$ (see \cite[Setting 3.4, Theorem 3.11]{gmp}). Thus $\m$ is $\omega$-Ulrich by Lemma \ref{redul}. So, $\m \in \cm(B(\omega))$ by \cite[Theorem 4.6]{dms}. Hence $\m\subseteq \tr_R(\m) \subseteq (R:_{Q(R)}B(\omega))$ by \cite[Theorem 2.9]{dms}. Thus the conductor $c_R(B(\omega)):=(R:_{Q(R)}B(\omega))$ of $B(\omega)$ is either $\m$ or $R$. As $R$ has minimal multiplicity, by using \cite[4.6.14(c)]{bh} we see that $c_R(B(\omega))$ satisfies the condition (2) of \cite[Corollary 3.8]{gmp}. Thus $B(\omega)$ is Gorenstein by \cite[Corollary 3.8]{gmp}. Now the claim follows from  Theorem \ref{projgor}(3). 
\end{proof}

\paragraph{Acknowledgment}
We are grateful for partial support from the Simons Collaborator Grant FND0077558 and the University of Kansas Mathematics Department. We thank David Eisenbud for helpful conversations and for pointing out reference \cite{ES}. We also thank the anonymous referee for numerous helpful suggestions that improved the exposition of the paper. 
%%%%%%\paragraph{Data Availability Statement} Nulla non mauris vitae wisi posuere convallis. Sed eu nulla nec eros scelerisque pharetra.

\def\refname{References}

\end{document}